\documentclass[10pt]{amsart}
\usepackage{amsmath}
\usepackage{amsxtra}
\usepackage{amscd}
\usepackage{amsthm}
\usepackage{amsfonts}
\usepackage{amssymb}
\usepackage{eucal}
\usepackage[all]{xy}
\usepackage{graphicx}

\newtheorem{cor}[subsubsection]{Corollary}
\newtheorem{lem}[subsubsection]{Lemma}
\newtheorem{prop}[subsubsection]{Proposition}

\newtheorem{conj}[subsubsection]{Conjecture}
\newtheorem{thm}[subsubsection]{Theorem}
\newtheorem{qthm}[subsubsection]{Quasi-Theorem}

\newtheorem{corconj}[subsubsection]{Corollary-of-Conjecture}
\newtheorem{defn}[subsubsection]{Definition}

\theoremstyle{definition}

\theoremstyle{remark}
\newtheorem{rem}[subsubsection]{Remark}

\newcommand{\thmref}[1]{Theorem~\ref{#1}}

\newcommand{\secref}[1]{Sect.~\ref{#1}}
\newcommand{\lemref}[1]{Lemma~\ref{#1}}
\newcommand{\propref}[1]{Proposition~\ref{#1}}
\newcommand{\corref}[1]{Corollary~\ref{#1}}
\newcommand{\conjref}[1]{Conjecture~\ref{#1}}

\numberwithin{equation}{section}

\newcommand{\nc}{\newcommand}
\nc{\renc}{\renewcommand}
\nc{\ssec}{\subsection}
\nc{\sssec}{\subsubsection}
\nc{\on}{\operatorname}

\nc\ol{\overline}
\nc\wt{\widetilde}
\nc\tboxtimes{\wt{\boxtimes}}
\nc\tstar{\wt{\star}}
\nc{\alp}{\alpha}

\nc{\ZZ}{{\mathbb Z}}
\nc{\NN}{{\mathbb N}}
\nc{\OO}{{\mathbb O}}
\renc{\SS}{{\mathbb S}}
\nc{\DD}{{\mathbb D}}
\nc{\GG}{{\mathbb G}}

\nc{\Fq}{{\mathbb F}_q}
\nc{\Fqb}{\ol{{\mathbb F}_q}}
\nc{\Ql}{\ol{{\mathbb Q}_\ell}}
\nc{\id}{\text{id}}
\nc\X{\mathcal X}

\nc{\Hom}{\on{Hom}}
\nc{\Lie}{\on{Lie}}
\nc{\Loc}{\on{Loc}_{\cG,\on{spec}}}
\nc{\Pic}{\on{Pic}}
\nc{\Bun}{\on{Bun}}
\nc{\IC}{\on{IC}}
\nc{\Aut}{\on{Aut}}
\nc{\rk}{\on{rk}}
\nc{\Sh}{\on{Sh}}
\nc{\Perv}{\on{Perv}}
\nc{\pos}{{\on{pos}}}
\nc{\Conv}{\on{Conv}}
\nc{\Sph}{\on{Sph}}
\nc{\Sym}{\on{Sym}}
\nc{\BunBb}{\overline{\Bun}_B}
\nc{\BunNb}{\overline{\Bun}_N}
\nc{\BunTb}{\overline{\Bun}_T}
\nc{\BunBbm}{\overline{\Bun}_{B^-}}
\nc{\BunBbel}{\overline{\Bun}_{B,el}}
\nc{\BunBbmel}{\overline{\Bun}_{B^-,el}}
\nc{\Buno}{\overset{o}{\Bun}}
\nc{\BunPb}{{\overline{\Bun}_P}}
\nc{\BunBM}{\Bun_{B(M)}}
\nc{\BunBMb}{\overline{\Bun}_{B(M)}}
\nc{\BunPbw}{{\widetilde{\Bun}_P}}
\nc{\BunBP}{\widetilde{\Bun}_{B,P}}
\nc{\GUb}{\overline{G/U}}
\nc{\GUPb}{\overline{G/U(P)}}

\nc{\Hhom}{\underline{\on{Hom}}}
\nc\syminfty{\on{Sym}^{\infty}}
\nc\lal{\ol{\lambda}}
\nc\xl{\ol{x}}
\nc\thl{\ol{\theta}}
\nc\nul{\ol{\nu}}
\nc\mul{\ol{\mu}}
\nc\Sum\Sigma
\nc{\oX}{\overset{o}{X}{}}
\nc{\hl}{\overset{\leftarrow}h{}}
\nc{\hr}{\overset{\rightarrow}h{}}
\nc{\M}{{\mathcal M}}
\nc{\N}{{\mathcal N}}
\nc{\F}{{\mathcal F}}
\nc{\D}{{\mathcal D}}
\nc{\Q}{{\mathcal Q}}
\nc{\Y}{{\mathcal Y}}
\nc{\G}{{\mathcal G}}
\nc{\E}{{\mathcal E}}
\nc{\CalC}{{\mathcal C}}
\nc\Dh{\widehat{\D}}

\nc{\C}{{\mathcal C}}
\nc{\K}{{\mathcal K}}
\renewcommand{\H}{{\mathcal H}}

\nc{\T}{{\mathcal T}}
\nc{\V}{{\mathcal V}}
\renc{\P}{{\mathcal P}}
\nc{\A}{{\mathcal A}}
\nc{\B}{{\mathcal B}}
\nc{\U}{{\mathcal U}}

\nc{\Gr}{{\on{Gr}}}

\nc{\frn}{{\check{\mathfrak u}(P)}}

\nc{\fC}{\mathfrak C}
\nc{\p}{\mathfrak p}
\nc{\q}{\mathfrak q}
\nc\f{{\mathfrak f}}

\nc{\qo}{{\mathfrak q}}
\nc{\po}{{\mathfrak p}}
\nc{\s}{{\mathfrak s}}
\nc\w{\text{w}}

\renewcommand{\mod}{{\on{-mod}}}

\nc\mathi\iota
\nc\Spec{\on{Spec}}
\nc\Mod{\on{Mod}}
\nc{\tw}{\widetilde{\mathfrak t}}
\nc{\pw}{\widetilde{\mathfrak p}}
\nc{\qw}{\widetilde{\mathfrak q}}
\nc{\jw}{\widetilde j}

\nc{\grb}{\overline{\Gr}}
\nc{\I}{\mathcal I}

\nc{\lambdach}{{\check\lambda}}
\nc{\Lambdach}{{\check\Lambda}{}}
\nc{\much}{{\check\mu}}
\nc{\omegach}{{\check\omega}}
\nc{\nuch}{{\check\nu}}
\nc{\etach}{{\check\eta}}
\nc{\alphach}{{\check\alpha}}
\nc{\oblvtach}{{\check\oblvta}}
\nc{\rhoch}{{\check\rho}}
\nc{\ch}{{\check h}}

\nc{\Hb}{\overline{\H}}

\nc{\BA}{{\mathbb{A}}}
\nc{\BC}{{\mathbb{C}}}
\nc{\BG}{{\mathbb{G}}}
\nc{\BM}{{\mathbb{M}}}
\nc{\BO}{{\mathbb{O}}}
\nc{\BD}{{\mathbb{D}}}
\nc{\BL}{{\mathbb{L}}}
\nc{\Bl}{{\mathbb{l}}}
\nc{\BN}{{\mathbb{N}}}
\nc{\BP}{{\mathbb{P}}}
\nc{\BQ}{{\mathbb{Q}}}
\nc{\BR}{{\mathbb{R}}}
\nc{\BZ}{{\mathbb{Z}}}
\nc{\BS}{{\mathbb{S}}}

\nc{\CA}{{\mathcal{A}}}
\nc{\CB}{{\mathcal{B}}}

\nc{\CE}{{\mathcal{E}}}
\nc{\CF}{{\mathcal{F}}}
\nc{\CH}{{\mathcal{H}}}

\nc{\CL}{{\mathcal{L}}}
\nc{\CC}{{\mathcal{C}}}
\nc{\CM}{{\mathcal{M}}}
\nc{\CN}{{\mathcal{N}}}
\nc{\CK}{{\mathcal{K}}}
\nc{\CO}{{\mathcal{O}}}
\nc{\CP}{{\mathcal{P}}}
\nc{\CQ}{{\mathcal{Q}}}
\nc{\CR}{{\mathcal{R}}}
\nc{\CS}{{\mathcal{S}}}
\nc{\CT}{{\mathcal{T}}}
\nc{\CU}{{\mathcal{U}}}
\nc{\CV}{{\mathcal{V}}}
\nc{\CW}{{\mathcal{W}}}
\nc{\CX}{{\mathcal{X}}}
\nc{\CY}{{\mathcal{Y}}}
\nc{\CZ}{{\mathcal{Z}}}
\nc{\CI}{{\mathcal{I}}}
\nc{\cD}{{\mathcal{D}}}

\nc{\csM}{{\check{\mathcal A}}{}}
\nc{\oM}{{\overset{\circ}{\mathcal M}}{}}
\nc{\obM}{{\overset{\circ}{\mathbf M}}{}}
\nc{\oCA}{{\overset{\circ}{\mathcal A}}{}}
\nc{\obA}{{\overset{\circ}{\mathbf A}}{}}
\nc{\ooM}{{\overset{\circ}{M}}{}}
\nc{\osM}{{\overset{\circ}{\mathsf M}}{}}
\nc{\vM}{{\overset{\bullet}{\mathcal M}}{}}
\nc{\nM}{{\underset{\bullet}{\mathcal M}}{}}
\nc{\oD}{{\overset{\circ}{\mathcal D}}{}}
\nc{\obD}{{\overset{\circ}{\mathbf D}}{}}
\nc{\oA}{{\overset{\circ}{\mathbb A}}{}}
\nc{\op}{{\overset{\bullet}{\mathbf p}}{}}
\nc{\cp}{{\overset{\circ}{\mathbf p}}{}}
\nc{\oU}{{\overset{\bullet}{\mathcal U}}{}}
\nc{\oZ}{{\overset{\circ}{\mathcal Z}}{}}
\nc{\ofZ}{{\overset{\circ}{\mathfrak Z}}{}}
\nc{\oF}{{\overset{\circ}{\fF}}}

\nc{\fa}{{\mathfrak{a}}}
\nc{\fb}{{\mathfrak{b}}}
\nc{\fc}{{\mathfrak{c}}}
\nc{\fch}{{\mathfrak{ch}}}
\nc{\fd}{{\mathfrak{d}}}
\nc{\ff}{{\mathfrak{f}}}
\nc{\fg}{{\mathfrak{g}}}
\nc{\fgl}{{\mathfrak{gl}}}
\nc{\fh}{{\mathfrak{h}}}
\nc{\fj}{{\mathfrak{j}}}
\nc{\fl}{{\mathfrak{l}}}
\nc{\fm}{{\mathfrak{m}}}
\nc{\fn}{{\mathfrak{n}}}
\nc{\fu}{{\mathfrak{u}}}
\nc{\fp}{{\mathfrak{p}}}
\nc{\fr}{{\mathfrak{r}}}
\nc{\fs}{{\mathfrak{s}}}
\nc{\ft}{{\mathfrak{t}}}
\nc{\fz}{{\mathfrak{z}}}
\nc{\fsl}{{\mathfrak{sl}}}
\nc{\hsl}{{\widehat{\mathfrak{sl}}}}
\nc{\hgl}{{\widehat{\mathfrak{gl}}}}
\nc{\hg}{{\widehat{\mathfrak{g}}}}
\nc{\chg}{{\widehat{\mathfrak{g}}}{}^\vee}
\nc{\hn}{{\widehat{\mathfrak{n}}}}
\nc{\chn}{{\widehat{\mathfrak{n}}}{}^\vee}

\nc{\fA}{{\mathfrak{A}}}
\nc{\fB}{{\mathfrak{B}}}
\nc{\fD}{{\mathfrak{D}}}
\nc{\fE}{{\mathfrak{E}}}
\nc{\fF}{{\mathfrak{F}}}
\nc{\fG}{{\mathfrak{G}}}
\nc{\fK}{{\mathfrak{K}}}
\nc{\fL}{{\mathfrak{L}}}
\nc{\fM}{{\mathfrak{M}}}
\nc{\fN}{{\mathfrak{N}}}
\nc{\fP}{{\mathfrak{P}}}
\nc{\fU}{{\mathfrak{U}}}
\nc{\fV}{{\mathfrak{V}}}
\nc{\fZ}{{\mathfrak{Z}}}

\nc{\bb}{{\mathbf{b}}}
\nc{\bc}{{\mathbf{c}}}
\nc{\bd}{{\mathbf{d}}}
\nc{\bbf}{{\mathbf{f}}}
\nc{\be}{{\mathbf{e}}}
\nc{\bg}{{\mathbf{g}}}
\nc{\bi}{{\mathbf{i}}}
\nc{\bj}{{\mathbf{j}}}
\nc{\bn}{{\mathbf{n}}}
\nc{\bp}{{\mathbf{p}}}
\nc{\bq}{{\mathbf{q}}}
\nc{\bu}{{\mathbf{u}}}
\nc{\bv}{{\mathbf{v}}}
\nc{\bx}{{\mathbf{x}}}
\nc{\bs}{{\mathbf{s}}}
\nc{\by}{{\mathbf{y}}}
\nc{\bw}{{\mathbf{w}}}
\nc{\bA}{{\mathbf{A}}}
\nc{\bK}{{\mathbf{K}}}
\nc{\bB}{{\mathbf{B}}}
\nc{\bC}{{\mathbf{C}}}
\nc{\bG}{{\mathbf{G}}}
\nc{\bD}{{\mathbf{D}}}
\nc{\bH}{{\mathbf{He}}}
\nc{\bM}{{\mathbf{M}}}
\nc{\bN}{{\mathbf{N}}}
\nc{\bO}{{\mathbf{O}}}
\nc{\bV}{{\mathbf{V}}}
\nc{\bW}{{\mathbf{Wh}}}
\nc{\bX}{{\mathbf{X}}}
\nc{\bZ}{{\mathbf{Z}}}
\nc{\bS}{{\mathbf{S}}}
\nc{\bT}{{\mathbf{T}}}

\nc{\sA}{{\mathsf{A}}}
\nc{\sB}{{\mathsf{B}}}
\nc{\sC}{{\mathsf{C}}}
\nc{\sD}{{\mathsf{D}}}
\nc{\sF}{{\mathsf{F}}}
\nc{\sG}{{\mathsf{G}}}
\nc{\sK}{{\mathsf{K}}}
\nc{\sM}{{\mathsf{M}}}
\nc{\sO}{{\mathsf{O}}}
\nc{\sU}{{\mathsf{U}}}
\nc{\sW}{{\mathsf{W}}}
\nc{\sQ}{{\mathsf{Q}}}
\nc{\sP}{{\mathsf{P}}}
\nc{\sZ}{{\mathsf{Z}}}
\nc{\sfp}{{\mathsf{p}}}
\nc{\sfq}{{\mathsf{q}}}
\nc{\sr}{{\mathsf{r}}}
\nc{\sk}{{\mathsf{k}}}
\nc{\su}{{\mathsf{u}}}
\nc{\sv}{{\mathsf{v}}}
\nc{\sg}{{\mathsf{g}}}
\nc{\sff}{{\mathsf{f}}}
\nc{\sfb}{{\mathsf{b}}}
\nc{\sfc}{{\mathsf{c}}}
\nc{\sd}{{\mathsf{d}}}

\nc{\BK}{{\bar{K}}}

\nc{\tA}{{\widetilde{\mathbf{A}}}}
\nc{\tB}{{\widetilde{\mathcal{B}}}}
\nc{\tg}{{\widetilde{\mathfrak{g}}}}
\nc{\tG}{{\widetilde{G}}}
\nc{\TM}{{\widetilde{\mathbb{M}}}{}}
\nc{\tO}{{\widetilde{\mathsf{O}}}{}}
\nc{\tU}{{\widetilde{\mathfrak{U}}}{}}
\nc{\TZ}{{\tilde{Z}}}
\nc{\tx}{{\tilde{x}}}
\nc{\tbv}{{\tilde{\bv}}}
\nc{\tfP}{{\widetilde{\mathfrak{P}}}{}}
\nc{\tz}{{\tilde{\zeta}}}
\nc{\tmu}{{\tilde{\mu}}}

\nc{\urho}{\underline{\rho}}
\nc{\uB}{\underline{B}}
\nc{\uC}{{\underline{\mathbb{C}}}}
\nc{\ui}{\underline{i}}
\nc{\uj}{\underline{j}}
\nc{\ofP}{{\overline{\mathfrak{P}}}}
\nc{\oB}{{\overline{\mathcal{B}}}}
\nc{\og}{{\overline{\mathfrak{g}}}}
\nc{\oI}{{\overline{I}}}

\nc{\eps}{\varepsilon}
\nc{\hrho}{{\hat{\rho}}}

\nc{\one}{{\mathbf{1}}}
\nc{\two}{{\mathbf{t}}}

\nc{\Rep}{{\mathop{\operatorname{\rm Rep}}}}
\nc{\Tot}{{\mathop{\operatorname{\rm Tot}}}}
\nc{\Ker}{{\mathop{\operatorname{\rm Ker}}}}
\nc{\Hilb}{{\mathop{\operatorname{\rm Hilb}}}}
\nc{\End}{{\mathop{\operatorname{\rm End}}}}
\nc{\Ext}{{\mathop{\operatorname{\rm Ext}}}}
\nc{\CHom}{{\mathop{\operatorname{{\mathcal{H}}\it om}}}}
\nc{\GL}{{\mathop{\operatorname{\rm GL}}}}
\nc{\gr}{{\mathop{\operatorname{\rm gr}}}}
\nc{\Id}{{\mathop{\operatorname{\rm Id}}}}
\nc{\de}{{\mathop{\operatorname{\rm def}}}}
\nc{\length}{{\mathop{\operatorname{\rm length}}}}
\nc{\supp}{{\mathop{\operatorname{\rm supp}}}}

\nc{\Cliff}{{\mathsf{Cliff}}}
\nc{\Fl}{\on{Fl}}
\nc{\Fib}{{\mathsf{Fib}}}
\nc{\Coh}{{\on{Coh}}}
\nc{\QCoh}{{\on{QCoh}}}
\nc{\IndCoh}{{\on{IndCoh}}}
\nc{\FCoh}{{\mathsf{FCoh}}}

\nc{\reg}{{\text{\rm reg}}}

\nc{\cplus}{{\mathbf{C}_+}}
\nc{\cminus}{{\mathbf{C}_-}}
\nc{\cthree}{{\mathbf{C}_*}}
\nc{\Qbar}{{\bar{Q}}}
\nc\Eis{{\on{Eis}}}
\nc\Eisb{\ol\Eis{}}
\nc\Eisr{\on{Eis}^{rat}{}}
\nc\wh{\widehat}
\nc{\Def}{\on{Def_{\check{\fb}}(E)}}
\nc{\barZ}{\overline{Z}{}}
\nc{\barbarZ}{\overline{\barZ}{}}
\nc{\barpi}{\overline\pi}
\nc{\barbarpi}{\overline\barpi}
\nc{\barpip}{\overline\pi{}^+}
\nc{\barpim}{\overline\pi{}^-}

\nc{\fq}{\mathfrak q}

\nc{\fqb}{\ol{\fq}{}}
\nc{\fpb}{\ol{\fp}{}}
\nc{\fpr}{{\fp^{rat}}{}}
\nc{\fqr}{{\fq^{rat}}{}}

\nc{\hattimes}{\wh\otimes}

\nc{\bh}{{{\mathbf h}}}
\nc{\bk}{{{\mathbf k}}}
\nc{\bOmega}{{\overline{\Omega(\check \fn)}}}

\nc{\seq}[1]{\stackrel{#1}{\sim}}

\nc{\cT}{{\check{T}}}
\nc{\cG}{{\check{G}}}
\nc{\cM}{{\check{M}}}
\nc{\cB}{{\check{B}}}
\nc{\cP}{{\check{P}}}

\nc{\ct}{{\check{\mathfrak t}}}
\nc{\cg}{{\check{\fg}}}
\nc{\cb}{{\check{\fb}}}
\nc{\cn}{{\check{\fn}}}

\nc{\cLambda}{{\check\Lambda}}

\nc{\cla}{{\check\lambda}}
\nc{\cmu}{{\check\mu}}
\nc{\cnu}{{\check\nu}}
\nc{\ceta}{{\check\eta}}

\nc{\DefbE}{{\on{Def}_{\cB}(E_\cT)}}

\nc{\imathb}{{\ol{\imath}}}
\nc{\rlr}{\overset{\longrightarrow}{\underset{\longrightarrow}\longleftarrow}}

\nc{\oBun}{\overset{\circ}\Bun}
\nc{\LocSys}{\on{LocSys}}
\nc{\BunBbb}{\ol{\ol{Bun}}_B}
\nc{\BunBr}{\Bun_B^{rat}}
\nc{\BunBrp}{\Bun_B^{rat,polar}}
\nc{\BunTrp}{\Bun_T^{rat,polar}}
\nc{\BunNr}{\Bun_N^{rat}}
\nc{\BunNre}{\Bun_N^{enh,rat}}
\nc{\BunTr}{\Bun_T^{rat}}
\nc{\Vect}{\on{Vect}}
\nc{\Whit}{\on{Whit}}
\nc{\CTb}{\ol{\on{CT}}}
\nc{\Ran}{\on{Ran}}
\nc{\CTr}{\on{CT}^{rat}{}}
\nc\jmathr{\jmath^{rat}{}}
\nc{\ux}{\underline{x}}
\nc{\clambda}{{\check\lambda}}
\nc{\calpha}{{\check\alpha}}
\nc{\ind}{{\mathbf{ind}}}
\nc{\oblv}{{\mathbf{oblv}}}
\nc{\coeff}{\on{W-coeff}}
\nc{\Poinc}{\on{Poinc}}
\nc{\Dmod}{\on{D-mod}}
\nc{\dr}{\on{dR}}
\nc{\oCZ}{\overset{\circ}\CZ}
\nc{\KL}{\on{KL}}
\nc{\triv}{{\mathbf{triv}}}

\nc{\dgSch}{\on{DGSch}}
\nc{\Sch}{\on{Sch}}
\nc{\affdgSch}{\on{DGSch}^{\on{aff}}}
\nc{\affSch}{\on{Sch}^{\on{aff}}}
\nc{\Sing}{\on{Sing}}
\nc{\inftygroup}{\infty\on{-Grpd}}
\renc{\dr}{{\on{dr}}}
\nc\Maps{\on{Maps}}
\nc\bMaps{\mathbf{Maps}}
\nc{\ul}{\underline}
\nc{\bNP}{\mathbf{N(P)}}
\nc{\ofc}{\overset{\circ}\fch}
\nc{\ppart}{(\!(t)\!)}
\nc{\qqart}{[\![t]\!]}
\nc{\crit}{\on{crit}}
\nc{\bDelta}{\mathbf{\Delta}}
\nc{\genB}{{\overset{\on{gen}}\to B}}
\nc{\genP}{{\underset{\on{gen}}\longrightarrow P}}
\nc{\genN}{{\underset{\on{gen}}\longrightarrow N}}

\begin{document}

\begin{quote}
{\small ``In jedem Minus steckt ein Plus. Vielleicht habe ich so etwas gesagt, aber man braucht das doch nicht allzu w\"ortlich zu nehmen."}
\hskip1.3cm {\tiny R.~Musil. Der Mann ohne Eigenschaften.}
\end{quote}

\vskip1cm

\title[Geometric Langlands conjecture for $GL_2$]{Outline of the proof of the \\ geometric Langlands conjecture for $GL_2$}

\author{Dennis Gaitsgory} 

\dedicatory{To G\'erard Laumon}

\date{\today}

\maketitle

\section*{Introduction}

\ssec{The goal of this paper}

The goal of this paper is to describe work-in-progress by D.~Arinkin, V.~Drinfeld and the author \footnote{The responsibility
for any deficiency or undesired outcome of this paper lies with the author of this paper.} towards the proof 
of the (categorical) geometric Langlands conjecture. 

\medskip

The contents of the paper can be summarized as follows: we reduce the geometric Langlands conjecture
to a combination of two sets of statements. 

\medskip

The first set is what we call ``quasi-theorems." These are plausible 
(and tractable) statements that involve Langlands duality, but either for proper Levi subgroups, or of local nature,
or both. Hopefully, these quasi-theorems will soon turn into actual theorems.     

\medskip

The second set are two conjectures (namely, Conjectures \ref{c:Whit ext ff} and \ref{c:Op gen}),
both of which are theorems for $GL_n$.  However, these conjectures \emph{do not}
involve Langlands duality: \conjref{c:Whit ext ff} only involves the geometric side of the correspondence, and
and \conjref{c:Op gen} only the spectral side. 

\ssec{Strategy of the proof}  \label{ss:strategy}

In this subsection we will outline the general scheme of the argument. We will be working over
an algebraically closed field $k$ of characteristic $0$. Let $X$ be a smooth and complete curve over $k$,
and $G$ a reductive group. We let $\cG$ denote the Langlands dual group, also viewed as an algebraic
group over $k$. 

\sssec{Formulation of the conjecture}  \label{sss:intr form}

The categorical geometric Langlands conjecture is supposed to compare two triangulated (or rather DG categories).
One is the ``geometric" (or ``automorphic") side that has to do with D-modules on the stack $\Bun_G$ of $G$-bundles on $X$. The other  
is the ``spectral" (or ``Galois") side that has to do with quasi-coherent sheaves on the stack $\LocSys_\cG$ on
$\cG$-local systems on $X$. 

\medskip

In our formulation of the conjecture, the geometric side is taken ``as is." I.e., we consider the DG category
$\Dmod(\Bun_G)$ of D-modules on $\Bun_G$. We refer the reader to \cite{DrGa2} for the definition of this
category and a discussion of its general properties (e.g., this category is \emph{compactly generated}
for non-tautological reasons). 

\medskip

A naive guess for the spectral side is the DG category $\QCoh(\LocSys_\cG)$. However, this guess turns
out to be slightly wrong, whenever $G$ is not a torus. A quick way to see that it is wrong is via the
compatibility of the conjectural geometric Langlands equivalence with the functor of Eisenstein series, see Property $\on{Ei}$
stated in \secref{sss:cond E}. Namely, if $P$ is a parabolic of $G$ with Levi quotient $M$, we have the Eisenstein series functors
$$\Eis_P:\Dmod(\Bun_M)\to \Dmod(\Bun_G) \text{ and }
\Eis_{\cP,\on{spec}}:\QCoh(\LocSys_\cM)\to \QCoh(\LocSys_\cG),$$
that are supposed to match up under the geometric Langlands equivalence (up to a twist by some line bundles).
However, this cannot be the case because the functor $\Eis_P$ preserves compactness (see \cite{DrGa3}), whereas 
$\Eis_{\cP,\on{spec}}$ does not.

\medskip

Our ``fix" for the spectral side is designed to make the above problem with Eisenstein series go away
in a minimal way (see \propref{p:Eis generation}). We observe that the non-preservation of compactness
by the functor $\Eis_{\cP,\on{spec}}$ has to do with the fact that the stack $\LocSys_\cG$ is not smooth.
Namely, it expresses itself in that some coherent complexes on $\LocSys_\cG$ are non-perfect. 

\medskip

Our modified version for the spectral side is the category that we denote $$\IndCoh_{\on{Nilp}^{\on{glob}}_\cG}(\LocSys_\cG),$$
see \secref{sss:define spectral side}. It is a certain enlargement of $\QCoh(\LocSys_\cG)$, whose definition uses the fact that
$\LocSys_\cG$ is a derived locally complete intersection, and the theory of singular support of coherent sheaves
for such stacks developed in \cite{AG}. 

\sssec{Idea of the proof}

The idea of the comparison between the categories $\Dmod(\Bun_G)$ and $\IndCoh_{\on{Nilp}^{\on{glob}}_\cG}(\LocSys_\cG)$
pursued in this paper is the following: we embed each side into a more tractable category and compare the 
essential images. 

\medskip

For the geometric side, the more tractable category in question is the category that we denote
$\Whit^{\on{ext}}(G,G)$, and refer to it as the \emph{extended Whittaker category}; the nature
of this category is explained in \secref{sss:explain ext Whit} below. The functor
$$\Dmod(\Bun_G)\to \Whit^{\on{ext}}(G,G)$$
(which, according to \conjref{c:Whit ext ff}, is supposed to be fully faithful) is that of \emph{extended Whittaker coefficient},
denoted $\on{coeff}_{G,G}^{\on{ext}}$. 

\medskip

For the spectral side, the more tractable category is denoted $\on{Glue}(\cG)_{\on{spec}}$, and the functor
$$\IndCoh_{\on{Nilp}^{\on{glob}}_\cG}(\LocSys_\cG)\to \on{Glue}(\cG)_{\on{spec}}$$
is denoted by $\on{Glue}(\on{CT}^{\on{enh}}_{\on{spec}})$
(this functor is fully faithful by \thmref{t:coLoc ext ff}). The idea of the pair 
$(\on{Glue}(\cG)_{\on{spec}},\on{Glue}(\on{CT}^{\on{enh}}_{\on{spec}}))$ is explained in 
\secref{sss:explain glue}. 

\medskip

We then claim (see Quasi-Theorems \ref{t:glued equiv} and \ref{t:glued functor}) that there exists a canonically
defined fully faithful functor
$$\BL_{G,G}^{\Whit^{\on{ext}}}:\on{Glue}(\cG)_{\on{spec}}\to \Whit^{\on{ext}}(G,G).$$

Thus, we have the following diagram
\begin{equation} \label{e:main diagram}
\CD
\on{Glue}(\cG)_{\on{spec}}  @>{\BL_{G,G}^{\Whit^{\on{ext}}}}>>  \Whit^{\on{ext}}(G,G) \\
@A{\on{Glue}(\on{CT}^{\on{enh}}_{\on{spec}})}AA    @AA{\on{coeff}_{G,G}^{\on{ext}}}A   \\
\IndCoh_{\on{Nilp}^{\on{glob}}_\cG}(\LocSys_\cG) & & \Dmod(\Bun_G),
\endCD
\end{equation}
with all the arrows being fully faithful.

\medskip

Assume that the essential images of the functors 
\begin{equation} \label{e:ess images}
\BL_{G,G}^{\Whit^{\on{ext}}} \circ \on{Glue}(\on{CT}^{\on{enh}}_{\on{spec}}) \text{ and }
\on{coeff}_{G,G}^{\on{ext}}
\end{equation}
coincide. We then obtain that diagram \eqref{e:main diagram} can be (uniquely) completed
to a commutative diagram by means of a functor 
$$\BL_G:\IndCoh_{\on{Nilp}^{\on{glob}}_\cG}(\LocSys_\cG) \to \Dmod(\Bun_G),$$
and, moreover, $\BL_G$ is automatically an equivalence.

\medskip

The required fact about the essential images of the functors \eqref{e:ess images} follows
from \conjref{c:Op gen}.

\sssec{The extended Whittaker category}  \label{sss:explain ext Whit} 
 
The extended Whittaker category $\Whit^{\on{ext}}(G,G)$ is defined as the DG category of D-modules
on a certain space (prestack), by imposing a certain equivariance condition. It may be easiest to explain what
$\Whit^{\on{ext}}(G,G)$ is via an analogy with the classical adelic picture.

\medskip

Assume for simplicity that $G$ has a connected center. Consider the adelic quotient $G(\BA)/G(\BO)$.
Let $\fch(K)$ denote the set of characters of the ad\`ele group $N(\BA)$  (here $N$ is the unipotent
radical of the Borel group $B$) that are trivial on $N(K)\subset N(\BA)$ (here $K$ denotes the global
field corresponding to $X$). The set $\fch(K)$ is naturally
acted on by the Cartan group $T(K)$ by conjugation.

\medskip

The space of functions which is the analog of the category $\Whit^{\on{ext}}(G,G)$ is the subspace
of all functions on the set $$G(\BA)/G(\BO)\times \fch(K)$$ that satisfy the following two conditions:

\medskip

\begin{itemize}

\item $f(t\cdot g,\on{Ad}_t(\chi))=f(g,\chi)$, \quad $t\in T(K)$, \quad $g\in G(\BA)/G(\BO)$, \quad $\chi\in \fch(K)$. 

\medskip

\item $f(n\cdot g,\chi)=\chi(n)\cdot f(g,\chi)$, \quad $g\in G(\BA)/G(\BO)$, \quad $\chi\in \fch(K)$, \quad $n\in N(\BA)$.

\end{itemize}

\medskip

The analog of the functor $\on{coeff}_{G,G}^{\on{ext}}$ is the map from the space of functions on
$$G(K)\backslash G(\BA)/G(\BO)$$ that takes a function $\wt{f}$ to 
$$f(g,\chi):=\underset{N(K)\backslash N(\BA)}\int\, \wt{f}(n\cdot g)\cdot \chi^{-1}(n).$$

\medskip

By construction, the category $\Whit^{\on{ext}}(G,G)$ is glued from the categories that we denote $\Whit(G,P)$
(here $P$ is a parabolic in $G$) and call ``degenerate Whittaker categories." In the function-theoretic analogy,    
for a parabolic $P$, the category $\Whit(G,P)$ corresponds to the subspace of functions supported on those
characters $\chi\in \fch(K)$ that satisfy: 

\medskip

\begin{itemize}

\item $\chi$ is non-zero on any simple root subgroup corresponding to roots inside $M$;

\medskip

\item $\chi$ is zero on any simple root \emph{not} in $M$. 

\end{itemize}

\medskip

One can rewrite this subspace as the space of functions $f$ on the set $G(\BA)/G(\BO)$ that satisfy

\medskip

\begin{itemize}

\item $f$ is invariant with respect to the subgroup $Z_M(K)$;

\medskip

\item $f$ is invariant with respect to $N(P)(\BA)$, where $N(P)$ is the unipotent radical of $P$;

\medskip

\item $f$ is equivariant with respect to $N(M)(\BA)$ against a fixed non-degenerate character,
where $N(M):=N\cap M$ and $M$ is the Levi subgroup of $P$.

\end{itemize}

\medskip

In particular, the ``open stratum" in $\Whit^{\on{ext}}(G,G)$ is a version of the usual Whittaker
category $\Whit(G,G)$  (with the imposed extra condition of equivariance with respect to the
group of rational points of $Z_G$). 

\medskip

The other extreme is the ``closed stratum", which is the principal
series category denoted $\on{I}(G,B)$. The latter is the analog of the space of functions on the double quotient
$$T(K)\cdot N(\BA)\backslash G(\BA)/G(\BO).$$

\medskip

The functor $\on{coeff}_{G,G}^{\on{ext}}$ can be thus thought of as taking for each parabolic 
the corresponding functor of \emph{constant term}, and then taking the non-degenerate Whittaker
coefficients for the Levi. 

\medskip

The category $\Whit^{\on{ext}}(G,G)$ is more tractable than the original category $\Dmod(\Bun_G)$
because it is comprised of the categories $\Whit(G,P)$, each of which is a combination of local
information and that involving a proper Levi subgroup. 

\sssec{The glued category on the spectral side}  \label{sss:explain glue}

The category $\on{Glue}(\cG)_{\on{spec}}$ is defined by explicitly gluing certain categories
$$\sF_{\cP}\mod(\QCoh(\LocSys_{\cP})),$$  
where $\cP$ runs through the poset of parabolic subgroups of $\cP$.

\medskip

Each category $\sF_{\cP}\mod(\QCoh(\LocSys_{\cP}))$ is defined as follows. We consider the map
$$\sfp_{\cP,\on{spec}}:\LocSys_\cP\to \LocSys_\cG,$$
and $\sF_{\cP}\mod(\QCoh(\LocSys_{\cP}))$ is the DG category of quasi-coherent sheaves on 
$\LocSys_\cP$ equipped with a connection along the fibers of the map $\sfp_{\cP,\on{spec}}$.

\medskip

The gluing functors, and the functor $\on{Glue}(\on{CT}^{\on{enh}}_{\on{spec}})$ are defined
naturally via pull-back, see \secref{ss:gluing spec} for details. 

\medskip

To explain the reason why the category $\on{Glue}(\cG)_{\on{spec}}$  is more tractable than the original
category $\IndCoh_{\on{Nilp}^{\on{glob}}_\cG}(\LocSys_\cG)$, let us consider the ``open stratum", i.e., the
category
$$\sF_{\cG}\mod(\QCoh(\LocSys_{\cG}))=\QCoh(\LocSys_\cG).$$

We claim that this category embeds fully faithfully into the ``open stratum" on the geometric side, i.e.,
the category $\Whit(G,G)$. This is shown by combining the following two results: 

\medskip

One is \propref{p:local to global LocSys}
that says that the category $\QCoh(\LocSys_\cG)$ admits a fully faithful functor
$$\on{co-Loc}_{\cG,\on{spec}}:\QCoh(\LocSys_\cG)\to \Rep(\cG)_{\Ran(X)},$$
where $\Rep(\cG)_{\Ran(X)}$ is a version of the category $\Rep(\cG)$ 
\emph{spread over the Ran space of $X$}. \footnote{The notation ``$\on{co-Loc}_{\cG,\on{spec}}$" is not intended to suggest
that this functor is a \emph{co-localization} in the sense of category theory (i.e., admits a fully faithful left adjoint). Rather,
it is the right adjoint to a functor $\on{Loc}_{\cG,\on{spec}}$, which is a localization-type functor in the sense of \cite{BB}.
The latter happens to be a localization in the sense of category theory as its right adjoint, i.e., $\on{co-Loc}_{\cG,\on{spec}}$ is fully faithful.}

\medskip

The second is a geometric version of the Casselman-Shalika formula,
Quasi-Theorem \ref{t:Cass Shal}, that says that $\Whit(G,G)$ is equivalent to a category obtained by slightly modifying 
$\Rep(\cG)_{\Ran(X)}$. 

\sssec{Comparing the essential images}  \label{sss:comparing images}

Finally, let us comment on the last step of the proof, namely, the comparison of the essential images in diagram
\eqref{e:main diagram}. 

\medskip

The idea is to show that there exist two families of objects 
$$\CF_a\in \IndCoh_{\on{Nilp}^{\on{glob}}_\cG}(\LocSys_\cG)  \text{ and }
\CM_a\in \Dmod(\Bun_G),$$
parameterized by the same set $A$, such that

\begin{itemize}

\item The objects $\CF_a$ generate $\IndCoh_{\on{Nilp}^{\on{glob}}_\cG}(\LocSys_\cG)$;

\item The objects $\CM_a$ generate $\Dmod(\Bun_G)$;

\item For each $a\in A$ we have an isomorphism
$$\BL_{G,G}^{\Whit^{\on{ext}}} \circ \on{Glue}(\on{CT}^{\on{enh}}_{\on{spec}})(\CF_a)\simeq 
\on{coeff}_{G,G}^{\on{ext}}(\CM_a).$$

\end{itemize}

We construct the required families $\CF_a$ and $\CM_a$ as follows. By induction on the rank,
we can assume that the geometric Langlands conjecture holds for proper Levi
subgroups of $G$. Then Quasi-Theorem \ref{t:CT and Eis} implies that for a proper parabolic
$P$ with Levi quotient $M$, we have a diagram
$$
\CD
\on{Glue}(\cG)_{\on{spec}}  @>{\BL_{G,G}^{\Whit^{\on{ext}}}}>>   \Whit^{\on{ext}}(G,G) \\
@A{\on{Glue}(\on{CT}^{\on{enh}}_{\on{spec}})}AA    @AA{\on{coeff}_{G,G}^{\on{ext}}}A   \\
\IndCoh_{\on{Nilp}^{\on{glob}}_\cG}(\LocSys_\cG) & & \Dmod(\Bun_G) \\
@A{\Eis_{\cP,\on{spec}}}AA     @AA{\Eis_P}A  \\
\IndCoh_{\on{Nilp}^{\on{glob}}_\cM}(\LocSys_\cM)    @>{\BL_M}>>  \Dmod(\Bun_M) 
\endCD
$$
that commutes up to a (specific) self-equivalence of $\Dmod(\Bun_M)$.  Here $\Eis_{\cP,\on{spec}}$ and $\Eis_P$
are the Eisenstein series functors on the spectral and geometric sides, respectively.

\medskip

However, the essential images of the functor $\Eis_{\cP,\on{spec}}$ (resp., $\Eis_P$) for all proper parabolics
$P$ are not sufficient to generate the category $\IndCoh_{\on{Nilp}^{\on{glob}}_\cG}(\LocSys_\cG)$
(resp., $\Dmod(\Bun_G)$). Namely, on the spectral side we are missing the entire locus of \emph{irreducible}
local systems, and on the geometric side the full subcategory $\Dmod(\Bun_G)_{\on{cusp}}$ corresponding
to cuspidal objects. 

\medskip

Another family of objects is provided by the commutative diagram
\begin{equation} \label{e:pentagon}
\xy 
(-25,25)*+{\on{Glue}(\cG)_{\on{spec}} }="X";
(25,25)*+{\Whit^{\on{ext}}(G,G)}="Y";
(-25,0)*+{\IndCoh_{\on{Nilp}^{\on{glob}}_\cG}(\LocSys_\cG) }="Z";
(25,0)*+{\Dmod(\Bun_G) }="W";
(0,-25)*+{\QCoh(\on{Op}(\cG)^{\on{glob}}_{\lambda^I})}="U";
{\ar@{->}^{\BL_{G,G}^{\Whit^{\on{ext}}}} "X";"Y"};
{\ar@{->}^{\on{Glue}(\on{CT}^{\on{enh}}_{\on{spec}})} "Z";"X"};
{\ar@{->}_{\on{coeff}_{G,G}^{\on{ext}}} "W";"Y"};
{\ar@{->}^{(\sv_{\lambda^I})_*} "U";"Z"};
{\ar@{->}_{\on{q-Hitch}_{\lambda^I}} "U";"W"};
\endxy
\end{equation}

Here $\QCoh(\on{Op}(\cG)^{\on{glob}}_{\lambda^I})$ is the scheme of global \emph{opers}
on the curve $X$ with specified singularities (encoded by the index $\lambda^I$), see \secref{s:opers}.

\medskip

The functor $(\sv_{\lambda^I})_*$ is that of direct image with respect to the natural forgetful map
$$\sv_{\lambda^I}:\QCoh(\on{Op}(\cG)^{\on{glob}}_{\lambda^I})\to \LocSys_\cG.$$

The functor $\on{q-Hitch}_{\lambda^I}$ is obtained by generalizing the construction of
\cite{BD2} that attaches objects in $\Dmod(\Bun_G)$ to quasi-coherent sheaves on the scheme
of opers.

\medskip

Now, the essential images of the functors $\Eis_P$ (for all proper parabolics $P$) and those
of the functors $\on{q-Hitch}_{\lambda^I}$ do generate $\Dmod(\Bun_G)$ by \thmref{t:generate BunG}.

\medskip

The generation of $\IndCoh_{\on{Nilp}^{\on{glob}}_\cG}(\LocSys_\cG)$ by the essential images of the
functors $\Eis_{\cP,\on{spec}}$ and $(\sv_{\lambda^I})_*$ follows from \conjref{c:Op gen}.

\sssec{Summary}

One can summarize the idea of the proof as playing off against each other the operations of taking the (extended)
Whittaker coefficient and the Beilinson-Drinfeld construction of D-modules on $\Bun_G$ via opers, and tracing through
the corresponding operations on the spectral side.

\medskip

We should remark that the compatibility of the two operations on the geometric and spectral sides is the limiting
case of the more general quantum Langlands phenomenon. This idea was present and explored in the papers 
\cite{Fr} (specifically, Sect. 6.4) and \cite{Sto}; these papers record part of the research in this direction, carried 
out by B.~Feigin, E.~Frenkel and A.~Stoyanovsky in the early 90's. 

\ssec{Other approaches to the construction of the functor}

\sssec{The Drinfeld-Laumon approach: the case of an arbitrary reductive group}

Let $G$ be still an arbitrary reductive group. Let 
$$\LocSys_\cG^{\on{irred}}\overset{\jmath}\hookrightarrow \LocSys_\cG$$
be the embedding of the locus of irreducible local systems. By a slight abuse
of notation we shall denote by $\jmath_*$ the functor
\footnote{In fact, the difference
between the two categories  
$\QCoh(\LocSys_\cG)$ and $\IndCoh_{\on{Nilp}^{\on{glob}}_\cG}(\LocSys_\cG)$
disappears once we restrict to $\LocSys_\cG^{\on{irred}}$.}
$$\QCoh(\LocSys_\cG^{\on{irred}})\to \QCoh(\LocSys_\cG)\to \IndCoh_{\on{Nilp}^{\on{glob}}_\cG}(\LocSys_\cG).$$

The resulting functor 
$$\BL_G\circ \jmath_*:\QCoh(\LocSys^{\on{irred}}_\cG)\to \Dmod(\Bun_G)$$
can be described as follows:

\medskip

Starting from $\CF\in \QCoh(\LocSys^{\on{irred}}_\cG)$, we regard it as an object of 
$\on{Glue}(\cG)_{\on{spec}}$ extended by zero from the ``open stratum''
$$\QCoh(\LocSys_\cG)\hookrightarrow \on{Glue}(\cG)_{\on{spec}}.$$

Applying the functor 
$$\BL_{G,G}^{\Whit^{\on{ext}}}:\on{Glue}(\cG)_{\on{spec}}\to \Whit^{\on{ext}}(G,G),$$
we obtain an object extended by zero from the ``open stratum" 
$$\Whit(G,G)\hookrightarrow \Whit^{\on{ext}}(G,G).$$

I.e., we do not need to worry about constant terms and gluing; our sought-for object
of $\Bun_G$ will be cuspidal, and thus will only have non-degenerate Whittaker coefficients. 

\medskip

One can interpret the Drinfeld-Laumon approach (which takes its origin in the classical
theory of automorphic functions) as attempting to prove directly that the above object 
$$\BL_{G,G}^{\Whit^{\on{ext}}}\circ \jmath_*(\CF)\in \Whit^{\on{ext}}(G,G)$$
uniquely descends to an object 
$$\BL_G\circ \jmath_*(\CF)\in \Dmod(\Bun_G).$$

\medskip

The difference between this approach and one in the present paper is that instead of proving
the descent statement mentioned above for an arbitrary $\CF\in \QCoh(\LocSys^{\on{irred}}_\cG)$,
we do it on the set of generators of that category. These are given as direct images
$$\sv_{\lambda^I}(\CF),\quad \CF\in \QCoh(\on{Op}(\cG)^{\on{glob,irred}}_{\lambda^I}).$$

For such objects descent is proved by pinpointing the corresponding object of $\Dmod(\Bun_G)$. Namely,
it is one given by the Beilinson-Drinfeld construction, i.e., $\on{q-Hitch}_{\lambda^I}(\CF)$. 
 
\sssec{The Drinfeld-Laumon approach: the case of $GL_n$}

In reality, the Drinfeld-Laumon approach as it appears in \cite{Dr}, \cite{Lau1} and \cite{Lau2}, and developed 
further in \cite{FGKV} and \cite{FGV2}, is specialized to the case of $G=GL_n$. 

\medskip

The main feature of this special case is that one can replace the space (prestack) on which we realize
$\Whit^{\on{ext}}(G,G)$ by an actual algebraic stack  \footnote{This was crucial at the time of writing of \cite{FGV2},
as it was not clear how to define or deal with D-modules on arbitrary prestacks.},
at the cost of losing fully-faithfulness of the functor
$$\on{coeff}_{G,G}^{\on{ext}}:\Dmod(\Bun_G)\to \Whit^{\on{ext}}(G,G).$$

\medskip

In our notations, the construction of \cite{FGV2} can be interpreted as follows. One 
introduces a certain full subcategory 
$$\Whit(G,G)^{\on{non-polar;ext}}\subset \Whit^{\on{ext}}(G,G),$$
(whose definition only involves usual algebraic stacks). The above inclusion admits a right adjoint, denoted
$$\Upsilon:\Whit^{\on{ext}}(G,G)\to \Whit(G,G)^{\on{non-polar;ext}},$$
and one considers the functor 
$$\on{coeff}_{G,G}^{\on{non-polar;ext}}=\Upsilon\circ \on{coeff}_{G,G}^{\on{ext}},\quad 
\Dmod(\Bun_G)\to \Whit(G,G)^{\on{non-polar;ext}}.$$

\medskip

The functor $\on{coeff}_{G,G}^{\on{non-polar;ext}}$ is no longer fully faithful. However, it has 
the property that it is fully faithful on the subcategory
$$\Dmod(\overset{\circ}\Bun_G)^\heartsuit\subset \Dmod(\Bun_G),$$
where $\overset{\circ}\Bun_G$ is the open substack corresponding to $G$-bundles (i.e., rank $n$ vector bundles) with
vanishing $H^1$, and where the superscript ``$\heartsuit$" denotes the heart of the t-structure. 

\medskip

Starting from an irreducible $n$-dimensional local system $\sigma$ on $X$, one wants to construct the
corresponding object 
$$\CM_\sigma:=\BL_G(k_\sigma)\in \Dmod(\Bun_G),$$
where $k_\sigma$ is the sky-scraper at the point $\sigma\in \LocSys^{\on{irred}}_\cG$. 

\medskip

To $\sigma$ one explicitly associates an object of 
$\Whit(G,G)^{\on{non-polar;ext}}$, which in our notations is 
\begin{equation} \label{e:to descend}
\Upsilon \circ \BL_{G,G}^{\Whit^{\on{ext}}}\circ \on{Glue}(\on{CT}^{\on{enh}}_{\on{spec}}) \circ 
\jmath_*(k_\sigma)\in \Whit(G,G)^{\on{non-polar;ext}},
\end{equation}
or, by slightly abusing the notation and ignoring the contributions of proper parabolics, 
$$\Upsilon  \circ \BL_{G,G}^{\Whit}\circ \on{co-Loc}_{\cG,\on{spec}}(k_\sigma).$$

One constructs the restriction of $\CM_\sigma$ to $\Bun^d_G$  (the connected component of $\Bun_G$
corresponding to vector bundles of degree $d$) for $d\gg 0$ by showing that the direct summand of 
\eqref{e:to descend} living over $\Bun_G^d$ descends to (=canonically comes as the image under 
$\on{coeff}_{G,G}^{\on{non-polar;ext}}$ of) an object of $\Dmod(\Bun^d_G)$ by making heavy use of the t-structures on 
$\Dmod(\Bun_G)$ and $\Whit(G,G)^{\on{non-polar;ext}}$. 

\sssec{The Beilinson-Drinfeld approach via opers}

Extending the construction of \cite{BD2}, one may attempt to define a functor
$$\BL_G|_{\QCoh(\LocSys_\cG)}:\QCoh(\LocSys_\cG)\to \Dmod(\Bun_G)$$
by requiring that the diagram
$$
\xy 
(-25,0)*+{\QCoh(\LocSys_\cG)}="Z";
(25,0)*+{\Dmod(\Bun_G) }="W";
(0,-25)*+{\QCoh(\on{Op}(\cG)^{\on{glob}}_{\lambda^I})}="U";
{\ar@{->}^{\BL_G} "Z";"W"};
{\ar@{->}^{(\sv_{\lambda^I})_*} "U";"Z"};
{\ar@{->}_{\on{q-Hitch}_{\lambda^I}} "U";"W"};
\endxy
$$
be commutative for every parameter $\lambda^I$. 

\medskip

This would be possible if one knew \conjref{c:Op quot}. 

\sssec{Beilinson's spectral projector}

There exists yet one more approach to the construction of the functor 
$$\BL_G|_{\QCoh(\LocSys_\cG)}:\QCoh(\LocSys_\cG)\to \Dmod(\Bun_G).$$

It is based on the idea that the Hecke functors applied at all points of $X$
comprise an action of the symmetric monoidal category $\QCoh(\LocSys_\cG)$
on $\Dmod(\Bun_G)$. A precise statement along these lines is formulated as
\thmref{t:generalized vanishing}. 

\medskip

This does indeed define the restriction of the functor $\BL_G$ to 
$$\QCoh(\LocSys_\cG)\subset \IndCoh_{\on{Nilp}^{\on{glob}}_\cG}(\LocSys_\cG)$$
by applying the above action to
the object of $\Dmod(\Bun_G)$ that is supposed to correspond under $\BL_G$
to $\CO_{\LocSys_\cG}\in \QCoh(\LocSys_\cG)$. This object is identified in 
\secref{sss:1st Whittaker coefficient and Langlands}.

\ssec{What is new in this paper?}

\sssec{Old and new ideas}  \label{sss:old and new}

Many of the ideas present in this paper are not at all new. 

\medskip

The basic initial idea is the same one as in the Drinfeld-Laumon approach. It consists of  accessing the 
category $\Dmod(\Bun_G)$ through the Whittaker model, while the latter can be directly compared to the spectral side. 

\medskip

The next idea was already mentioned above: it consists of playing off the functors
$$\on{coeff}_{G,G}:\Dmod(\Bun_G)\to \Whit(G,G) \text{ and } \on{Loc}_G:\on{KL}(G,\crit)_{\Ran(X)}\to \Dmod(\Bun_G)$$
against each other, and comparing them to their counterparts on the spectral side. 

\medskip

I.e., we want to complete both the Drinfeld-Laumon approach and the Beilinson-Drinfeld approach 
to an equivalence of categories, by comparing them to each other. As was mentioned already, the fruitfulness 
of such a comparison was explored already in \cite{Fr} and \cite{Sto}. 

\medskip

Among the new ideas one could mention the following ones: (a) the modification of the spectral side,
given by $\IndCoh_{\on{Nilp}^{\on{glob}}_\cG}(\LocSys_\cG)$; (b) the idea that one can consider D-modules
on arbitrary prestacks rather than algebraic stacks or ind-algebraic stacks; (c) categories living over the
Ran space and ``local-to-global" constructions they give rise to; (d) contractibility of the space of generically defined 
maps from $X$ to a connected algebraic group.

\medskip

All of these ideas became available as a result of bringing the machinery of derived algebraic geometry and
higher category theory to the paradigm of Geometric Langlands. We learned about these subjects from 
J.~Lurie. 

\sssec{Ind-coherent sheaves and singular support} \hfill

\medskip

The definition of $\IndCoh_{\on{Nilp}^{\on{glob}}_\cG}(\LocSys_\cG)$ is based on the theory of singular support
of coherent sheaves on a scheme (or algebraic stack) which is a \emph{derived locally complete intersection}. 
This theory is developed in \cite{AG} and reviewed in \secref{s:sing supp}. 

\medskip

The idea of singular support is also an old one, and apparently goes back to D.~Quillen. Given 
a triangulated category $\CC$, an object $c\in \CC$ and an evenly graded commutative algebra $A$ 
mapping to the algebra $$\underset{i}\oplus\, \on{Ext}^{2i}_\CC(c,c),$$
it makes sense to say that $c$ is supported over a given Zariski-closed subset of $\Spec(A)$. 

\medskip

When $\CC$ is the derived category of quasi-coherent sheaves on an affine DG scheme $X$, we take $A$
to be the even part of $\on{HH}(X)$, the Hochschild cohomology algebra of $X$. 

\medskip

So, at the end of the day, singular support of a coherent sheaf $\CF$ measures which cohomological operations
$\CF\to \CF[2i]$ vanish when iterated a large number of times. 

\medskip

Further details are given in \secref{s:sing supp}.

\sssec{D-modules on prestacks}

The idea of considering D-modules on prestacks is really an essential one for this paper. Here is a typical example of
a prestack,
considered in \secref{s:Whit} and denoted $\Bun_G^{B\on{-gen}}$, which is used in the definition of the category 
$\Whit^{\on{ext}}(G,G)$. 

\medskip

The prestack $\Bun_G^{B\on{-gen}}$ classifies $G$-bundles on $X$, equipped with a reduction to the Borel subgroup $B$,
\emph{defined generically on $X$}. 

\medskip

The idea to realize $\Whit^{\on{ext}}(G,G)$ on $\Bun_G^{B\on{-gen}}$ defined as above was suggested by J.~Barlev. 

\medskip

Of course, for an arbitrary prestack $\CY$, the category $\Dmod(\CY)$, although well-defined, will be pretty
intractable. In the case of  $\Bun_G^{B\on{-gen}}$, the nice properties of $\Dmod(\Bun_G^{B\on{-gen}})$ are ensured
by \propref{p:Barlev} that says that $\Bun_G^{B\on{-gen}}$ can be realized as a quotient of an algebraic stack by
a schematic and proper equivalence relation. 

\medskip

One feature of $\Dmod(\Bun_G^{B\on{-gen}})$ is that it \emph{does not have} a t-structure with the usual
properties of a t-structure of the category of D-modules on a scheme or algebraic stack. For many people,
including the author, this is one of the reasons why this category has not been considered earlier.

\medskip

Another class of examples of prestacks has to do with the \emph{Ran space} of $X$, denoted $\Ran(X)$, which classifies
\emph{non-empty finite subsets of $X$}. 

\sssec{Local-to-global}  \label{sss:rep over Ran new}

To give an example of a ``local-to-global'' principle employed in this paper, we consider the category $\QCoh(\LocSys_\cG)$.
This is a ``global" object, since the stack $\LocSys_\cG$ itself is of global nature, as it depends on the curve $X$.

\medskip

The corresponding local category is $\Rep(\cG)$ of algebraic representations of $\cG$. We spread it over the Ran space
and obtain the category, denoted $\Rep(\cG)_{\Ran(X)}$, introduced in \secref{ss:rep Ran}. The ``local-to-global'' principle
for this case, stated in \propref{p:local to global LocSys}, says that there is a pair of adjoint functors 
$$\on{Loc}_{\cG,\on{spec}}:\Rep(\cG)_{\Ran(X)}\rightleftarrows \QCoh(\LocSys_\cG):\on{co-Loc}_{\cG,\on{spec}}$$
with the right adjoint $\on{co-Loc}_{\cG,\on{spec}}$ being fully faithful. 

\medskip

Hence, we obtain a fully faithful functor from a ``global" category to a ``local" one, which is what we mean by
a ``local-to-global'' principle. 

\medskip

We will also encounter non-trivial generalizations of the above example in Quasi-Theorems 
\ref{t:principal series} and \ref{t:deg Whit spec}. However, the corresponding ``local-to-global"
principle will not appear in the statement, but rather constitutes one of the steps in the proof, which is
not discussed explicitly in the paper.

\medskip

To explain its flavor, we consider the following example. Consider the natural map
$$\sfp_{\cB,\on{spec}}:\LocSys_\cB\to \LocSys_\cG.$$
We are interested in the category, denoted by  $\sF_{\cB}\mod(\QCoh(\LocSys_{\cB}))$, mentioned
in \secref{sss:explain glue}. It is equipped with a pair of adjoint functors:
$$\ind_{\sF_\cB}:\QCoh(\LocSys_{\cB})\rightleftarrows \sF_{\cB}\mod(\QCoh(\LocSys_{\cB})):\oblv_{\sF_\cB},$$

\medskip

The composition 
$\oblv_{\sF_\cB}\circ \ind_{\sF_\cB}:\QCoh(\LocSys_{\cB})\to \QCoh(\LocSys_{\cB})$
has thus a structure of monad (i.e., algebra object in the monoidal category $\End(\QCoh(\LocSys_{\cB}))$).
We would like to describe this monad in ``local" terms. 

\medskip

The latter turns out to be possible. The answer is given in terms of the Ran version of the \emph{spectral Hecke stack}
(see \secref{sss:spec Ran Hecke}) and is obtained by generalizing the construction of \cite{Ro}.

\sssec{Contractibility}

Finally, the contractibility result mentioned in \secref{sss:old and new} says that if $H$ is a connected
affine algebraic group, then the prestack $\bMaps(X,H)^{\on{gen}}$ that classifies maps from $X$ to $H$,
\emph{defined generically on $X$}, is homologically contractible.  

\medskip

The latter means that the pull-back functor 
$$\Vect=\Dmod(\on{pt})\to \Dmod(\bMaps(X,H)^{\on{gen}})$$
is fully faithful.

\medskip

This result is the reason behind the validity of \thmref{t:Whit for GLn} (fully faithfulness 
of the functor $\on{coeff}_{G,G}^{\on{ext}}$) for $GL_n$.

\medskip

We also note (although this is not used in this paper), that the above-mentioned contractibility
provides a ``local-to-global'' principle on the geometric side. Namely, as is explained in \cite[Sect. 4.1]{Ga2},
it implies that the pull-back functor
$$\Dmod(\Bun_G)\to \Dmod(\Gr_{G,\Ran(X)})$$
is fully faithful, where $\Gr_{G,\Ran(X)}$ is the Ran version of the affine Grassmannian of the group $G$.

\ssec{Notations and conventions}

\sssec{The theory of $\infty$-categories}

Even though the statement of the categorical geometric Langlands conjecture can be perceived 
as an equivalence of two triangulated categories (rather than DG categories), the 
language of $\infty$-categories is essential for this paper. The main reason they appear is the following:

\medskip

Some of the crucial constructions in this paper use fact that we can define the (DG) 
category of D-modules (and quasi-coherent sheaves) on an arbitrary prestack. 
The latter category is, by definition, constructed as a limit \emph{taken in the $\infty$-category
of DG categories}, see \secref{sss:prestack}.  

\medskip

So, essentially all we need is to have the
notion of \emph{diagram of DG categories, parameterized by some index category (which is
typically an ordinary  category)}, and to have the ability to take the limit of such a diagram. 
Now, to have such an ability (and to know some of its basic properties) amounts to including
\cite[Chapters 1-5]{Lu} into our tool kit. 

\medskip

We will not attempt to review the theory of $\infty$-categories here. \footnote{That said, the reader who is
completely new to $\infty$-categories can pretend that the notion of $\infty$-category is an enhancement of
that of ordinary category. The main point of difference is that morphisms between two objects no longer
form a set, but rather an $\infty$-groupoid, i.e., a non-discrete homotopy type.} An excellent review is provided by 
\cite[Chapter 1]{Lu}.  So, our suggestion to the reader is to familiarize
oneself with {\it loc.cit.} and start using the theory pretending having a full knowledge of it
(knowing the proofs from the bulk of \cite{Lu} will not really enhance one's ability to understand 
how the theory is applied in practice). 

\sssec{}

The conventions in this paper regarding $\infty$-categories and DG categories
follow verbatim those adopted in \cite{DrGa1}. The most essential ones are:

\medskip

\noindent(i) When we say ``category" by default we mean ``$(\infty,1)$"-category.  

\smallskip

\noindent(ii) For a category
$\bC$ and objects $\bc_1,\bc_2\in \bC$ we shall denote by $\Maps_\bC(\bc_1,\bc_2)$
the $\infty$-groupoid of maps between them. We shall denote by $\Hom_\bC(\bc_1,\bc_2)$
the \emph{set} $\pi_0(\Maps_\bC(\bc_1,\bc_2))$, i.e., $\Hom$ in the ordinary category
$\on{Ho}(\bC)$. 

\smallskip

\noindent(iii) All DG categories are assumed to be pretriangulated and, unless explicitly stated otherwise, 
cocomplete (that is, they contain arbitrary direct sums). 
All functors between DG categories are assumed to be exact and continuous (that is, commuting with arbitrary 
direct sums, or equivalently, with all colimits). In particular, all subcategories are by default
assumed to be closed under arbitrary direct sums.

\smallskip

\noindent(iv) We let $\Vect$ denote the DG category of complexes of vector spaces;
thus, the usual category of $k$-vector spaces is denoted by $\Vect^\heartsuit$. 

\smallskip

\noindent(v) 
The category of $\infty$-groupoids is denoted by $\inftygroup$.

\sssec{}

Our conventions regarding DG schemes and prestacks follow verbatim those adopted in \cite{DrGa1},
Sect. 0.6.8-0.6.9.

\ssec{Acknowledgements}

Geometric Langlands came into existence as a result of the pioneering papers of 
V.~Drinfeld and G.~Laumon. The author would like to thank them for creating this
field, which provided the main vector of motivation for him as well as numerous other people. 

\medskip

The author is tremendously grateful to D.~Arinkin and V.~Drinfeld for collaboration on this project.

\medskip

The author would like to thank J.~Barlev, D.~Beraldo, G.~Fortuna, S.~Raskin, R.~Reich, N.~Rozenblyum
and S.~Schieder for taking up various aspects of geometric Langlands as their own projects. 

\medskip

The approach to geometric Langlands developed in this paper became possible after J.~Lurie
taught us how to use $\infty$-categories for problems in geometric representation theory. Our
debt to him is huge. 

\medskip

The author is grateful to A.~Beilinson, J.~Bernstein, R.~Bezrukavnikov, A.~Braverman, M.~Finkelberg,
E.~Frenkel, D.~Kazhdan, V.~Lafforgue, J.~Lurie, S.~Lysenko, I.~Mirkovic, K.~Vilonen and E.~Witten, for illuminating discussions related 
to geometric Langlands that we have had over many years. 

\section{A roadmap to the contents}

The structure of the main body of the paper may not make it obvious what role each section plays in the construction
of the geometric Langlands equivalence, so we shall now proceed to describe the contents and
main ideas of each section.

\ssec{Singular support and the statement of the geometric Langlands equivalence}

\sssec{}

In \secref{s:sing supp} we review the theory of singular support of coherent sheaves. This is needed in order
to define the spectral side of geometric Langlands.

\medskip

We first define the notion of quasi-smooth DG scheme (a.k.a. derived locally complete intersection). These are DG
schemes for which the notion of singular support is defined. We then proceed to the definition of singular support
itself via cohomological operations.

\medskip

In the next step we review the theory of ind-coherent sheaves, and define the main player for the spectral
side of geometric Langlands, the category of ind-coherent sheaves with specified singular support. 
We first do it for DG schemes, and then for algebraic stacks.

\medskip

In the process we review the construction of $\QCoh$ on an arbitrary prestack, and of its renormalized version,
denoted $\IndCoh$, on an algebraic stack. 

\sssec{}

In \secref{s:statement} we state the geometric Langlands conjecture according to the point of view taken
in this paper. 

\medskip

We first take a look at the stack $\LocSys_\cG$ and explain why it is quasi-smooth, and describe the corresponding 
stack $\Sing(\LocSys_\cG)$.  Then we introduce
the spectral side of geometric Langlands as the category $\IndCoh_{\on{Nilp}^{\on{glob}}_\cG}(\LocSys_\cG)$. 

\medskip

We state the geometric Langlands conjecture as the existence and uniqueness of an equivalence
$$\BL_G:\IndCoh_{\on{Nilp}^{\on{glob}}_\cG}(\LocSys_\cG)\to \Dmod(\Bun_G)$$
satisfying the property of compatibilty with the extended Whittaker model, denoted
$\on{Wh}^{\on{ext}}$. This property itself will be stated in \secref{s:gluing} after
a good deal of preparations. As part of the statement of the geometric Langlands conjecture we 
include the compatibility with Hecke functors, Eisenstein series, and Kac-Moody localization.

\medskip

Finally, we introduce the full subcategory
$$\Dmod(\Bun_G)_{\on{temp}}\subset \Dmod(\Bun_G)$$
which under the (conjectural) equivalence $\BL_G$ corresponds to
$$\QCoh(\LocSys_\cG)\subset \IndCoh_{\on{Nilp}^{\on{glob}}_\cG}(\LocSys_\cG).$$

\ssec{Hecke action}

\sssec{}

\secref{s:Hecke} is devoted to the discussion of Hecke functors on both the geometric and spectral
sides of geometric Langlands. This is needed in order to formulate the property of the functor
$\BL_G$ which has been traditionally perceived as the main property satisfied by Langlands 
correspondence, and also for one of the crucial steps in the proof of the existence of $\BL_G$
(used in Sects. \ref{sss:irred 1}-\ref{sss:irred 2}).

\medskip

We begin by discussing what we call \emph{the naive geometric Satake}. We consider the category $\Rep(\cG)$,
and consider its version spread over the Ran space, denoted $\Rep(\cG)_{\Ran(X)}$, and the pair
of adjoint functors
$$\on{Loc}_{\cG,\on{spec}}:\Rep(\cG)_{\Ran(X)}\rightleftarrows \QCoh(\LocSys_\cG):\on{co-Loc}_{\cG,\on{spec}},$$
already mentioned in \secref{sss:rep over Ran new}. 

\medskip

By  \propref{p:local to global LocSys}, the functor 
$\on{Loc}_{\cG,\on{spec}}$ realizes $\QCoh(\LocSys_\cG)$ as a monoidal quotient category of 
$\Rep(\cG)_{\Ran(X)}$. 

\medskip

We quote \propref{p:Satake Hecke} which can be regarded as stating the existence of 
the \emph{naive geometric Satake functor}, denoted 
$$\on{Sat}(G)^{\on{naive}}_{\Ran(X)}:\Rep(\cG)_{\Ran(X)}\to \Dmod(\on{Hecke}(G)_{\Ran(X)}).$$

\medskip

The functor $\on{Sat}(G)^{\on{naive}}_{\Ran(X)}$ defines an 
action of the monoidal category $\Rep(\cG)_{\Ran(X)}$ on $\Dmod(\Bun_G)$. We
then proceed to \thmref{t:generalized vanishing}, which says that the above action 
factors through an action of the monoidal category $\QCoh(\LocSys_\cG)$ on $\Dmod(\Bun_G)$.

\medskip

The property of compatibility of the geometric Langlands equivalence with the Hecke action says that $\BL_G$
intertwines the natural action of $\QCoh(\LocSys_\cG)$ on $\IndCoh_{\on{Nilp}^{\on{glob}}_\cG}(\LocSys_\cG)$
(by pointwise tensor product) with the above action of $\QCoh(\LocSys_\cG)$ on $\Dmod(\Bun_G)$.

\sssec{}

Next we indicate (but do not discuss in full detail) the extension of the \emph{naive geometric Satake}
to the \emph{full geometric Satake}. The latter involves an analog of the Hecke stack on the spectral side,
and says that the functor $\on{Sat}(G)^{\on{naive}}_{\Ran(X)}$ can be extended to a monoidal functor 
$$\on{Sat}(G)_{\Ran(X)}:\IndCoh(\on{Hecke}(\cG,\on{spec})^{\on{loc}}_{\Ran(X)}))\to \Dmod(\on{Hecke}(G)_{\Ran(X)}).$$

The functor $\on{Sat}(G)_{\Ran(X)}$ can be used to intrinsically characterise the full subcategory
$$\Dmod(\Bun_G)_{\on{temp}}\subset \Dmod(\Bun_G)$$ mentioned above. 

\ssec{Whittaker and parabolic categories}

Sects. \ref{s:Whit}-\ref{s:gluing} contain the bulk of the geometric constructions in this paper. It is 
these sections that contain the ``quasi-theorems" on which hinges the proof of the geometric Langlands 
conjecture. 

\medskip

These sections deal with the various versions (i.e., genuine, degenerate and extended) of the Whittaker category,
and the parabolic category.  In each case there is a quasi-theorem that describes the corresponding 
category in spectral terms. Some of the quasi-theorems rely on the validity of the geometric Langlands
conjecture for proper Levi subgroups of $G$, and some do not. 
 
\sssec{The genuine Whittaker category}  \label{sss:idea Whit}

\secref{s:Whit} is devoted to the discussion of two versions of the genuine Whittaker category, denoted
$\Whit(G)$ and $\Whit(G,G)$, respectively. 

\medskip

In the function-theoretic analogy, the category $\Whit(G)$ corresponds to the space of functions on
$G(\BA)/G(\BO)\times \fch(K)$ that satisfy 
$$f(n\cdot g,\chi)=\chi(n)\cdot f(g,\chi), \quad g\in G(\BA)/G(\BO), \quad n\in N(\BA),$$
where $\chi$ is a non-degenerate charcater on $N(\BA)$ trivial on $N(K)$. 

\medskip

The category $\Whit(G,G)$ is a full subcategory of $\Whit(G)$, where we impose an extra condition
of invariance with respect to $Z_G(K)$.

\medskip

In what follows, for simplicity, we will discuss $\Whit(G)$. We realize $\Whit(G)$ as the category
of D-modules on a certain prestack satisfying an equivariance condition with respect to a certain
groupoid against a canonically defined character.

\medskip

The prestack in question, denoted $\CQ_G$, is a version of the prestack $\Bun_G^{B\on{-gen}}$ mentioned above.
The difference is that in addition to the data of a generic reduction of our $G$-bundle to $B$,
we specify the data of (generic) identification of the induced $T$-bundle with one
induced by the cocharacter $2\check\rho$ from the line bundle $\omega_X^{\frac{1}{2}}$,
where $\omega_X$ is the canonical line bundle on $X$, and $\omega_X^{\frac{1}{2}}$ is its
(chosen once and for all) square root. Up to generically trivializing $\omega_X^{\frac{1}{2}}$, the prestack $\CQ_G$
identifies with $\Bun_G^{N\on{-gen}}$, and its set of $k$-points
of $\CQ_G$ identifies with the double quotient
$$N(K)\backslash G(\BA)/G(\BO).$$

\sssec{}  \label{sss:tricky groupoid}

The definition of the groupoid involved in the definition of $\Whit(G)$, denoted $\bN$, is trickier. 
At the level of functions, one would like to consider the groupoid
$$
\xy
(-25,0)*+{N(K)\backslash G(\BA)/G(\BO)}="X";
(25,0)*+{N(K)\backslash G(\BA)/G(\BO),}="Y";
(0,25)*+{N(K)\backslash  N(\BA)\overset{N(K)}\times G(\BA)/G(\BO)}="Z";
{\ar@{->} "Z";"X"};
{\ar@{->} "Z";"Y"};
\endxy
$$
where $-\overset{N(K)}\times -$ means the quotient by the diagonal action of $N(K)$. (In the above diagram the left arrow
is the projection on the second factor, and the right arrow is given by the action of $N(\BA)$ on $G(\BA)/G(\BO)$.)

\medskip

However, it is not clear how to implement this idea in algebraic geometry, i.e., how to realize such a 
groupoid as a prestack. Instead we use a certain surrogate, whose idea
is explained in \secref{ss:idea of groupoid}. Here we will just mention that it relies on the phenomenon of 
\emph{strong approximation} for the group $N$. 

\sssec{Whittaker category via the affine Grassmannian}

One could also approach the definition of $\Whit(G)$ slightly differently (and the same applies to 
the degenerate, extended and parabolic versions). Namely, instead of the prestack $\CQ_G$, we can realize our category as a full
subcategory of $\Dmod(\Gr_{G,\Ran(X)})$, where $\Gr_{G,\Ran(X)}$ is the
Ran version of the affine Grassmannian. The point is that $\CQ_G$ is isomorphic to the quotient of $\Gr_{G,\Ran(X)}$
by the action of the group-prestack $\bMaps(X,N)^{\on{gen}}$ of generically defined maps $X\to N$. 

\medskip

This way of realizing
$\CQ_G$ gives rise to a more straightforward way of imposing the equivariance condition needed for the definition
of $\Whit(G)$. \footnote{The price that one has to pay if one uses only this approach in the case of
the parabolic category $\on{I}(G,P)$ is that the definition of the functors of \emph{enhanced Eisenstein series} and 
\emph{constant term} becomes more cumbersome.}  

\medskip

In any 
case, having this other approach to $\Whit(G)$ (and its degenerate, extended and parabolic versions) is necessary in order
to prove its description in spectral terms. 

\sssec{Spectral description of the Whittaker category}  \hfill

\medskip

The spectral description of $\Whit(G)$, given by Quasi-Theorem \ref{t:Cass Shal}, says that it is equivalent to the \emph{unital} version
of the category $\Rep(\cG)_{\Ran(X)}$, denoted $\Rep(\cG)^{\on{unital}}_{\Ran(X)}$. 
This is the first of the quasi-theorems in this paper, and it is expected to follow rather easily from the already known results. 

\medskip

We note that Quasi-Theorem \ref{t:Cass Shal} 
is a geometric version of the Casselman-Shalika formula that describes the unramified Whittaker model in terms of
Satake parameters. 

\ssec{The parabolic category}

In \secref{s:parabolic} we discuss the parabolic category, denoted $\on{I}(G,P)$, for a given parabolic subgroup
$P\subset G$ with Levi quotient $M$. 

\sssec{The idea of the parabolic category}

In terms of the function-theoretic analogy, the category $\on{I}(G,P)$ corresponds to the space of functions on the double quotient
$$M(K)\cdot N(P)(\BA)\backslash G(\BA)/G(\BO).$$

The actual defintion of $\on{I}(G,P)$ uses the prestack $\Bun_G^{P\on{-gen}}$ of $G$-bundles equipped with a generic
reduction to $P$. We define $\on{I}(G,P)$ to be the category of D-modules on $\Bun_G^{P\on{-gen}}$ that are equivariant
with respect to the appropriately defined groupoid. The idea of this groupoid, denoted $\bNP$, is similar to that of 
$\bN$, mentioned in \secref{sss:idea Whit}. As in the case of $\Whit(G)$, we can alternatively define $\on{I}(G,P)$
using the Ran version of the affine Grassmannian. 

\medskip

The forgetful functor 
$$\on{I}(G,P)\to \Dmod(\Bun_G^{P\on{-gen}})$$
is actually fully faithful due to a unipotence property of the groupoid $\bNP$.

\medskip

We note that in addition to the prestack $\Bun_G^{P\on{-gen}}$, we have the usual algebraic stack $\Bun_P$ classifying 
$P$-bundles on $X$. There exists a naturally defined map
$$\imath_P:\Bun_P\to \Bun_G^{P\on{-gen}},$$
which defines a bijection at the level of $k$-points. We can think of $\Bun_G^{P\on{-gen}}$ as decomposed into locally
closed sub-prestacks, with $\Bun_P$ being the disjoint union of the strata. Accordingly, we have a conservative
restriction functor
$$\imath_P^\dagger:\Dmod(\Bun_G^{P\on{-gen}})\to \Dmod(\Bun_P),$$
and we can think of $\Dmod(\Bun_G^{P\on{-gen}})$ as glued from $\Dmod(\Bun_P)$ on the various connected components
of $\Bun_P$ in a highly non-trivial way. For the above restriction functor we have a commutative diagram
\begin{equation} \label{e:parabolic and naive}
\CD
\on{I}(G,P)   @>>>  \Dmod(\Bun_M) \\
@VVV    @VVV    \\
\Dmod(\Bun_G^{P\on{-gen}})   @>{\imath_P^\dagger}>>  \Dmod(\Bun_P),
\endCD
\end{equation}
with fully faithful vertical arrows, which is moreover a \emph{pull-back diagram}. Here the right vertical arrow is the functor of 
pull-back along the projection $\sfq_P:\Bun_P\to \Bun_M$; it is fully faithful, since the map $\sfq_P$ is smooth with contractible
fibers. 

\sssec{Eisenstein and constant term functors}

The category $\on{I}(G,P)$ is related to the category $\Dmod(\Bun_G)$ by a pair of adjoint functors
$$\Eis^{\on{enh}}_P:\on{I}(G,P)\rightleftarrows \Dmod(\Bun_G):\on{CT}^{\on{enh}}_P,$$
that we refer to as \emph{enhanced Eisenstein series and constant term} functors.

\medskip

These functors are closely related (but carry significantly more information) than the corresponding ``usual"
Eisenstein and constant term functors
$$\Eis_P:\on{I}(G,P)\rightleftarrows \Dmod(\Bun_G):\on{CT}_P,$$
defined by pull-push along the diagram
$$
\xy
(-15,0)*+{\Bun_G}="X";
(15,0)*+{\Bun_M.}="Y";
(0,15)*+{\Bun_P}="Z";
{\ar@{->}_{\sfp_P} "Z";"X"};
{\ar@{->}^{\sfq_P}  "Z";"Y"};
\endxy
$$

For example, the functor $\on{CT}_P$ is the composition of the functor $\on{CT}^{\on{enh}}_P$, followed by
the restriction functor (the top horizontal arrow in the diagram \eqref{e:parabolic and naive}).

\sssec{Parabolic category on the spectral side}

We now discuss the spectral counterpart of the above picture. For the ``usual" Eisenstein series functor, the picture
is what one would naively expect. We consider the diagram
$$
\xy
(-15,0)*+{\LocSys_\cG}="X";
(15,0)*+{\LocSys_\cM.}="Y";
(0,15)*+{\LocSys_\cP}="Z";
{\ar@{->}_{\sfp_{\cP,\on{spec}}} "Z";"X"};
{\ar@{->}^{\sfq_{\cP,\on{spec}}}  "Z";"Y"};
\endxy
$$
and the corresponding pull-push functor
$$\Eis_{\cP,\on{spec}}:\IndCoh_{\on{Nilp}^{\on{glob}}_\cM}(\LocSys_\cM)\to \IndCoh_{\on{Nilp}^{\on{glob}}_\cG}(\LocSys_\cG).$$

The geometric Langlands equivalence is supposed to make the following diagram commute:
$$
\CD
\IndCoh_{\on{Nilp}^{\on{glob}}_\cG}(\LocSys_\cG)  @>{\BL_G}>>  \Dmod(\Bun_G)  \\
@A{\Eis_{\cP,\on{spec}}}AA    @AA{\Eis_P}A   \\
\IndCoh_{\on{Nilp}^{\on{glob}}_\cM}(\LocSys_\cM)  @>{\BL_M}>>  \Dmod(\Bun_M), 
\endCD
$$
up to a twist by a (specific) line bundle on $\LocSys_\cM$.

\medskip

The situation with enhanced Eisenstein series is more involved and more interesting. First, we need to give a spectral
description of the category $\on{I}(G,P)$. As is natural to expect, for the latter we need to assume the validity of the
Langlands conjecture for the group $M$. 

\medskip

First, we consider the appropriate modification of $\QCoh(\LocSys_\cP)$ given by the singular support condition.
We denote the resulting category by $\IndCoh_{\on{Nilp}^{\on{glob}}_\cP}(\LocSys_\cP)$. Next, we consider the 
map 
$$\sfp_{\cP,\on{spec}}:\LocSys_\cP\to \LocSys_\cG,$$
and we consider the category of objects of $\IndCoh_{\on{Nilp}^{\on{glob}}_\cP}(\LocSys_\cP)$, endowed with a 
\emph{right action of vector fields on $\LocSys_\cP$ along the (derived) fibers of the map $\sfp_{\cP,\on{spec}}$.}
We denote this category by
$$\sF_\cP\mod(\IndCoh_{\on{Nilp}^{\on{glob}}_\cP}(\LocSys_\cP))$$
(the notation $\sF_\cP$ stands for the monad induced by the action of the vector fields mentioned above). 
We have a naturally defined functor, denoted, and given by pushforward:
$$\Eis^{\on{enh}}_{P,\on{spec}}:\sF_\cP\mod(\IndCoh_{\on{Nilp}^{\on{glob}}_\cP}(\LocSys_\cP))\to 
\IndCoh_{\on{Nilp}^{\on{glob}}_\cG}(\LocSys_\cG).$$

\sssec{Geometric Langlands equivalence for parabolic categories}  \hfill

\medskip

One of the central quasi-theorems in this paper, namely Quasi-Theorem \ref{t:principal series}, says that we have 
a canonically defined equivalence:
$$\BL_P:\sF_\cP\mod(\IndCoh_{\on{Nilp}^{\on{glob}}_\cP}(\LocSys_\cP))\to \on{I}(G,P).$$

\medskip

It is supposed to be related to the geometric Langlands equivalence for $G$ via the following 
commutative diagram
$$
\CD
\IndCoh_{\on{Nilp}^{\on{glob}}_\cG}(\LocSys_\cG)  @>{\BL_G}>>  \Dmod(\Bun_G)  \\
@A{\Eis^{\on{enh}}_{\cP,\on{spec}}}AA    @AA{\Eis^{\on{enh}}_P}A   \\
\sF_\cP\mod(\IndCoh_{\on{Nilp}^{\on{glob}}_\cP}(\LocSys_\cP))  @>{\BL_P}>>  \on{I}(G,P).
\endCD
$$

\ssec{Degenerate and extended Whittaker categories}

\sssec{}

\secref{s:deg Whit} deals with the degenerate Whittaker category, denoted $\Whit(G,P)$, and 
\secref{s:ext Whit} with the extended Whittaker category, denoted $\Whit^{\on{ext}}(G,G)$. The function-theoretic
analogues of these categories were explained in \secref{sss:explain ext Whit}.

\medskip

The relevance of the categories $\Whit(G,P)$ (as $P$ runs through the set of conjugacy classes
of parabolics) is that they constitute buliding blocks of the category $\Whit^{\on{ext}}(G,G)$. The 
category $\Whit^{\on{ext}}(G,G)$ plays a crucial role, being the recipient of the functor
$$\on{coeff}(G,G)^{\on{ext}}:\Dmod(\Bun_G)\to \Whit^{\on{ext}}(G,G).$$

As was mentioned earlier, a crucial conjecture (which is a quasi-theorem for $GL_n$) says that 
the functor $\on{coeff}(G,G)^{\on{ext}}$ is fully faithful. This statement is at the heart of our 
approach to proving the geometric Langlands equivalence.

\medskip

We omit a detailed discussion of the contents of these two sections
as the ideas involved essentially combine those from Sects. \ref{s:Whit} and \ref{s:parabolic}, except 
for the following:

\medskip

Recall the set $\fch(K)$ mentioned in \secref{sss:explain ext Whit}. The definition as given in {\it loc.cit.}
is correct only when $G$ has a connected center. In general, the definition needs to be modified,
see \secref{ss:ch}, and involves a certain canonically defined toric variety acted on by $T$, such
that the stabilizer of each point is the \emph {connected} center of the corresponding Levi subgroup.

\sssec{}

The goal of \secref{s:gluing} is to provide a spectral description of the category $\Whit^{\on{ext}}(G,G)$.
As was explained in \secref{ss:strategy}, this is another crucial step in our approach to proving the
geometric Langlands conjecture. 

\medskip

First, we recall the general pattern of gluing of DG categories, mimicking the procedure of describing
the category of sheaves on a topological space from the knowledge of the corresponding categories
on strata of a given stratification.

\medskip

We observe that, more or less tautologically, the category $\Whit^{\on{ext}}(G,G)$ is glued from the
categories $\Whit(G,P)$. 

\medskip

Next, we explicitly construct the glued category on the spectral side, by taking as building blocks the
categories  
$$\sF_\cP\mod(\QCoh(\LocSys_\cP)).$$

We now arrive at a crucial assertion, Quasi-Theorem \ref{t:glued equiv} that says that the glued category
on the spectral side embeds fully faithfully into $\Whit^{\on{ext}}(G,G)$. 

\medskip

So far, this Quasi-Theorem has
been verified in a particular case (assuming Quasi-Theorem \ref{t:principal series} for $P=B$),
when we want to glue the open stratum (corresponding to $P=G$)
to the closed stratum (corresponding to $P=B$); this case, however, suffices for the group $G=GL_2$. 
The proof of Quasi-Theorem \ref{t:glued equiv} in the above case is a rather illuminating explicit calculation,
which we unfortunately have to omit for reasons of length of this paper. 

\ssec{Kac-Moody localization}

\sssec{}

\secref{s:opers} deals with a construction of objects of $\Dmod(\Bun_G)$ of a nature totally different from
one discussed in Sects. \ref{s:Whit}-\ref{s:gluing}.  

\medskip

The previous sections approach D-modules on $\Bun_G$
geometrically, i.e., by considering various spaces that map to $\Bun_G$ and appying functors of direct
and inverse image. In particular, these constructions make sense not just in the category of D-modules,
but also in that of $\ell$-adic sheaves (modulo the technical issue of the existence of the formalism of $\ell$-adic 
sheaves as a functor of $\infty$-categories).

\medskip

By contrast,  in \secref{s:opers} we construct D-modules on $\Bun_G$ ``by generators and relations." In particular,
we (implicitly) use the forgetful functor $\Dmod(\Bun_G)\to \QCoh(\Bun_G)$ (or, rather, its left adjoint). 
More precisely, the construction that
we use is that of \emph{localization} of modules over the Kac-Moody algebra (at a given level).

\medskip

This construction is needed in order to create the commutative diagram \eqref{e:pentagon}, which is another
crucial ingredient in the proof of the geometric Langands conjecture. 

\sssec{}

Historically, the pattern of localization originated form \cite{BB}. In \cite{BD2} it was extended to the following
situation: if we have a group $H$ acting on a scheme $Y$, and $H'\subset H$ is a subgroup, then we have
a canonical functor of localization
$$(\fh,H')\mod\to \Dmod(H'\backslash Y),$$
where $(\fh,H')\mod$ is the DG category of $H'$-equivariant objects in the DG category $\fh\mod$ 
of $\fh$-modules (also known as the DG category of modules over the Harish-Chandra pair $(\fh,H')$).

\medskip

If one looks at what this construction does in down-to-earth terms, it associates to a $(\fh,H')$-module
a certain quotient of the free D-module, where relations are given by the action of vector fields in $Y$
induced by the action of elements of $\fh$. 

\sssec{}

In \cite{BD2}, this construction was applied to $H$ being the (critical central extension of the) loop group ind-scheme
$\fL(G)=G\ppart$, and $H'$ being the group of arcs $\fL^+(G):=G\qqart$. The corresponding category of
Harish-Chandra modules is denoted $\on{KL}(G,\on{crit})$. 
\footnote{``KL" stands for Kazhdan-Lusztig, who were the first to systematically study this category in the 
negative level case.}

\medskip

The scheme $Y$ in question is
$\Bun_{G,x}$, the moduli space of $G$-bundles on $X$ with a full level structure at a point $x$. Here we
think of $k\qqart$ as the completed local ring of $X$ at $x$. 
We do not review this construction in this paper, but rather refer the reader to \cite{BD2}; we should note, however, 
that modern technology allows to rewrite this construction in a more concise way.

\medskip

In fact, we need an extension of the above construction to the situation, when instead of a fixed point $x\in X$
we have a finite number of points that are allowed to move along $X$. Ultimately, we obtain a functor
$$\on{Loc}_{G,\Ran(X)}:\on{KL}(G,\on{crit})_{\Ran(X)}\to \Dmod(\Bun_G).$$

A  crucial property of the functor $\on{Loc}_G$ is that ``almost all D-modules on $\Bun_G$ lie in its essential image."
The word \emph{almost} is important here. We refer the reader to \propref{p:loc surj} for a precise formulation.

\medskip

We also remark that the functor $\on{Loc}_G$ should be thought of as a \emph{quantum} case of the functor
$$\on{Loc}_{G,\on{spec}}:\Rep(G)_{\Ran(X)}\to \QCoh(\LocSys_G),$$
mentioned earlier. In fact, the two are the special cases of a family whose intermediate values correspond
to the situation of quantum geometric Langlands. 

\sssec{}

In the rest of \secref{s:opers} we review the connection between the category $\on{KL}(G,\on{crit})_{\Ran(X)}$
and the scheme of \emph{local opers}. The key input is a generalization of the result of \cite{BD2} that relates the 
functor $\on{Loc}_G$ to the scheme of \emph{global opers}. All of this is needed in order to form the diagram 
\eqref{e:pentagon}. 

\sssec{}

Finally, in \secref{s:proof}, we assemble the ingredients developed in the previous sections in order to
prove the geometric Langlands conjecture, modulo Conjectures \ref{c:Whit ext ff} and \ref{c:Op gen},
and the Quasi-Theorems.

\medskip

The proof proceeds along the lines indicated in \secref{ss:strategy}, modulo the fact that the last step of the proof,
namely, one described in \secref{sss:comparing images}, is a bit of an oversimplification. For the actual proof,
we break the category $\Dmod(\Bun_G)$ into ``cuspidal" and "Eisenstein" parts, and deal with each separately.  

\section{The theory of singular support}  \label{s:sing supp}

\ssec{Derived locally complete intersections}

The contents of this subsection are a brief review of \cite[Sect. 2]{AG}. We refer the reader to {\it loc.cit.}
for the proofs. 

\medskip

We remind that throughout the paper we will be working with an algebraically closed field $k$ of characteristic $0$. 

\sssec{}

The theory of singular support for coherent sheaves makes substantial use of
derived algebraic geometry. We cannot afford to make a thorough review here,
but let us mention the following few facts, which is all we will need for this paper:

\medskip

\noindent(1) Let $A$ be a CDGA (commutative differential graded algebra) over
$k$, which lives in cohomological degrees $\leq 0$. To $A$ one attaches
the affine DG scheme $\Spec(A)$. If $A\to A'$ is a quasi-isomorphism, then
the corresponding map $\Spec(A')\to \Spec(A)$ is, by definition, an isomorphism
of DG schemes. The underlying topological space of $\Spec(A)$
is the same as that of the classical scheme $\Spec(H^0(A))$. The basic affine
opens of $\Spec(A)$ are of the form $\Spec(A_f)$, where $f\in H^0(A)$
(more generally, it makes sense to take localizations of $A$ with respect to multiplicative
subsets of $H^0(A)$). 

\medskip

\noindent(1') Arbitrary DG schemes are glued from affines in the same sense
as in classical algebraic geometry. 

\medskip

\noindent(2) There exists a fully faithful functor $\Sch\to \dgSch$ from
classical schemes to derived schemes. This functor admits a right adjoint,
which we will refer to as \emph{taking the underlying classical scheme}
and denote by $Y\mapsto {}^{cl}Y$. 
For affine DG schemes the latter functor corresponds to sending $A$ to
$H^0(A)$. In general, it is convenient to have the following analogy in mind
``classical schemes to derived schemes are what reduced classical schemes
are to all schemes."

\medskip

\noindent(3) The DG category of quasi-coherent sheaves on a DG scheme is defined 
so that $$\QCoh(\Spec(A))=A\mod,$$
the latter being the DG category of all $A$-modules 
(i.e., no finiteness assumptions). 

\medskip

\noindent(4) The category of DG schemes admits fiber products: for 
$\Spec(A_1)\to \Spec(A)\leftarrow \Spec(A_2)$, we have
$$\Spec(A_1)\underset{\Spec(A)}\times \Spec(A_2)=\Spec(A_1\underset{A}\otimes A_2),$$
where the tensor product $A_1\underset{A}\otimes A_2$ is understood in the derived sense
(in particular, $A_1\underset{A}\otimes A_2$ may be derived even if $A$, $A_1$ and $A_2$
are classical). 

\medskip

\noindent(4') A basic non-trivial example of a DG scheme is 
$$\on{pt}\underset{V}\times \on{pt},$$
where $V$ is a finite-dimensional vector space (considered as a scheme). The above DG scheme is by definition
$\Spec(\on{Sym}(V^*[1]))$.  Here $\on{pt}:=\Spec(k)$. 

\medskip

\noindent(5) Let 
$$
\CD
Y'_1  @>{g_1}>>  Y_1 \\
@V{f'}VV   @VV{f}V  \\
Y'_2  @>{g_2}>>  Y_2
\endCD
$$
be a Cartesian square of DG schemes with the vertical morphisms quasi-compact
and quasi-separated. Then the base change natural transformation
$$g_2^*\circ f_*\to f'_*\circ g_1^*$$
is an isomorphism. (Note that the corresponding fact is \emph{false} in classical
algebraic geometry: i.e., even if $Y_1$, $Y_2$ and $Y'_2$ are classical, we
need to understand $Y'_1$ is the derived sense.)

\medskip

\noindent(6) One word of warning is necessary: the category $\dgSch$ is not an ordinary
category, but an $\infty$-category, i.e., maps between objects no longer form sets,
but rather $\infty$-groupoids (in the various models of the theory of $\infty$-categories
the latter can be realized as simplicial sets, topological spaces, etc.).

\sssec{}

We shall now define what it means for a DG scheme $Y$ to be a \emph{derived locally complete intersection},
a.k.a. \emph{quasi-smooth}. 

\medskip

The condition is Zariski-local, so we can assume that $Y$ is affine. 

\begin{defn}
We shall say that $Y$ is quasi-smooth if it can be realized as a derived fiber product
\begin{equation} \label{e:q-smooth as fiber product}
\CD
Y @>>>  \CU  \\
@VVV   @VV{f}V   \\
\on{pt}  @>{v}>>  \CV,
\endCD
\end{equation}
where $\CU$ and $\CV$ are smooth classical schemes. 
\end{defn}

More invariantly, one can phrase this definition as follows:

\begin{defn} A DG scheme $Y$ is quasi-smooth if 
it is locally \emph{almost of finite type}\footnote{This means that the underlying classical scheme $^{cl}Y$
is locally of finite type over $k$, and the cohomology sheaves $H^i(\CO_Y)$ are finitely generated over
$H^0(\CO_Y)=\CO_{^{cl}\!Y}$.} and for each $k$-point $y\in Y$, the derived cotangent space
$T^*_y(Y)$ has cohomologies only in degrees $0$ and $-1$. 
\end{defn}

In fact, for $Y$ written as in \eqref{e:q-smooth as fiber product}, the derived cotangent space at $y\in Y$
is canonically isomorphic to the complex
$$T^*_{f(y)}(\CV)\to T^*_y(\CU).$$

\sssec{}

It follows easily from the definitions that a classical scheme which is a locally complete intersection in the
classical sense is such in the derived sense, i.e., quasi-smooth as a derived scheme. 

\ssec{The $\Sing$ space of a quasi-smooth scheme}

\sssec{}

Let $Y$ be a quasi-smooth derived scheme. We are going to attach to it a classical scheme
$\Sing(Y)$ that measures the extent to which $Y$ fails to be smooth.

\medskip

Suppose that $Y$ is locally written as a fiber product \eqref{e:q-smooth as fiber product}. 
Consider the vector bundles $T^*(\CU)|_{^{cl}Y}$ and $T^*(\CV)|_{^{cl}Y}$, considered as 
schemes over $^{cl}Y$. 

\medskip

The differential of $f$ defines a map of classical schemes 
\begin{equation} \label{e:co-differential}
T^*(\CV)|_{^{cl}Y}\to T^*(\CU)|_{^{cl}Y}.
\end{equation}

\medskip

We let $\Sing(Y)$ be the pre-image under the map \eqref{e:co-differential} of the 
zero-section $^{cl}Y\to T^*(\CU)|_{^{cl}Y}$.

\medskip

The scheme $\Sing(Y)$ carries a natural action of the group $\BG_m$ inherited
from one on $T^*(\CV)|_{^{cl}Y}$. 

\sssec{}

Explicitly, one can describe $k$-points if $\Sing(Y)$ as follows. These are pairs 
$(y,\xi)$, where $y$ is a $k$-point of $Y$, and $\xi$ is an element in
$$\on{ker}\left(df:T^*_{v}(\CV)\to T^*_y(\CU)\right).$$

In particular, $f$ is smooth (which is equivalent to $Y$ being a smooth classical scheme) 
if and only if the projection $\Sing(Y)\to Y$ is an isomorphism, i.e., if $\Sing(Y)$ 
consists of the zero-section. 

\sssec{}   \label{sss:Sing via cotang}

More invariantly, one can think of $\xi$ as an element in the vector space
$$H^{-1}(T^*_y(Y)).$$

This implies that  $\Sing(Y)$ is well-defined in the sense that it is independent 
of the presentation of $Y$ as a fiber product as in \eqref{e:q-smooth as fiber product}. 
In particular, we can define $\Sing(Y)$ for $Y$ not necessarily affine.

\ssec{Cohomological operations}

\sssec{}

Let $Y$ be a quasi-smooth DG scheme written as  in \eqref{e:q-smooth as fiber product}. 
Let us denote by $V$ the tangent space of $\CV$ at the point $v$, and let $V^*$ be its
dual, i.e., the cotangent space. 

\medskip

We claim that for every $\CF\in \QCoh(Y)$ there is a canonically define map
of graded algebras
\begin{equation} \label{e:map of exts}
\on{Sym}(V)\to \underset{i}\oplus\, \Hom_{\QCoh(Y)}(\CF,\CF[i]),
\end{equation}
where we set $\on{deg}(V)=2$. 

\medskip

We shall define \eqref{e:map of exts} in the framework of the following geometric construction.

\sssec{}

First, we consider the derived fiber product
$$\on{pt}\underset{\CV}\times \on{pt}.$$

This is a groupoid over $\on{pt}$, i.e., a derived group-scheme (a group object in the category
of derived schemes). 

\medskip

In particular, the category
$$\QCoh(\on{pt}\underset{\CV}\times \on{pt})$$
acquires a monoidal structure given by convolution.

\medskip

The unit in this category is $k_{\on{pt}}$, the direct image of $k\in \QCoh(\on{pt})$ 
under the diagonal morphism
$$\on{pt}\to \on{pt}\underset{\CV}\times \on{pt}.$$

\sssec{}

We claim that the derived group-scheme $\on{pt}\underset{\CV}\times \on{pt}$ canonically acts
on $Y$. This follows from the next diagram
$$
\xy
(-25,0)*+{Y}="X";
(25,0)*+{Y}="Y";
(0,15)*+{Y\underset{\CU}\times Y}="Z";
(-25,-20)*+{\on{pt}}="X_1";
(25,-20)*+{\on{pt}.}="Y_1";
(0,-5)*+{\on{pt}\underset{\CV}\times \on{pt}}="Z_1";
{\ar@{->}"Z";"X"};
{\ar@{->}"Z";"Y"};
{\ar@{->}"Z_1";"X_1"};
{\ar@{->}"Z_1";"Y_1"};
{\ar@{->}"Z";"Z_1"};
{\ar@{->}"X";"X_1"};
{\ar@{->}"Y";"Y_1"};
\endxy
$$
in which both squares are Cartesian. 

\sssec{}

In particular, we obtain that the category $\QCoh(Y)$ acquires an action of the 
monoidal category $\QCoh(\on{pt}\underset{\CV}\times \on{pt})$.

\medskip

Hence, every $\CF\in \QCoh(Y)$ acquires an action of the algebra of the endomorphisms of the unit
object of $\QCoh(\on{pt}\underset{\CV}\times \on{pt})$, i.e., 
we have a canonical map of graded algebras
$$\underset{i}\oplus\, \Hom_{\QCoh(\on{pt}\underset{\CV}\times \on{pt})}(k_{\on{pt}},k_{\on{pt}}[i])\to 
\underset{i}\oplus\, \Hom_{}(\CF,\CF[i]).$$
 
\sssec{}

Finally, to construct the map \eqref{e:map of exts} we notice we have a canonical isomorphism
of graded algebras
$$\on{Sym}(V)\to \underset{i}\oplus\, \Hom_{\QCoh(\on{pt}\underset{\CV}\times \on{pt})}(k_{\on{pt}},k_{\on{pt}}[i]).$$

\ssec{The singular support of a coherent sheaf}

The material of this subsection corresponds the approach to singular support in \cite[Sects. 5.3]{AG}. We refer
the reader to {\it loc.cit.} for the proofs of the statements quoted here. 

\sssec{}

We continue to assume that $Y$ is a quasi-smooth DG shceme written as  \eqref{e:q-smooth as fiber product}. 
Note that by construction, $\Sing(Y)$ is a conical Zariski-closed (=$\BG_m$-invariant) closed subset in
$^{cl}Y\times V^*$. 

\medskip

From \eqref{e:map of exts}, we obtain that for $\CF\in \QCoh(Y)$ we have a map of graded commutative algebras
\begin{equation} \label{e:map from extended Hoch}
\Gamma({}^{cl}Y\times V^*,\CO_{^{cl}Y\times V^*})\simeq 
\Gamma(Y,\CO_{^{cl}Y})\otimes \on{Sym}(V)\to  \underset{i}\oplus\, \Hom_{\QCoh(Y)}(\CF,\CF[i]).
\end{equation}

\medskip

We have the following assertion: 

\begin{lem} \label{l:sing supp matters}
Let $f\in  \Gamma({}^{cl}Y\times V^*,\CO_{^{cl}Y\times V^*})$ be a homogeneous element
that vanishes when restricted to $\Sing(Y)$. Then some power of $f$ belongs to the kernel
of the map \eqref{e:map from extended Hoch}.
\end{lem}

The above lemma allows to define the notion of singular support of coherent sheaves. 

\sssec{}

Let $\Coh(Y)\subset \QCoh(Y)$ be the full subcategory that consists of \emph{coherent} sheaves. I.e.,
these are objects that have only finitely many non-zero cohomologies, and such that each cohomology
is finitely generated over $\CO_{^{cl}Y}$.

\medskip

\begin{defn}  \label{defn:sing supp}
The singular support of $\CF\in \Coh(Y)$ is the conical 
Zariski-closed subset $$\on{sing.supp.}(\CF)\subset {}^{cl}Y\times V^*,$$ corresponding to the ideal,
given by the kernel of the map \eqref{e:map from extended Hoch}.
\end{defn} 

Note that by \lemref{l:sing supp matters}, we automatically have
$$\on{sing.supp.}(\CF)\subset \Sing(Y),$$
as Zariski-closed subsets of $^{cl}Y\times V^*$.

\sssec{}

Dually, given a conical Zarsiki-closed subset $\CN\subset \Sing(Y)$, we let
$$\Coh_{\CN}(Y)\subset \Coh(Y)$$
be the full subcategory, consisting of objects whose singular support is contained 
in $\CN$.

\medskip

By definition, for $\CF\in \Coh(Y)$, we have $\CF\in \Coh_{\CN}(Y)$ if and only if for every
homogeneous element $f\in \Gamma({}^{cl}Y\times V^*,\CO_{^{cl}Y\times V^*})$ such that
$f|_{\CN}=0$, some power of $f$ lies in the kernel of \eqref{e:map from extended Hoch}.

\sssec{}

We have the following assertion:

\begin{prop}
For $\CF\in \Coh(Y)$, the subset $\on{sing.supp.}(\CF)\subset \Sing(Y)$ is independent 
of the choice of presentation of $Y$ as in \eqref{e:q-smooth as fiber product}. 
\end{prop}

Thus, the notion of singular support of an object of $\Coh(Y)$ and the category 
$\Coh_{\CN}(Y)$ make sense for any quasi-smooth DG scheme (not necessarily affine).

\medskip

In addition, we have:

\begin{prop}  \label{p:perfect}
For $\CF\in \Coh(Y)$, its singular support is the zero-section $\{0\}\subset \Sing(Y)$
if and only if $\CF$ is \emph{perfect}.
\end{prop}

\ssec{Ind-coherent sheaves}

The material in this subsection is a summary of \cite[Sect. 1]{Ga3}.

\sssec{}  

Let $Y$ be a quasi-compact DG scheme almost of finite type. We consider the DG category $\IndCoh(Y)$ to be
the ind-completion of $\Coh(Y)$.

\medskip

I.e., this is a cocomplete DG category, equipped with a functor $\Coh(Y)\to \IndCoh(Y)$,
which is universal in the following sense: for a cocomplete DG category $\bC$, a functor 
$$\Coh(Y)\to \bC$$ uniquely extends to a \emph{continuous} functor 
$\IndCoh(Y)\to \bC$. 

\medskip

One shows that the functor $\Coh(Y)\to \IndCoh(Y)$ is fully faithful and that its essential image
compactly generates $\IndCoh(Y)$.

\medskip

By the universal property of $\IndCoh(Y)$, the tautological embedding $\Coh(Y)\hookrightarrow \QCoh(Y)$
canonically extends to a continuous functor
\begin{equation} \label{e:functor Psi}
\Psi_Y:\IndCoh(Y)\to \QCoh(Y). 
\end{equation}
Note however, that the functor \eqref{e:functor Psi} is \emph{no longer} fully faithful.

\medskip

Another crucial piece of structure on $\IndCoh(Y)$ is that we have a canonical action of
$\QCoh(Y)$, regarded as a monoidal category, on $\IndCoh(Y)$. It is obtained by ind-extending
the action of $\QCoh(Y)^{\on{perf}}$ on $\Coh(Y)$ be tensor products. 

\sssec{}

Suppose now that $Y$ is eventually coconnective, which means that its structure
sheaf has finitely many non-zero cohomologies. For example, any quasi-smooth DG
scheme has this property.

\medskip

In this case we have an inclusion $\QCoh(Y)^{\on{perf}}\subset \Coh(Y)$ as full subcategories
of $\QCoh(Y)$. 

\medskip

By the functoriality of the construction of forming the ind-completion, we have a naturally
defined functor
\begin{equation} \label{e:pre functor Xi}
\on{Ind}(\QCoh(Y)^{\on{perf}})\to \IndCoh(Y).
\end{equation}

Note, however, that by the Thomason-Trobaugh theorem (see, e.g., \cite{Ne}),
the natural functor
$$\on{Ind}(\QCoh(Y)^{\on{perf}})\to \QCoh(Y)$$
is an equivalence.

\medskip

Hence, from \eqref{e:pre functor Xi} we obtain a functor
\begin{equation} \label{e:functor Xi}
\Xi_Y:\QCoh(Y)\to \IndCoh(Y). 
\end{equation}

It follows from the construction, that the functor \eqref{e:functor Xi} is fully faithful
and provides a left adjoint of the functor \eqref{e:functor Psi}. 

\medskip

Thus, we obtain that $\QCoh(Y)$ is a \emph{co-localization} of $\IndCoh(Y)$. I.e., 
$\IndCoh(Y)$ is a ``refinenement" of $\QCoh(Y)$. 

\medskip

Of course, if $Y$ is a smooth classical scheme, there is no difference between  
$\QCoh(Y)^{\on{perf}}$ and $\Coh(Y)$, and the functors \eqref{e:functor Psi}
and \eqref{e:functor Xi} are mutually inverse equivalences. 

\ssec{Ind-coherent sheaves with prescribed support}

The material of this subsection corresponds to \cite[Sect. 4.1-4.3]{AG}. 

\sssec{}  \label{sss:varying N}

Assume now that $Y$ is quasi-smooth. In a similar way to the definition of $\IndCoh(Y)$, 
starting from $\Coh_{\CN}(Y)$, we construct the category $\IndCoh_{\CN}(Y)$.

\medskip

As in the case of $\IndCoh(Y)$, we have a canonical monoidal action of $\QCoh(Y)$ 
on $\IndCoh_{\CN}(Y)$. 

\medskip

We recover all of $\IndCoh(Y)$ by setting $\CN=\Sing(Y)$. 

\sssec{}   \label{sss:varying N bis}

Note that by \propref{p:perfect}, for $\CN$ being the zero-section $\{0\}\subset Y$ we have
$$\Coh_{\{0\}}(Y)=\QCoh(Y)^{\on{perf}},$$ so
$$\IndCoh_{\{0\}}(Y)\simeq \on{Ind}(\QCoh(Y)^{\on{perf}})\simeq \QCoh(Y),$$
and we have tautologically defined fully faithful functors
$$\QCoh(Y)\simeq \IndCoh_{\{0\}}(Y)\hookrightarrow \IndCoh_\CN(Y)\hookrightarrow \IndCoh(Y),$$
whose composition is the functor \eqref{e:functor Xi}. 

\medskip

We shall denote the above functor $\QCoh(Y)\to \IndCoh_\CN(Y)$ by $\Xi_{Y,\CN}$, and its right adjoint
(tautologically given by $\Psi_Y|_{\IndCoh_\CN(Y)}$) by $\Psi_{Y,\CN}$. 

\sssec{}

The category $\IndCoh_\CN(Y)$ will be the principal actor on the spectral side of the geometric
Langlands conjecture.

\ssec{$\QCoh$ and $\IndCoh$ on stacks}

This subsection makes a brief review of the material of \cite[Sects. 1.1, 1.2, 2.1 and 5.1]{QCoh} and \cite[Sects. 11]{Ga3} 
relevant for this paper. 

\sssec{}  \label{sss:prestack}

For later use we give the following definition. Let $\CY$ be a prestack, i.e., an arbitrary
functor
$$(\affdgSch)^{\on{op}}\to \inftygroup.$$

We define the category $\QCoh(\CY)$ as
$$\underset{S\to \CY}{\underset{\longleftarrow}{lim}}\, \QCoh(S),$$
where the inverse limit is taken over the category of affine DG schemes over $\CY$.

\medskip

I.e., informally, an object $\CF\in \IndCoh_\CN(\CY)$ is an assignment for every map $S\to \CY$ of an object
$$\CF_S\in  \QCoh(S),$$
and for map $f:S_1\to S_2$ over $\CY$ of an isomorphism
$$f^*(\CF_{S_2})\simeq \CF_{S_1},$$
where these isomorphisms must be equipped with a data of homotopy-coherence for higher
order compositions. 

\medskip

For a map of prestacks $f:\CY_1\to \CY_2$ we have a tautologically defined functor 
$$f^*:\QCoh(\CY_2)\to \QCoh(\CY_1).$$

\medskip

If $f$ is schematic quasi-compact and quasi-separated 
(i.e., its base change by a DG scheme yields a quasi-compact and quasi-separated DG scheme), the functor $f^*$
admits a continuous right adjoint, denoted $f_*$. 

\sssec{}  \label{sss:classical prestacks}

A prestack $\CY$ is said to be \emph{classical} if in the category $\affdgSch_{/\CY}$
of affine DG schemes mapping to $\CY$ the full subcategory $\affSch_{/\CY}$ is cofinal. 
I.e., if any map $S\to \CY$, where $S\in \affdgSch$ can be factored as
$$S\to S'\to \CY,$$
where $S'$ is classical, and the category of such factorizations is \emph{contractible}.

\medskip

If $\CY$ is classical, then the category $\QCoh(\CY)$ can be recovered just from the
knowledge of $\QCoh(S)$ for classical schemes $S$ over $\CY$. Precisely,
the restriction functor
$$\QCoh(\CY):=\underset{S\in (\affdgSch_{/\CY})^{\on{op}}}{\underset{\longleftarrow}{lim}}\, \QCoh(S)\to
\underset{S\in (\affSch_{/\CY})^{\on{op}}}{\underset{\longleftarrow}{lim}}\, \QCoh(S)$$
is an equivalence. 

\sssec{}

Let now $\CY$ be a (derived) algebraic stack (see \cite[Sect. 1.1]{DrGa1} for our conventions regarding algebraic stacks).
In this case, one can rewrite the definition of $\QCoh(\CY)$ as follows:

\medskip

Instead of taking the limit over the category of all affine DG schemes over $\CY$, we can replace
it by a full subcategory 
$$(\affdgSch)_{/\CY,\on{smooth}},$$
where we restrict objects to those $y:S\to \CY$ for which the map $y$ is smooth,
and morphisms to those maps $f:S_1\to S_2$ over $\CY$ for which $f$ is smooth. 

\sssec{}  \label{sss:indcoh on stack}

Suppose that $\CY$ is a (derived) algebraic stack \emph{locally almost of finite type} (i.e., it admits a smooth atlas 
consisting of DG schemes that are almost of finite type). In this case one can define $\IndCoh(\CY)$
as
$$\underset{S\in ((\affdgSch)_{/\CY,\on{smooth}})^{\on{op}}}{\underset{\longleftarrow}{lim}}\, \IndCoh(\CY).$$

\medskip

Informally, an object $\CF\in \IndCoh(\CY)$ is an assignment for every smooth map $S\to \CY$ of an object
$$\CF_S\in  \IndCoh(S),$$
and for every smooth map $f:S_1\to S_2$ over $\CY$ of an isomorphism
$$f^*(\CF_{S_2})\simeq \CF_{S_1},$$
where these isomorphisms must be equipped with a data of homotopy-coherence for higher
order compositions. 

\sssec{}  \label{sss:! pull-back}

If $f:\CY_1\to \CY_2$ is a schematic quasi-compact map of algebraic DG stacks (both assumed 
locally almost of finite type), we have a naturally defined continuous pushforward functor
$$f^{\IndCoh}_*:\IndCoh(\CY_1)\to \IndCoh(\CY_2).$$

In addition, if $f$ is an \emph{arbitrary} map between algebraic DG stacks, there exists a well-defined functor
$$f^!:\IndCoh(\CY_2)\to \IndCoh(\CY_1).$$

\medskip

The functor $f^!$ is the \emph{right adjoint} of $f^{\IndCoh}_*$ if $f$ is schematic and proper, and is the 
\emph{left adjoint} of $f^{\IndCoh}_*$ if $f$ is an open embedding.

\medskip

The functors of pushforward and !-pull-back satisfy a base change property: for a Cartesian square
of algebraic DG stacks almost of finite type
$$
\CD
\CY'_1  @>{g_1}>>  \CY_1  \\
@V{f'}VV    @VV{f}V    \\
\CY'_2  @>{g_2}>>  \CY_2,
\endCD
$$
with the vertical maps being schematic and quasi-compact, 
there is a \emph{canonically defined} isomorphism of functors
\begin{equation} \label{e:base cange}
g_2^!\circ f^{\IndCoh}_*\simeq (f')^{\IndCoh}_*\circ g_1^!.
\end{equation}

Note, however, that unless $f$ is proper or open, there is a priori no map in either
direction in \eqref{e:base cange}.

\medskip

Finally, if $f$ is locally of finite Tor-dimension, we also have a functor
$$f^{\IndCoh,*}:\IndCoh(\CY_2)\to \IndCoh(\CY_1).$$
If $f$ is schematic and quasi-compact then $f^{\IndCoh,*}$ is the left-adjoint of $f^\IndCoh_*$. 
If $f$ is smooth (or more generally, Gorenstein), then the functors $f^!$ and $f^{\IndCoh,*}$
differ by a twist by the relative dualizing line bundle. 

\ssec{Singular support on algebraic stacks}

The material of this subsection corresponds to \cite[Sect. 8]{AG}.

\sssec{}

Let $\CY$ be a (derived) algebraic stack. We shall say that $\CY$ is quasi-smooth if for any DG scheme 
and a smooth map $Y\to \CY$, the DG scheme $Y$ is quasi-smooth. 

\medskip

Equivalently, $\CY$ is quasi-smooth if it admits a smooth atlas consisting of 
quasi-smooth DG schemes.

\medskip

One can also express this in terms of the cotangent complex of $\CY$. Namely, $\CY$ is quasi-smooth
if and only if it is locally almost of finite type, and for any $k$-point $y\in Y$, the derived cotangent space $T^*_y(\CY)$ 
lives in degrees $[-1,1]$.  
(The cohomology in degree $1$ is responsible for the Lie algebra of the algebraic group of automorphisms
of $y$.)

\sssec{}

For a quasi-smooth derived algebraic stack $\CY$, one defines the classical algebraic stack
$$\Sing(\CY)\to \CY$$ using descent:

\medskip

\noindent For a smooth map $Y\to \CY$, where $Y$ is a DG scheme, we have
$$Y\underset{\CY}\times \Sing(\CY)\simeq \Sing(Y).$$

The fact that this is well-defined relies in the following lemma:

\begin{lem}
For a smooth map of quasi-smooth DG schemes $Y_1\to Y_2$, the natural map
$$Y_1\underset{Y_2}\times \Sing(Y_2)\to \Sing(Y_1)$$
is an isomorphism.
\end{lem}

More invariantly, $\Sing(Y)$ consists of pairs $(y,\xi)$, where $y\in \CY$, and $\xi\in H^{-1}(T^*_x(\CY))$. 

\sssec{}  \label{sss:category on stack}

Let $\CN\subset \Sing(\CY)$ be a conical Zariski-closed subset. We define the category
$\IndCoh_\CN(\CY)$ to be the full subcategory of $\IndCoh(\CY)$ introduced in \secref{sss:indcoh on stack}
defined by the following condition:

\medskip

An object $\CF\in \IndCoh(\CY)$ belongs to $\IndCoh_\CN(\CY)$ if for every $Y\in \affdgSch$ equipped
with a smooth map $Y\to \CY$  (equivalently, for some atlas of such $Y$'s), the corresponding object
$\CF_Y\in \IndCoh(Y)$ belongs to
$$\IndCoh_{Y\underset{\CY}\times \CN}(Y)\subset \IndCoh(Y).$$

\sssec{}  \label{sss:Xi Psi stacks}

By construction, we have a canonically defined action of the monoidal category $\QCoh(\CY)$ on $\IndCoh_{\CN}(\CY)$. 

\medskip

By \secref{sss:varying N bis} we have an adjoint pair of continuous functors
$$\Xi_{\CY,\CN}:\QCoh(\CY)\rightleftarrows \IndCoh_{\CN}(\CY):\Psi_{\CY,\CN}$$
with $\Xi_{\CY,\CN}$ fully faithful. 

\section{Statement of the categorical geometric Langlands conjecture}  \label{s:statement}

For the rest of the paper, we fix $X$ to be a smooth and complete curve over $k$. 

\ssec{The de Rham functor}

\sssec{}

The following general construction will be useful in the sequel. Ley $\CY$ be an arbitrary prestack,
see \secref{sss:prestack}.

\medskip

We define a new prestack $\CY_\dr$ by 
$$\Maps(S,\CY_\dr)=\Maps(({}^{cl}\!S)_{red},\CY),\quad S\in \affdgSch.$$

In the above formula $({}^{cl}\!S)_{red}$ denotes the reduced classical scheme underlying $S$.

\sssec{}  \label{sss:crystals}

For what follows we define the DG category $\Dmod(\CY)$ of D-modules on $\CY$ by 
$$\Dmod(\CY):=\QCoh(\CY_\dr).$$

We refer the reader to \cite{GR}, where this point of view on the theory of D-modules is developed.

\medskip

If $f:\CY_1\to \CY_2$ is a map of prestacks, we shall denote by
$f^\dagger$ the resulting pull-back functor 
$$f^\dagger:\Dmod(\CY_2)\to \Dmod(\CY_1).$$
I.e., $f^\dagger:=(f_\dr)^*$, where $f_\dr:(\CY_1)_\dr\to (\CY_2)_\dr$. 

\sssec{}  \label{sss:no derived}

The following observation makes life somewhat easier:

\medskip

Let $\CY$ be a prestack, which is \emph{locally almost of finite type} 
(see \cite[Sect. 1.3.9]{Stacks} for the definition\footnote{This is a techincal condition satisfied for the prestacks
of interest to us.}).  In this case we have (see \cite[Proposition 1.3.3]{GR}):

\begin{lem}
The prestack $\CY_\dr$ is classical (see \secref{sss:classical prestacks} for what this means)
and locally of finite type.
\end{lem}

The upshot of this lemma is that in order to ``know'' the category $\Dmod(\CY):=\QCoh(\CY_\dr)$,
it is sufficient to consider maps $({}^{cl}\!S)_{red}\to \CY$, where $S$ is classical and of finite type.
In particular, \emph{we do not need derived algebraic geometry} when we study D-modules.

\ssec{The stack of local systems}

The contents of this subsection are a brief digest of \cite[Sect. 10]{AG}. We refer the reader to {\it loc.cit.}
for the proofs of the statements that we quote. 

\sssec{}

Let $G$ be an algebraic group. 
We let $\on{pt}/G$ be the algebraic stack that classifies $G$-torsors. We define the prestacks
$\Bun_G(X)$ and $\LocSys_G(X)$ by
$$\Maps(S,\Bun_G(X))=\Maps(S\times X,\on{pt}/G)$$ and 
$$\Maps(S,\LocSys_G(X))=\Maps(S\times X_\dr,\on{pt}/G).$$

\medskip

Note that we have a natural forgetful map $\LocSys_G(X)\to \Bun_G(X)$
corresponding to the tautological map $X\to X_\dr$. 

\medskip

One shows that $\Bun_G(X)$ is in fact a smooth classical algebraic stack, and that $\LocSys_G(X)$
is a derived algebraic stack. 

\medskip

As $X$ is fixed, we will simply write $\Bun_G$ and $\LocSys_G$, omitting $X$ from the notation. 


\sssec{}

We claim that $\LocSys_G$ is in fact quasi-smooth. Indeed, the cotangent space 
at a point $\sigma\in \LocSys_G$ is canonically isomorphic to
$$\Gamma_\dr(X,\fg^*_\sigma)[1],$$
where $\fg^*_\sigma$ is the local system of vector spaces corresponding to $\sigma$ and the
co-adjoint representation of $G$. 

\medskip

In particular, the complex $\Gamma_\dr(X,\fg^*_\sigma)[1]$ has cohomologies in degrees $[-1,1]$,
as required. 

\sssec{}

The same computation provides a description of the stack $\Sing(\LocSys_G)$:

\begin{cor} \label{c:Arth}
The (classical) stack $\Sing(\LocSys_G)$ is the moduli space of pairs $(\sigma,A)$
where $\sigma\in \LocSys_G$, and $A$ is a horizontal section of the local system
$\fg^*_\sigma$, associated with the co-adjoint representation of $G$.
\end{cor}

\sssec{}

The following property of $\LocSys_G$ is shared by any quasi-smooth algebraic stack which
can be globally written as a complete intersection, see \cite[Corollary 9.2.7 and Sect. 10.6]{AG}:

\begin{lem}  
For any conical Zariski-closed subset $\CN\subset \Sing(\LocSys_G)$, the category
$\IndCoh_\CN(\LocSys_G)$ is compactly generated by its subcategory 
$\Coh_\CN(\LocSys_G)$.
\end{lem}

\ssec{The spectral side of geometric Langlands}

From now on we will assume that $G$ is a connected reductive group.
We let $\cG$ denote the Langlands dual of $G$.

\sssec{}

Consider the stack $\Sing(\LocSys_\cG)$. We will also denote it by $\on{Arth}_\cG$.
This is the stack of Arthur parameters.

\medskip

Let $\on{Nilp}^{\on{glob}}_\cG$ be the conical Zariski-lcosed subset of $\on{Arth}_\cG$
corresponding to those pairs $(\sigma,A)$ (see \corref{c:Arth}) for which $A$
is \emph{nilpotent}, i.e., its value at any (equivalently, some) point of $X$
lies in the cone of nilpotent elements of $\cg^*$. 

\sssec{}  \label{sss:define spectral side}

According to \secref{sss:category on stack}, we have a well-defined DG category
$$\IndCoh_{\on{Nilp}^{\on{glob}}_\cG}(\LocSys_\cG).$$

This is the category that we propose as the spectral (i.e., Galois) side of the
categorical geometric Langlands conjecture. 

\sssec{}

By \secref{sss:Xi Psi stacks}, we have an adjoint pair of functors
\begin{equation} \label{e:embedding QCoh}
\Xi_\cG:\QCoh(\LocSys_\cG)\rightleftarrows \IndCoh_{\on{Nilp}^{\on{glob}}_\cG}(\LocSys_\cG):\Psi_\cG
\end{equation}
with $\Xi_\cG$ fully faithful (we use the subscript ``$\cG$" as a shorthand for ``$\LocSys_\cG,\on{Nilp}^{\on{glob}}_\cG$"). 

\medskip

In other words, the category $\IndCoh_{\on{Nilp}^{\on{glob}}_\cG}(\LocSys_\cG)$ is a modification
of $\QCoh(\LocSys_\cG)$ that has to do with the fact that the derived algebraic stack $\LocSys_\cG$
is not smooth, but only quasi-smooth. 

\medskip

In particular, the functor $\Xi_\cG$ becomes an equivalence 
once we restrict to the open substack of $\LocSys_\cG$ that consists of irreducible local
systems (i.e., ones that do not admit a reduction to a proper parabolic). In fact, the
equivalence takes place over a larger open substack; namely, one corresponding to those local
systems that do not admit a unipotent subgroup of automorphisms. 

\sssec{}

Finally, note that if $G$ (and hence $\cG$) is a torus, then $\on{Nilp}^{\on{glob}}_\cG$ is the zero-section
of $\on{Arth}_\cG$. So, for tori, the spectral side of geometric Langlands is the usual category
$\QCoh(\LocSys_\cG)$. 

\ssec{The geometric side}

\sssec{}

We consider the algebraic stack $\Bun_G$ and the corresponding category $\Dmod(\Bun_G)$
as defined in \secref{sss:crystals}.

\medskip

The categorical geometric Langlands conjecture says:

\begin{conj} \label{c:GL}   \hfill

\smallskip

\noindent{\em(a)}
There exists a uniquely defined equivalence of categories
$$\IndCoh_{\on{Nilp}^{\on{glob}}_\cG}(\LocSys_\cG)\overset{\BL_G}
\longrightarrow \Dmod(\Bun_G),$$
satisfying Property $\on{Wh}^{\on{ext}}$ stated in \secref{sss:cond W ext}.

\smallskip

\noindent{\em(b)} 
The functor $\BL_G$ satisfies Properties $\on{He}^{\on{naive}}$, $\on{Ei}^{\on{enh}}$ and $\on{Km}^{\on{prel}}$, 
stated in Sects. \ref{sss:cond H}, \ref{sss:cond E rat},  and \ref{sss:cond O prel}, respectively.

\end{conj}

\sssec{}

In the rest of the paper we will show that \conjref{c:GL} can be deduced, modulo a number of 
more tractable results that we call ``quasi-theorems", from two more conjectures, namely 
Conjectures \ref{c:Whit ext ff} and \ref{c:Op gen}, the former pertaining
exclusively to $\Dmod(\Bun_G)$, and the latter exclusively to $\IndCoh_{\on{Nilp}^{\on{glob}}_\cG}(\LocSys_\cG)$. 

\medskip

The ``quasi-theorems" referred to above are very close to being theorems for $G=GL_2$ (and we hope will be soon turned
into ones for general $G$). In addition, Conjectures \ref{c:Whit ext ff} and \ref{c:Op gen} are
also are theorems for $G=GL_n$. So, we obtain that \conjref{c:GL} is 
very close to be a theorem for $GL_2$, and is within reach for $GL_n$. 

\medskip

The case of an arbitrary $G$ remains wide open.  

\ssec{The tempered subcategory}

In this subsection we will assume the validity of \conjref{c:GL}. 

\sssec{}

Recall the fully faithful embedding $\Xi_\cG$ of \eqref{e:embedding QCoh}. We obtain that the DG category $\Dmod(\Bun_G)$
contains a full subcategory that under the equivalence of  \conjref{c:GL} corresponds to
$$\QCoh(\LocSys_\cG)\overset{\Xi_\cG}\hookrightarrow \IndCoh_{\on{Nilp}^{\on{glob}}_\cG}(\LocSys_\cG).$$

We denote this subcategory $\Dmod(\Bun_G)_{\on{temp}}$. We regard it as a geometric analog of
the subspace of automorphic functions corresponding to tempered ones.  

\sssec{}  \label{sss:intrinsic}

It is a natural question to ask whether one can define the subcategory 
$$\Dmod(\Bun_G)_{\on{temp}}\subset \Dmod(\Bun_G)$$
intrinsically, i.e., without appealing to the spectral side of Langlands correspondence.

\medskip

This is indeed possible, using the \emph{derived Satake equivalence}, see \cite[Sect. 12.8]{AG} for a 
precise statement (see also \secref{sss:derived Satake temp} below).

\sssec{}

The equivalence
$$\QCoh(\LocSys_\cG)\simeq \Dmod(\Bun_G)_{\on{temp}}$$
implies, in particular, that to every $k$-point $\sigma\in \LocSys_G$ one can
attach an object $\CM_\sigma\in \Dmod(\Bun_G)_{\on{temp}}$; moreover
$\CM_\sigma$ is acted on by the group of automorphisms of $\sigma$. 

\sssec{}

However, it is not clear (and perhaps not true) that the assignment
$$\sigma\rightsquigarrow \CM_\sigma$$
can be extended to points of $\on{Nilp}^{\on{glob}}_\cG$. Indeed, there is no
obvious way to assign to points of $\on{Nilp}^{\on{glob}}_\cG$ objects of
$\IndCoh_{\on{Nilp}^{\on{glob}}_\cG}$. 
\medskip

I.e., at the moment we see no reason that there should be a way of assigning objects
of $\Dmod(\Bun_G)$ to Arthur parameters. 
Rather, what we have is that for an object $\CM\in \Dmod(\Bun_G)$,
there is a well-defined support, which is a closed subset of $\on{Nilp}^{\on{glob}}_\cG$.

\section{The Hecke action}  \label{s:Hecke}

\ssec{The Ran space}

\sssec{}  \label{sss:Ran}

We define the Ran space $\Ran(X)$ of $X$ to be the following prestack: 

\medskip

For $S\in \affdgSch$, the $\infty$-groupoid
$\Maps(S,\Ran(X))$ is the \emph{set} (i.e., a discrete $\infty$-groupoid) of \emph{non-empty finite subsets} of the set 
$$\Maps(S,X_\dr)=\Maps(({}^{cl}\!S)_{red},X).$$

Note that by construction, the map $\Ran(X)\to \Ran(X)_\dr$ is an isomorpism. 

\sssec{} \label{sss:finite sets}

One can right down $\Ran(X)$ explicitly as a colimit in $\on{PreStk}$:
$$\Ran(X)\simeq \underset{I}{\underset{\longrightarrow}{colim}}\, (X^I)_\dr,$$
where the colimit is taken over the category $(\on{fSet}_{\on{surj}})^{\on{op}}$
opposite to that of non-empty finite sets and surjective maps \footnote{The definition
of the Ran space as a colimit was in fact the original definition in \cite{BD1}. The 
definition from \secref{sss:Ran} was suggested in \cite{Bar}.}. (Here for a surjection of finite sets
$I_2\twoheadrightarrow I_1$, the map $X^{I_1}\to X^{I_2}$ is the corresponding diagonal
embedding.)

\sssec{}

We shall symbolically denote points of $\Ran(X)$ by $\ul{x}$. For each $\ul{x}\in \Maps(S,\Ran(X))$
we let $\Gamma_{\ul{x}}$ be the Zariski-closed \emph{subset} of $S\times X$ equal to 
the union of the graphs of the maps $({}^{cl}\!S)_{red}\to X$ that comprise $\ul{x}$. 

\medskip

In particular, we obtain an open subset 
$$S\times X-\{\ul{x}\}:=S\times X-\Gamma_{\ul{x}}\subset S\times X.$$

\medskip

In addition, we have a well-defined formal scheme $\cD_{\ul{x}}$ obtained as the formal completion of
$S\times X$ along $\Gamma_{\ul{x}}$. This formal scheme should be thought of as the $S$-family of formal
disc in $X$ around the points that comprise $\ul{x}$.

\sssec{}

A crucial piece of structure that exists on $\Ran(X)$ is that of commutative semi-group object in the category of prestacks.
The corresponding operation on $\Maps(S,\Ran(X))$ is that of union of finite sets. We denote the resulting map
$$\Ran(X)\times \Ran(X)\to \Ran(X)$$ by 
$\cup$. 

\sssec{} 

Another fundamental fact about the Ran space is its \emph{contractibility}. We will use it its weaker form,
namely \emph{homological contractibility} (see \cite[Sect. 6]{Ga2} for the proof): 

\begin{prop} \label{p:Ran contractible}
The functor 
$$\Vect=\QCoh(\on{pt})\overset{p^*}\to \QCoh(\Ran(X))$$
is fully faithful, where $p$ denotes the projection $\Ran(X)\to \on{pt}$.
\end{prop}

(Note also that the fact that the map $\Ran(X)\to \Ran(X)_\dr$ is an isomorphism implies
that the natural forgetful functor $\Dmod(\Ran(X))\to \QCoh(\Ran(X))$ is an equivalence.)
 
\ssec{Representations spread over the Ran space}  \label{ss:rep Ran}

\sssec{}

We shall now define the Ran version of the category of representations of $\cG$ (here $\cG$ may be
any algebraic group). In fact we are going to start with an arbitrary prestack $\CY$ (in our case
$\CY=\on{pt}/\cG$) and attach to it a new prestack, denoted $\CY_{\Ran(X)}$, equipped with a map
to $\Ran(X)$.

\medskip

Namely, we let an $S$-point of $\CY_{\Ran(X)}$ to be the data of a pair $(\ul{x},y)$, where
$\ul{x}$ is an $S$-point of $\Ran(X)$, and $y$ is a datum of a map
$$(\cD_{\ul{x}})_\dr\underset{S_\dr}\times S\to \CY.$$

\sssec{}

In order to decipher this definition, let us describe explicitly the fiber of $\CY_{\Ran(X)}$
over a given $k$-point $\ul{x}$ of $\Ran(X)$.

\medskip

Let $\ul{x}$ correspond to the finite collection of distinct points $x_1,...,x_n$ of $X$. We claim
that the fiber product
$$\CY_{\Ran(X)}\underset{\Ran(X)}\times \on{pt}$$
identifies with the product of copies of $\CY$, one for each index $i$.

\medskip

This follows from the fact that $\cD_{\ul{x}}$ is the disjoint union of the formal discs $\cD_{x_i}$. Hence, 
the prestack $(\cD_{\ul{x}})_\dr$ identifies with the disjoint union of copies of $\on{pt}$,
one for each $x_i$.

\sssec{} \label{sss:monoidal structure Ran}

We set 
$$\QCoh(\CY)_{\Ran(X)}:=\QCoh(\CY_{\Ran(X)}).$$

We claim that the DG category $\QCoh(\CY)_{\Ran(X)}$ has a naturally defined structure of (non-unital)
symmetric monoidal category. 

\medskip

Namely, consider the fiber product
$$\CY_{\Ran(X)}\underset{\Ran(X)}\times (\Ran(X)\times \Ran(X)),$$
where the map $\Ran(X)\times \Ran(X)\to \Ran(X)$ is $\cup$. 

\medskip

We have a diagram
$$
\CD
\CY_{\Ran(X)}\underset{\Ran(X)}\times (\Ran(X)\times \Ran(X)) @>{\mathsf{res}}>>  \CY_{\Ran(X)}\times \CY_{\Ran(X)} \\
@V{\on{id}\times \cup}VV   \\
\CY_{\Ran(X)},
\endCD
$$
where the map $\mathsf{res}$ corresponds to restricting maps to $\CY$ along
$$\cD_{\ul{x}'}\to \cD_{\ul{x}'\cup \ul{x}''}\leftarrow \cD_{\ul{x}''}.$$

\medskip

We define the functor
$$\QCoh(\CY_{\Ran(X)})\otimes \QCoh(\CY_{\Ran(X)})\to \QCoh(\CY_{\Ran(X)})$$
to be the composition
$$(\on{id}\times \cup)_!\circ (\mathsf{res})^*,$$
where $(\on{id}\times \cup)_!$ is the \emph{left} adjoint \footnote{The fact that this left adjoint exists requires a proof;
in our case this essentially follows from the fact that map $\cup$ is proper.} of the functor $(\on{id}\times \cup)^!$.

\sssec{}

Thus, we set
$$\Rep(\cG)_{\Ran(X)}:=\QCoh(\on{pt}/\cG)_{\Ran(X)}:=\QCoh((\on{pt}/\cG)_{\Ran(X)}).$$
We view it as a (non-unital) symmetric monoidal category. 

\ssec{Relation to the stack of local systems}

\sssec{}

Note that by construction we have the following diagram of prestacks
\begin{equation} \label{e:local to global LocSys}
\CD
\LocSys_\cG \times \Ran(X)  @>{\on{ev}}>>  (\on{pt}/\cG)_{\Ran(X)}  \\
@V{\on{id}\times p}VV   \\
\LocSys_\cG,
\endCD
\end{equation}
where the map $\on{ev}$ corresponds to restriction of a map to the target $\on{pt}/\cG$ along $(\cD_{\ul{x}})_\dr\to X_\dr$.

\medskip

We have a pair of mutually adjoint functors 
\begin{equation} \label{e:local to global LocSys functors}
(\on{id}\times p)_!\circ \on{ev}^*:\QCoh((\on{pt}/\cG)_{\Ran(X)})\rightleftarrows \QCoh(\LocSys_\cG):\on{ev}_*\circ (\on{id}\times p)^*.
\end{equation}
The left adjoint functor (i.e., $(\on{id}\times p)_!\circ \on{ev}^*$) has a natural symmetric monoidal structure, where the symmetric
monoidal structure on $\QCoh(\LocSys_\cG)$ is the usual tensor product. 

\begin{rem}
We note that the diagram \eqref{e:local to global LocSys} and the functors 
\eqref{e:local to global LocSys functors} makes sense more generally, when $\on{pt}/\cG$ is replaced by an arbitrary 
prestack $\CY$. In this case instead of $\LocSys_\cG$ we have the prestack $\bMaps(X_\dr,\CY)$, defined so that
$$\Maps(S,\bMaps(X_\dr,\CY))=\Maps(S\times X_\dr,\CY).$$
\end{rem}

\sssec{}

We denote 
$$\on{Loc}_{\cG,\on{spec}} :=(\on{id}\times p)_!\circ \on{ev}^* \text{ and } \on{co-Loc}_{\cG,\on{spec}}:=\on{ev}_*\circ (\on{id}\times p)^*.$$

\medskip

We have the following result:

\begin{prop}[joint with J.~Lurie, unpublished]   \label{p:local to global LocSys}
The functor 
$$\on{co-Loc}_\cG:\QCoh(\LocSys_\cG)\to \QCoh((\on{pt}/\cG)_{\Ran(X)})=\Rep(\cG)_{\Ran(X)}$$
is fully faithful.
\end{prop}

Thus, \propref{p:local to global LocSys} realizes a ``local-to-global" principle for $\LocSys_\cG$, namely,
it embeds the ``global" category $\QCoh(\LocSys_\cG)$ into a ``local" one, namely, $\Rep(\cG)_{\Ran(X)}$.

\begin{rem}
The assertion of \propref{p:local to global LocSys} is valid more generally, when the stack $\on{pt}/\cG$
is replaced by an arbitrary quasi-compact derived algebraic stack $\CY$ locally almost of finite type
with an affine diagonal. 
\end{rem} 

\ssec{Hecke action}

\sssec{}  \label{sss:Hecke stack}

We define the Ran version of the Hecke stack $\on{Hecke}(G)_{\Ran(X)}$ as follows: its $S$-points
are quadruples $(\ul{x},\CP^1_G,\CP^2_G,\beta)$, where $\ul{x}$ is an $S$-point of $\Ran(X)$,
$\CP^1_G$ and $\CP^2_G$ are two $S$-points of $\Bun_G$, and $\beta$ is the isomorphism
of $G$-bundles
$$\CP^1_G|_{S\times X-\ul{x}}\simeq \CP^2_G|_{S\times X-\ul{x}}.$$ 

We let $\hl$ and $\hr$ denote the two forgetful maps $\on{Hecke}(G)_{\Ran(X)}\to \Bun_G$.

\sssec{}  \label{sss:Hecke action}

We claim that the category $\Dmod(\on{Hecke}(G)_{\Ran(X)})$ has a naturally 
defined (non-unital) monoidal structure, and that the resulting monoidal category acts on 
$\Dmod(\Bun_G)$.

\medskip

These two pieces of structure are constructed by pull-push as in \secref{sss:monoidal structure Ran} using the diagrams

$$
\CD
\on{Hecke}(G)_{\Ran(X)}\underset{\hr,\Bun_G,\hl}\times  \on{Hecke}(G)_{\Ran(X)}   @>>>   
\on{Hecke}(G)_{\Ran(X)}\times \on{Hecke}(G)_{\Ran(X)}  \\
@VVV    \\
\on{Hecke}(G)_{\Ran(X)}\underset{\Ran(X)}\times (\Ran(X)\times \Ran(X))   \\
@V{\on{id}\times \cup}VV  \\
\on{Hecke}(G)_{\Ran(X)}
\endCD
$$ 
and
$$
\CD   
\on{Hecke}(G)_{\Ran(X)}  @>{\on{id}\times \hr}>>  \on{Hecke}(G)_{\Ran(X)}\times \Bun_G  \\
@V{\hl}VV   \\
\Bun_G
\endCD
$$
respectively. 

\sssec{}

We have the following input from the geometric Satake equivalence:

\begin{prop}  \label{p:Satake Hecke}
There exists a canonically defined monoidal functor 
\begin{equation} \label{e:Satake Hecke}
\on{Sat}(G)^{\on{naive}}_{\Ran(X)}:\Rep(\cG)_{\Ran(X)}\to \Dmod(\on{Hecke}(G)_{\Ran(X)}).
\end{equation} 
\end{prop}

\noindent The functor $\on{Sat}(G)^{\on{naive}}_{\Ran(X)}$ follows from the \emph{naive} (or ``usual")
geometric Satake equivalence, and is essentialy constructed in \cite{MV}.

\sssec{Compatibility with the Hecke action}  \label{sss:cond H}

We are now able to formulate Property $\on{He}^{\on{naive}}$ (``He" stands for ``Hecke") 
of the geometric Langlands functor $\BL_G$ in \conjref{c:GL}:

\medskip

\noindent{\bf Property $\mathbf{He}^{\mathbf{naive}}$:} {\it 
We shall say that the functor $\BL_G$ satisfies \emph{Property $\on{He}^{\on{naive}}$} if it
intertwines the monoidal actions of $\Rep(\cG)_{\Ran(X)}$ on the categories
$\IndCoh_{\on{Nilp}^{\on{glob}}_\cG}(\LocSys_\cG)$ and $\Dmod(\Bun_G)$, where:

\begin{itemize}

\item
The action of $\Rep(\cG)_{\Ran(X)}$ on $\IndCoh_{\on{Nilp}^{\on{glob}}_\cG}(\LocSys_\cG)$ is
obtained via the the monoidal functor
$$\on{Loc}_{\cG,\on{spec}}:\Rep(\cG)_{\Ran(X)}=\QCoh((\on{pt}/\cG)_{\Ran(X)})\to \QCoh(\LocSys_\cG)$$
and the action of $\QCoh(\LocSys_\cG)$ on $\IndCoh_{\on{Nilp}^{\on{glob}}_\cG}(\LocSys_\cG)$ 
(see \secref{sss:category on stack});

\item 
The action of $\Rep(\cG)_{\Ran(X)}$ on $\Dmod(\Bun_G)$ is obtained via 
the monoidal functor $\on{Sat}(G)^{\on{naive}}_{\Ran(X)}$ and the action of
$\Dmod(\on{Hecke}(G)_{\Ran(X)})$ on $\Dmod(\Bun_G)$ (see \secref{sss:Hecke action}).

\end{itemize}}

\ssec{The vanishing theorem}

\sssec{}

Consider again the action of $\Rep(\cG)_{\Ran(X)}$ on $\Dmod(\Bun_G)$, described above. We claim:

\begin{thm} \label{t:generalized vanishing}
The action of the monoidal ideal
$$\on{ker}\left(\on{Loc}_{\cG,\on{spec}}:\Rep(\cG)_{\Ran(X)}\to \QCoh(\LocSys_\cG)\right)$$
on $\Dmod(\Bun_G)$ is zero. 
\end{thm} 

The proof of this theorem will be sketched in \secref{ss:proof of vanishing}. It uses the same
basic ingredients as the proof of \conjref{c:GL}, but is much simpler. A more detailed exposition
can be found in \cite{GenVan}.

\begin{rem} 
We note that \thmref{t:generalized vanishing} is a generalization of a vanishing theorem proved in \cite{Ga1}
that concerned the case of $G=GL_n$ and a particular object of the category $\Rep(\cG)_{\Ran(X)}$ lying
in the kernel of $\on{Loc}_{\cG,\on{spec}}$. 
\end{rem}

\sssec{}

Combining \propref{p:local to global LocSys} and \thmref{t:generalized vanishing}, we obtain:

\begin{cor} \label{c:generalized vanishing}
The monoidal action of $\Rep(\cG)_{\Ran(X)}$ on $\Dmod(\Bun_G)$ uniquely factors through a monoidal
action of $\QCoh(\LocSys_\cG)$ on $\Dmod(\Bun_G)$.
\end{cor}  

\ssec{Derived Satake}

The material of this subsection is not crucial for the understanding of the outline of the proof of \conjref{c:GL}
presented in the rest of the paper.

\sssec{}

Let us fix a $k$-point $x\in X$. We let $\on{Hecke}(\cG,\on{spec})_x$ denote the DG algebraic stack 
whose $S$-points are triples
$$((\CP^1_\cG,\nabla^1),(\CP^2_G,\nabla^2),\beta),$$
where $(\CP^i_\cG,\nabla^i)$ are objects of $\Maps(S,\LocSys_\cG)$ and $\beta$ is an isomorphism
of the resulting two maps $S\times (X-x)_\dr\to \on{pt}/\cG$ obtained from $(\CP^i_\cG,\nabla^i)$ by
restriction along
$$S\times (X-x)_\dr\hookrightarrow S\times X_\dr.$$

\medskip

The two projections $\hl_{\on{spec}},\hr_{\on{spec}}:\on{Hecke}(\cG,\on{spec})_x\to \LocSys_\cG$
define on $\on{Hecke}(\cG,\on{spec})_x$ a structure of groupoid acting on $\LocSys_\cG$. In fact, 
we have a canonicaly defined commutative diagram, in which both sides are Cartesian
\begin{equation} \label{e:Hecke spec}
\xy
(-30,0)*+{\LocSys_\cG}="X";
(30,0)*+{\LocSys_\cG}="Y";
(0,20)*+{\on{Hecke}(\cG,\on{spec})_x}="Z";
(-30,-30)*+{\on{pt}/\cG}="X_1";
(30,-30)*+{\on{pt}/\cG,}="Y_1";
(0,-10)*+{\on{Hecke}(\cG,\on{spec})^{\on{loc}}_x}="Z_1";
{\ar@{->}_{\hl_{\on{spec}}} "Z";"X"};
{\ar@{->}^{\hr_{\on{spec}}} "Z";"Y"};
{\ar@{->}"Z_1";"X_1"};
{\ar@{->}"Z_1";"Y_1"};
{\ar@{->}"Z";"Z_1"};
{\ar@{->}"X";"X_1"};
{\ar@{->}"Y";"Y_1"};
\endxy
\end{equation} 
where
$$\on{Hecke}(\cG,\on{spec})^{\on{loc}}_x:=(\on{pt}\underset{\cg}\times \on{pt})/\cG,$$
see \cite[Sect. 12.7]{AG}.

\sssec{}  \label{sss:Hecke spec action}

The structure of groupoid on $\on{Hecke}(\cG,\on{spec})^{\on{loc}}_x$ defines on 
$\IndCoh(\on{Hecke}(\cG,\on{spec})^{\on{loc}}_x)$ a structure of monoidal category,
where we use the $(\IndCoh,*)$-pushforward and !-pull-back as our pull-push functors.

\medskip

Moreover, the diagram \eqref{e:Hecke spec} defines an action of 
$\IndCoh(\on{Hecke}(\cG,\on{spec})^{\on{loc}}_x)$ on the category $\IndCoh(\LocSys_\cG)$ that
preserves the subcategory 
$$\IndCoh_{\on{Nilp}^{\on{glob}}_\cG}(\LocSys_\cG)\subset \IndCoh(\LocSys_\cG).$$

\medskip

There is a naturally defined monoidal functor
$$\Rep(\cG)=\QCoh(\on{pt}/\cG)\to \IndCoh(\on{Hecke}(\cG,\on{spec})^{\on{loc}}_x)$$
corresponding to the diagonal map
$$\on{pt}/\cG\to (\on{pt}\underset{\cg}\times \on{pt})/\cG=:\on{Hecke}(\cG,\on{spec})^{\on{loc}}_x.$$


\sssec{}

Let $\on{Hecke}(G)_x$ be the fiber of $\on{Hecke}(G)_{\Ran(X)}$ over the point $\{x\}\in \Ran(X)$. 
The restrictions of the projections of $\hl$ and $\hr$ to $\on{Hecke}(G)_x$ define on it a structure
of groupoid acting on $\Bun_G$. Hence, the category $\Dmod(\on{Hecke}(G)_x)$ acquires a monoidal 
structure. Direct image (i.e., the functor \emph{left adjoint} to restriction) defines a monoidal functor
$$\Dmod(\on{Hecke}(G)_x)\to \Dmod(\on{Hecke}(G)_{\Ran(X)}).$$

\medskip

Similarly, we have a naturally defined monoidal functor
$$\Rep(\cG)\to \Rep(\cG)_{\Ran(X)},$$
\emph{left adjoint} to the restriction functor. Part of the construction of the functor $\on{Sat}(G)^{\on{naive}}_{\Ran(X)}$
is that we have a naturally defined monoidal functor
$$\on{Sat}(G)^{\on{naive}}_x:\Rep(\cG)\to \Dmod(\on{Hecke}(G)_x)$$
that makes the diagram
$$
\CD
\Rep(\cG)   @>{\on{Sat}(G)^{\on{naive}}_x}>>  \Dmod(\on{Hecke}(G)_x) \\
@VVV   @VVV   \\
\Rep(\cG)_{\Ran(X)}   @>{\on{Sat}(G)^{\on{naive}}_{\Ran(X)}}>>   \Dmod(\on{Hecke}(G)_{\Ran(X)})
\endCD
$$
commute.

\medskip

We now claim:

\begin{prop} \label{p:derived Satake pt}  
There exists a canonically defined monoidal functor
$$\on{Sat}(G)_x:\IndCoh(\on{Hecke}(\cG,\on{spec})^{\on{loc}}_x)\to \Dmod(\on{Hecke}(G)_x)$$
that makes the diagram
$$
\CD
\Rep(\cG)  @>{\on{Sat}(G)^{\on{naive}}_x}>>  \Dmod(\on{Hecke}(G)_x)  \\
@VVV   @VV{\on{id}}V    \\
\IndCoh(\on{Hecke}(\cG,\on{spec})^{\on{loc}}_x)  @>{\on{Sat}(G)_x}>>  \Dmod(\on{Hecke}(G)_x) 
\endCD
$$
commute. 
\end{prop}

\propref{p:derived Satake pt} follows from the local \emph{full} (or ``derived")
geometric Satake equivalence, see \cite[Theorem 12.5.3]{AG}
(which in turn follows from \cite[Theorem 5]{BF}).  

\begin{rem}  \label{r:Sat equiv}
Note that the stack  $$\on{Hecke}(\cG,\on{spec})^{\on{loc}}_x\simeq (\on{pt}\underset{\cg}\times \on{pt})/\cG$$
is quasi-smooth, and the corresponding classical stack 
$\Sing(\on{Hecke}(\cG,\on{spec})^{\on{loc}}_x)$ identifies canonically with the classical stack $\cg^*/\cG$.
Let $\on{Nilp}^{\on{loc}}_\cG\subset \Sing(\on{Hecke}(\cG,\on{spec})^{\on{loc}}_x)$ be the nilpotent locus,
and consider the corresponding category 
$\IndCoh_{\on{Nilp}^{\on{loc}}_\cG}(\on{Hecke}(\cG,\on{spec})^{\on{loc}}_x)$. 

\medskip

One can show that the functor
$$\on{Sat}(G)_x:\IndCoh(\on{Hecke}(\cG,\on{spec})^{\on{loc}}_x) \to  \Dmod(\on{Hecke}(G)_x)$$
canonically factors as a composition of monoidal functors
\begin{multline}  \label{e:factoring Satake}
\IndCoh(\on{Hecke}(\cG,\on{spec})^{\on{loc}}_x) \to \IndCoh_{\on{Nilp}^{\on{loc}}_\cG}(\on{Hecke}(\cG,\on{spec})^{\on{loc}}_x)\to \\
\to \Dmod(\on{Hecke}(G)^{\on{loc}}_x)\to \Dmod(\on{Hecke}(G)_x),
\end{multline}
where

\begin{itemize} 

\item $\IndCoh(\on{Hecke}(\cG,\on{spec})^{\on{loc}}_x) \to \IndCoh_{\on{Nilp}^{\on{loc}}_\cG}(\on{Hecke}(\cG,\on{spec})^{\on{loc}}_x)$ 
is the co-localization functor, left adjoint to the tautological embedding 
$$\IndCoh_{\on{Nilp}^{\on{loc}}_\cG}(\on{Hecke}(\cG,\on{spec})^{\on{loc}}_x)\hookrightarrow \IndCoh(\on{Hecke}(\cG,\on{spec})^{\on{loc}}_x).$$

\item $\on{Hecke}(G)^{\on{loc}}_x$ is the local version of the Hecke stack, i.e., $G(\wh\CO_x)\backslash G(\wh\CK_x)/G(\wh\CO_x)$,
where $\wh\CO_x$ and $\wh\CK_x$ are the completed local ring and field at the point $x\in X$, respectively. 

\end{itemize}

\medskip

A salient feature of this situation is that the middle functor 
$$\IndCoh_{\on{Nilp}^{\on{loc}}_\cG}(\on{Hecke}(\cG,\on{spec})^{\on{loc}}_x)\to
\Dmod(\on{Hecke}(G)^{\on{loc}}_x)$$ in \eqref{e:factoring Satake}
is an \emph{equivalence} (unlike the version with $\on{Sat}(G)^{\on{naive}}_x$). 

\end{rem}

\sssec{}

We can now formulate the following variant of Property $\on{He}^{\on{naive}}$ of the functor $\BL_G$:

\medskip

\noindent{\bf Property $\mathbf{He}^{x}$:} {\it 
We shall say that the functor $\BL_G$ satisfies \emph{Property $\on{He}^x$} if it
intertwines the monoidal actions of $\IndCoh(\on{Hecke}(\cG,\on{spec})^{\on{loc}}_x)$ on the categories
$\IndCoh_{\on{Nilp}^{\on{glob}}_\cG}(\LocSys_\cG)$ and $\Dmod(\Bun_G)$, where:

\begin{itemize}

\item
The action of $\IndCoh(\on{Hecke}(\cG,\on{spec})^{\on{loc}}_x)$ on $\IndCoh_{\on{Nilp}^{\on{glob}}_\cG}(\LocSys_\cG)$ is
one from \secref{sss:Hecke spec action};

\item 
The action of $\IndCoh(\on{Hecke}(\cG,\on{spec})^{\on{loc}}_x)$ on $\Dmod(\Bun_G)$ is obtained via 
the monoidal functor $\on{Sat}(G)_x$ and the action of
$\Dmod(\on{Hecke}(G)_{\Ran(X)})$ on $\Dmod(\Bun_G)$ (see \secref{sss:Hecke action}).

\end{itemize}}

\medskip

It will follow from the constructions carried out in the rest of the paper that, in the same circumstances under
which we can prove \conjref{c:GL}, the resulting functor $\BL_G$ will also satisfy Property $\on{He}^x$
for any $x\in X$.

\sssec{} \label{sss:derived Satake temp}
The intrinsic characterization of the subcategory
$$\Dmod(\Bun_G)_{\on{temp}} \subset \Dmod(\Bun_G)$$
mentioned in \secref{sss:intrinsic} is formulated in terms of the above action of 
$\IndCoh(\on{Hecke}(\cG,\on{spec})^{\on{loc}}_x)$ on $\Dmod(\Bun_G)$:

\medskip

An object $\CM\in \Dmod(\Bun_G)$ belongs to $\Dmod(\Bun_G)_{\on{temp}}$ if and only
if the functor
$$\CF\mapsto \CF\star \CM,\quad \IndCoh(\on{Hecke}(\cG,\on{spec})^{\on{loc}}_x)\to \Dmod(\Bun_G)$$
factors through the quotient 
$$\IndCoh(\on{Hecke}(\cG,\on{spec})^{\on{loc}}_x) \overset{\Psi_{\on{Hecke}(\cG,\on{spec})^{\on{loc}}_x}}
\twoheadrightarrow \QCoh(\on{Hecke}(\cG,\on{spec})^{\on{loc}}_x)$$
(for any chosen point $x$). In the above formula $-\star-$ denotes the monoidal action of $\IndCoh(\on{Hecke}(\cG,\on{spec})^{\on{loc}}_x)$
on $\Dmod(\Bun_G)$, and we remind that $\Psi_{\on{Hecke}(\cG,\on{spec})^{\on{loc}}_x}$ denotes the functor introduced in 
\secref{sss:Xi Psi stacks}, for the stack $\on{Hecke}(\cG,\on{spec})^{\on{loc}}_x$. 

\ssec{The Ran version of derived Satake}

The material of this subsection will not be used elsewhere in the paper. The reason we include it is
to mention another important piece of structure present in the geometric Langlands picture, and one 
which is crucial for the proofs. \footnote{As of now, the material in this subsection does not have a 
reference in the existing literature.}

\sssec{}  \label{sss:spec Ran Hecke}

Along with the stack $\on{Hecke}(\cG,\on{spec})^{\on{loc}}_x$, one can consider
its Ran version, denoted $\on{Hecke}(\cG,\on{spec})^{\on{loc}}_{\Ran(X)}$ that fits into the Cartesian diagram
$$
\xy
(-30,0)*+{\LocSys_\cG\times \Ran(X)}="X";
(30,0)*+{\LocSys_\cG\times \Ran(X)}="Y";
(0,20)*+{\on{Hecke}(\cG,\on{spec})_{\Ran(X)}}="Z";
(-30,-30)*+{(\on{pt}/\cG)_{\Ran(X)}}="X_1";
(30,-30)*+{(\on{pt}/\cG)_{\Ran(X)},}="Y_1";
(0,-10)*+{\on{Hecke}(\cG,\on{spec})^{\on{loc}}_{\Ran(X)}}="Z_1";
{\ar@{->}_{\hl_{\on{spec}}} "Z";"X"};
{\ar@{->}^{\hr_{\on{spec}}} "Z";"Y"};
{\ar@{->}"Z_1";"X_1"};
{\ar@{->}"Z_1";"Y_1"};
{\ar@{->}"Z";"Z_1"};
{\ar@{->}"X";"X_1"};
{\ar@{->}"Y";"Y_1"};
\endxy
$$

The reason we do not formally give the definition of $\on{Hecke}(\cG,\on{spec})^{\on{loc}}_{\Ran(X)}$ is that it
involves the notion of $\cG$-local system on the \emph{parameterized formal punctured disc}
(as opposed to the parameterized formal non-punctured disc $\cD_{\ul{x}}$), the discussion of which would be 
too lengthy for the intended scope of this paper. 

\medskip

Let us, nonetheless, indicate the formal structure of this piece of the picture: 

\sssec{}  \label{sss:Hecke spec Ran}

Although the prestack $\on{Hecke}(\cG,\on{spec})^{\on{loc}}_{\Ran(X)}$ is not a DG algebraic stack,
the category $\IndCoh(\on{Hecke}(\cG,\on{spec})^{\on{loc}}_{\Ran(X)})$ is well-defined, carries
a monoidal structure, and as such acts on $\IndCoh(\LocSys_\cG)$
preserving $\IndCoh_{\on{Nilp}^{\on{glob}}_\cG}(\LocSys_\cG)$. 

\medskip

We have naturally defined monoidal functors
\begin{equation} \label{e:small to big Hecke}
\Rep(\cG)_{\Ran(X)}\to \IndCoh(\on{Hecke}(\cG,\on{spec})^{\on{loc}}_{\Ran(X)})\leftarrow 
\IndCoh(\on{Hecke}(\cG,\on{spec})^{\on{loc}}_x),
\end{equation}
and a monoidal functor
$$\on{Sat}(G)_{\Ran(X)}:
\IndCoh(\on{Hecke}(\cG,\on{spec})^{\on{loc}}_{\Ran(X)})\to \Dmod(\on{Hecke}(G)_{\Ran(X)})$$
that restricts to the functors $\on{Sat}(G)_x$ and $\on{Sat}(G)^{\on{naive}}_{\Ran(X)}$,
respectively. 

\begin{rem}
As in Remark \ref{r:Sat equiv}, the functor $\on{Sat}(G)_{\Ran(X)}$ factors through an equivalence from 
a co-localization 
$\IndCoh_{\on{Nilp}^{\on{loc}}_{\cG}}(\on{Hecke}(\cG,\on{spec})^{\on{loc}}_{\Ran(X)})$ to the appropriately defined
category $\Dmod(\on{Hecke}(G)^{\on{loc}}_{\Ran(X)})$.
\end{rem}

\sssec{}

The full Hecke comptatibility property reads:

\medskip

\noindent{\bf Property $\mathbf{He}$:} {\it 
We shall say that the functor $\BL_G$ satisfies \emph{Property $\on{He}$} if it
intertwines the monoidal actions of $\IndCoh(\on{Hecke}(\cG,\on{spec})^{\on{loc}}_{\Ran(X)})$ on the categories
$\IndCoh_{\on{Nilp}^{\on{glob}}_\cG}(\LocSys_\cG)$ and $\Dmod(\Bun_G)$, where:

\begin{itemize}

\item
The action of $\IndCoh(\on{Hecke}(\cG,\on{spec})^{\on{loc}}_{\Ran(X)})$ on $\IndCoh_{\on{Nilp}^{\on{glob}}_\cG}(\LocSys_\cG)$ is
as in \secref{sss:Hecke spec Ran}

\item 
The action of $\IndCoh(\on{Hecke}(\cG,\on{spec})^{\on{loc}}_{\Ran(X)})$ on $\Dmod(\Bun_G)$ is obtained via 
the functor $\on{Sat}(G)_{\Ran(X)}$ and the action of
$\Dmod(\on{Hecke}(G)_{\Ran(X)})$ on $\Dmod(\Bun_G)$ (see \secref{sss:Hecke action}).

\end{itemize}}

\medskip

Tautologically, Property $\on{He}$ implies both Properties $\on{He}^{\on{naive}}$ and $\on{He}^x$. 

\medskip

As with Property $\on{He}^x$, it will follow from the constructions carried out in the rest of the paper that, 
in the same circumstances under which we can prove \conjref{c:GL}, the resulting functor $\BL_G$
will satisfy Property $\on{He}$. 

\section{The Whittaker model}  \label{s:Whit}

\ssec{The space of generic reductions to the Borel}

In this subsection we are going to introduce a space (=prestack) $\Bun_G^{B\on{-gen}}$
that will figure prominently in this paper.
This is the space that classifies pairs consisting of a $G$-bundle and its reduction to the Borel subgroup
defined \emph{generically} on $X$. The approach to $\Bun_G^{B\on{-gen}}$ described below was 
developed by J.~Barlev in \cite{Bar}.\footnote{A much more cumbersome treatment, but one which
only uses algebraic stacks or ind-algebraic stacks can be found in \cite[Sects. 5 and 6]{ExtWhit}.}

\medskip

In this subsection, as well as in Sects. \ref{ss:for N}-\ref{ss:Whit coeff}, we will be exclusively dealing 
with D-modules, so derived algebraic geometry will play no role (see \secref{sss:no derived}).

\sssec{}  \label{sss:defn rat}

First, we consider the prestack that attaches to $S\in \affSch$ the groupoid of triples
$$(\CP_G,U,\alpha),$$
where 

\begin{itemize}

\item $\CP_G$ is a $G$-bundle on $S\times X$;

\item $U$ is a Zariski-open subset of $S\times X$, such that for each $k$-point of $S$,
the corresponding open subset 
$$\on{pt}\underset{S}\times U\subset \on{pt}\underset{S}\times (S\times X)\simeq X$$ 
is non-empty (equivalently, dense in $X$);  

\item $\alpha$ is a datum of a reduction of $\CP_G|_U$ to the Borel subgroup $B$. 

\end{itemize}

In what follows we shall denote by $\CP_{B,U}$ the $B$-bundle on $U$ corresponding to $\alpha$. 
We shall denote by $\CP_{T,U}$ the induced $T$-bundle. 

\medskip

We define $\Bun_G^{B\on{-gen}}$ to be the prestack that attaches to $S\in \affSch$
the quotient of the above groupoid of triples $(\CP_G,U,\alpha)$ by the equivalence relation that
identifies $(\CP^1_G,U^1,\alpha^1)$ and $(\CP^2_G,U^2,\alpha^2)$ if
$$\CP^1_G\simeq \CP^2_G$$
and for this identification, the data of $\alpha^1$ and $\alpha^2$ coincide over $U^1\cap U^2$. 

\sssec{}

We have a natural forgetful map $\sfp^{\on{enh}}_B:\Bun_G^{B\on{-gen}}\to \Bun_G$.  However,
the fibers of this map are neither indschemes nor algebraic stacks. 

\medskip

Nonetheless, we have the following assertion established in \cite[Proposition 3.3.2]{Bar}:

\begin{prop}  \label{p:Barlev}
There exists an algebraic stack $\CY_0$, equipped with a proper schematic map
to $\Bun_G$, and a proper schematic equivalence relation 
$$\CY_1\rightrightarrows \CY_0$$
such that $\Bun_G^{B\on{-gen}}$ identifies with the quotient of $\CY_0$ by
$\CY_1$, up to sheafification in the Zariski topology.
\end{prop}

\begin{rem}
The pair $\CY_1\rightrightarrows \CY_0$ is in fact very explicit. Namely, 
$\CY_0$ is the algebraic stack $\BunBb$ (the Drinfeld compactification),
and $\CY_1$ is defined as
$$\BunBb\underset{\Bun_G^{B\on{-gen}}}\times \BunBb.$$
One shows that $\CY_1$ is indeed an algebraic stack, and the two projections
from $\CY_1$ to $\CY_0$ are schematic and proper. 
\end{rem} 

\sssec{}

Consider now the usual stack $\Bun_B$ that classifies $B$-bundles on $X$. We have a tautological
map
$$\imath_B:\Bun_B\to \Bun_G^{B\on{-gen}}.$$

The valuative criterion of properness implies that the map $\imath_B$
induces an isomorphism of groupoids of field-valued points. In particular, the groupoid of $k$-points
of $\Bun_G^{B\on{-gen}}$ identifies canonically with the double quotient
$$B(K)\backslash G(\BA)/G(\BO),$$
where $K$ is the field of rational functions on $X$, $\BA$ denotes the ring of ad\`eles, and $\BO$
is the ring of integral ad\`eles.

\medskip

However, the map $\iota_B$ itself is, of course, not an isomorphism. 
For example, one can show that connected components of $\Bun_G^{B\on{-gen}}$ 
are in bijection with those of $\Bun_G$, i.e., $\pi_1(G)$, whereas connected
components of $\Bun_B$ are in bijection with the coweight lattice $\check\Lambda$ of $G$.

\medskip

One can view $\Bun_G^{B\on{-gen}}$ as equipped with a stratification, while 
the map $\imath_B$ is the map from the disjoint of the strata. 

\sssec{}  \label{sss:Hecke action on BunBrat}

Recall the Hecke stack $\on{Hecke}(G)_{\Ran(X)}$. We claim it naturally lifts to $\Bun_G^{B\on{-gen}}$ 
in the sense that we have a commutative diagram
$$
\xy
(-30,0)*+{\Bun_G^{B\on{-gen}}}="X";
(30,0)*+{\Bun_G^{B\on{-gen}}}="Y";
(0,20)*+{\Bun_G^{B\on{-gen}}\underset{\Bun_G}\times \on{Hecke}(G)_{\Ran(X)}\simeq \on{Hecke}(G)_{\Ran(X)}
\underset{\Bun_G}\times \Bun_G^{B\on{-gen}}}="Z";
(-30,-30)*+{\Bun_G}="X_1";
(30,-30)*+{\Bun_G,}="Y_1";
(0,-10)*+{\on{Hecke}(G)_{\Ran(X)}}="Z_1";
{\ar@{->} "Z";"X"};
{\ar@{->}  "Z";"Y"};
{\ar@{->}_{\hl} "Z_1";"X_1"};
{\ar@{->}^{\hr} "Z_1";"Y_1"};
{\ar@{->}"Z";"Z_1"};
{\ar@{->}"X";"X_1"};
{\ar@{->}"Y";"Y_1"};
\endxy
$$

In particular, we obtain a natural action of the monoidal category $\Dmod(\on{Hecke}(G)_{\Ran(X)})$
on $\Dmod(\Bun_G^{B\on{-gen}})$. 

\medskip

Using the functor $\on{Sat}(G)^{\on{naive}}_{\Ran(X)}$ we thus obtain 
an action of the monoidal category $\Rep(\cG)_{\Ran(X)}$ on $\Dmod(\Bun_G^{B\on{-gen}})$. 

\ssec{Replacing $B$ by its unipotent radical} \label{ss:for N}

In what follows we shall need a few variants of the space $\Bun_G^{B\on{-gen}}$.

\sssec{}

First, we have the prestack $\Bun_G^{N\on{-gen}}$, defined in the same way as 
$\Bun_G^{B\on{-gen}}$, with $B$ replaced by $N$. By construction, we have a natural projection
$$\Bun_G^{N\on{-gen}}\to \Bun_G^{B\on{-gen}}.$$

\medskip

An analog of \propref{p:Barlev} holds with no modifications. The groupoid on $k$-points of $\Bun_G^{N\on{-gen}}$ 
is canonically isomorphic to the double quotient
$$N(K)\backslash G(\BA)/G(\BO).$$

\sssec{}

For any target scheme (or even prestack) $Y$, we define the prestack $\bMaps(X,Y)^{\on{gen}}$ in a way 
analogous to the definition of $\Bun_G^{B\on{-gen}}$. 

\medskip

Namely, the groupoid of $S$-points of $Y$ is the quotient of the set of pairs
$$(U\subset S\times X;y:U\to Y)$$ by the equivalence relation that identifies
$(U^1,y^1)$ with $(U^2,y^2)$ if $y_1|_{U_1\cap U_2}=y_2|_{U_1\cap U_2}$. 

\sssec{}

Consider in particular the group-object in $\on{PreStk}$ given by $\bMaps(X,T)^{\on{gen}}$.

\medskip

We have a natural action of $\bMaps(X,T)^{\on{gen}}$ on $\Bun_G^{N\on{-gen}}$, and the quotient
is easily seen to identify with $\Bun_G^{B\on{-gen}}$.

\sssec{}

We can rewrite the definition of $\Bun_G^{N\on{-gen}}$ as follows. We consider the prestack that assigns to $S\in \affSch$
the groupoid of quadruples 
$$(\CP_G,U,\alpha,\gamma),$$
where $(\CP_G,U,\alpha)$ are as in the definition of $\Bun_G^{B\on{-gen}}$, and $\gamma$ is
a datum of a trivialization of the $T$-bundle $\CP_{T,U}$, see \secref{sss:defn rat} for the notation. 

\medskip

The prestack $\Bun_G^{N\on{-gen}}$ is obtained from the above prestack of quadruples 
by quotienting it by the equivalence relation that identifies 
$(\CP^1_G,U^1,\alpha^1,\gamma^1)$ and $(\CP^2_G,U^2,\alpha^2,\gamma^2)$
if the corresponding points $(\CP^1_G,U^1,\alpha^1)$ and $(\CP^2_G,U^2,\alpha^2)$
are identified, and the resulting isomorphism between $\CP^1_{T,U}$ and $\CP^2_{T,U}$ 
over $U_1\cap U_2$ maps $\gamma^1$ to $\gamma^2$.

\sssec{}

From now on in the paper we are going to fix a square root $\omega^{\frac{1}{2}}_X$ of 
the canonical line bundle on $X$. In particular, we obtain a well-defined $T$-bundle
$$\check\rho(\omega_X):=2\check\rho(\omega^{\frac{1}{2}}_X).$$

\medskip

We define the prestack $\CQ_G:=\Bun_G^{N^\omega\on{-gen}}$ to be a twist of $\Bun_G^{N\on{-gen}}$.
Namely, in the data $(\CP_G,U,\alpha,\gamma)$ we change the meaning of $\gamma$:

\medskip

Instead of being a trivialization of $\CP_{T,U}$ we now let $\gamma$ be
the datum of an isomorphism with $\check\rho(\omega_X)|_U$. 

\medskip

A choice of a generic trivialization of $\omega^{\frac{1}{2}}_X$ identifies $\CQ_G$ with $\Bun_G^{N\on{-gen}}$,
and in particular, the groupoid of its $k$-points with the 
double quotient
$$N(K)\backslash G(\BA)/G(\BO).$$

\sssec{}

Yet another space that we will need is the quotient of $\Bun_{G}^{N\on{-gen}}$
(or, rather, $\Bun_G^{N^\omega\on{-gen}}$) by the action of
$$\bMaps(X,Z_G^0)^{\on{gen}},$$
where $Z_G^0$ is the connected component of the center of $G$. 

\medskip

We denote the resulting prestack by $\CQ_{G,G}$ (its variant $\CQ_{G,P}$,
where $P\subset G$ is a parabolic and $M$ is the Levi quotient of $P$, will be introduced in \secref{s:deg Whit}). 

\medskip

A choice of a generic trivialization of $\omega^{\frac{1}{2}}_X$ identifies the 
groupoid on $k$-points of $\CQ_{G,G}$ with the 
double quotient
$$Z_G^0(K) \cdot N(K)\backslash G(\BA)/G(\BO).$$

\sssec{}

The prestack $\CQ_{G,G}$ can be explicitly described as follows. 

\medskip

We consider
the prestack that assigns to $S\in \affSch$ the groupoid of quadruples
$$(\CP_G,U,\alpha,\gamma),$$
where $(\CP_G,U,\alpha)$ as above, and $\gamma$ is a datum of isomorphism
over $U$ between the bundles with respect to $T/Z_G^0$, one being
induced from $\CP_{T,U}$, and the other from $\check\rho(\omega_X)|_U$. 

\medskip

(Note that when $G$ has a connected center, the data of $\gamma$ amounts to an isomorphism
of line bundles $\alpha_i(\CP_{T,U})\simeq \omega_X|_U$ for each simple root $\alpha_i$ of $G$.
In particular, it is independent of the choice of $\omega^{\frac{1}{2}}_X$.) 

\medskip

The prestack $\CQ_{G,G}$ is obtained from the above prestack of quadruples
by quotienting by the equivalence relation, defined in the same way as in the case of 
$\Bun_G^{N^\omega\on{-gen}}$. 

\sssec{}

We claim:

\begin{prop}  \label{p:pull-back fully faithful}
The pull-back functors 
$$\Dmod(\Bun_G^{B\on{-gen}})\to \Dmod(\CQ_{G,G})\to \Dmod(\CQ_G)$$
are fully faithful.
\end{prop}

\begin{proof}
Follows from the homological contractibility of the prestacks $\bMaps(X,T)^{\on{gen}}$
and $\bMaps(X,Z_G^0)^{\on{gen}}$, see \cite{Ga2}. 
\end{proof}

\sssec{}

As in \secref{sss:Hecke action on BunBrat}, we have a canonical action of the monoidal category 
$\Rep(\cG)_{\Ran(X)}$ on both $\Dmod(\CQ_{G})$ and $\Dmod(\CQ_{G,G})$.

\ssec{The groupoid: function-theoretic analogy}    \label{ss:idea of groupoid}

In order to introduce the Whittaker category, as well as several other categories of primary
interest for this paper, we will need to define a certain groupoid, denoted $\bN$ that
acts on $\Bun_G^{B\on{-gen}}$ and related geometric objects. 

\sssec{}  

We will now explain the idea of the definition of this groupoid through a function-theoretic analogy.

\medskip

As was mentioned above, the category $\Dmod(\Bun_G^{N\on{-gen}})$
is the geometric analog of the space of functions on $N(K)\backslash G(\BA)/G(\BO)$. 
What we want to achieve is to enforce the condition that our function, when considered as a function on
$G(\BA)/G(\BO)$, be invariant with respect to all of $N(\BA)$
(resp., equivariant against a fixed character of $N(\BA)$, which is trivial on $N(K)$ and $N(\BO)$). However, we want to
do this without actually lifting our function on $G(\BA)/G(\BO)$. 

\medskip

Here is how we will do this. The trick explained below stands behind the definition of the corresponding 
verions of the Whittaker category in \cite{FGV1} and \cite{Ga1}. 

\sssec{}

Let $\ul{x}$ be a finite collection of points on $X$, and let 
$\BA_{\ul{x}}$ denote the corresponding product of local fields. Let us say that we want to 
enforce invariance/equivariance with respect to the corresponding subgroup $N(\BA_{\ul{x}})\subset N(\BA)$. 

\medskip

Let $\ul{y}$ be another finite collection of points of $X$, which is \emph{non-empty} and \emph{disjoint} from $\ol{x}$.
Let 
$$N(K)\backslash G(\BA)/G(\BO)_{\text{goot at}\,\ul{y}}\subset N(K)\backslash G(\BA)/G(\BO)$$
be the subset equal to
$$N(K)\backslash \left(G(\BA^{\ul{y}})/G(\BO^{\ul{y}})\times N(\BA_{\ul{y}})/N(\BO_{\ul{y}}) \right),$$
where\footnote{I.e., $N(K)\backslash G(\BA)/G(\BO)_{\text{goot at}\,\ul{y}}$ is different from all of 
$N(K)\backslash G(\BA)/G(\BO)$ is that for $z\in \ul{y}$ we take $N(\wh\CK_z)/N(\wh\CO_z)$ instead of $G(\wh\CK_z)/G(\wh\CO_z)$}
$\BA^{\ul{y}}:=\underset{z\notin \ul{y}}\Pi\, \wh\CK_z$,  $\BO^{\ul{y}}:=\underset{z\notin \ul{y}}\Pi\, \wh\CO_z$.

\medskip

Clearly, the preimage of the subset $N(K)\backslash G(\BA)/G(\BO)_{\text{goot at}\,\ul{y}}$
in $G(\BA)/G(\BO)$ is invariant with respect to $N(\BA_{\ul{x}})$.
Moreover, Iwasawa decomposition implies that all of $N(K)\backslash G(\BA)/G(\BO)$ 
can be covered by subsets of this form for various choices of $\ul{y}$.

\medskip

Hence, it is sufficient to specify the invariance/equivariance condition for a function on
$N(K)\backslash G(\BA)/G(\BO)_{\text{goot at}\,\ul{y}}$. 

\sssec{}

Set
$$^\sim N(K)\backslash G(\BA)/G(\BO)_{\text{goot at}\,\ul{y}}:= 
N(K)\backslash
\left(G(\BA^{\ul{y}})/G(\BO^{\ul{y}})\times N(\BA_{\ul{y}})\right).$$

This set is acted on (by right multiplication) by $N(\BA_{\ul{y}})$, and the resulting action of the subgroup 
$N(\BO_{\ul{y}})\subset N(\BA_{\ul{y}})$
makes the projection
$$^\sim N(K)\backslash G(\BA)/G(\BO)_{\text{goot at}\,\ul{y}}\to N(K)\backslash G(\BA)/G(\BO)_{\text{goot at}\,\ul{y}}$$
into a $N(\BO_{\ul{y}})$-torsor.

\sssec{}

We are now finally ready to explain how we will enforce the sought-for invariance/condition with respect to $N(\BA_{\ul{x}})$
for a function on $N(K)\backslash G(\BA)/G(\BO)_{\text{goot at}\,\ul{y}}$. 
In fact, we will enforce equivariance with respect to all of $N(\BA^{\ul{y}})$. 

\medskip

Namely, we require that the lift of our function to 
$^\sim N(K)\backslash G(\BA)/G(\BO)_{\text{goot at}\,\ul{y}}$ 
be $N(\BA_{\ul{y}})$-invariant/equivariant. 

\medskip

The fact that this is the right thing to do follows from the \emph{strong approximation} for the group $N$, i.e., from the fact that the image 
of the map 
$$N(K)\to N(\BA^{\ul{y}}),$$
given by Taylor expansion, is dense.

\ssec{The groupoid: algebro-geometric definition}

The actual algebro-geometric definition of the groupoid $\bN$, given below, was suggested by J.~Barlev. 

\sssec{}  \label{sss:defn of groupoid}

We define the groupoid $\bN$ is as follows. First, the space (=prestack) that it will act on is not $\Bun_G^{B\on{-gen}}$,
but rather a certain open substack 
$$\left(\Bun_G^{B\on{-gen}}\times \Ran(X)\right){}_{\on{good}}$$ 
of $\Bun_G^{B\on{-gen}}\times \Ran(X)$. 

\medskip

Namely, $\left(\Bun_G^{B\on{-gen}}\times \Ran(X)\right){}_{\on{good}}$ corresponds those quadruples $(\CP_G,U,\alpha,\gamma,\ul{y})$
for which $U$ can be chosen to contain $\ul{y}$. 

\medskip

We consider the prestack that assigns to $S\in \affSch$ the groupoid of the following data:
$$((\CP^1_G,U,\alpha^1),(\CP^2_G,U,\alpha^2),\ul{y},\beta),$$
where $(\CP^1_G,U,\alpha^1)$ and $(\CP^2_G,U,\alpha^2)$ are as in \secref{sss:defn rat}, $\ul{y}\in U$, 
and $\beta$ is a datum is isomorphism of $B$-bundles
$$\CP^1_{B,U}|_{U-\ul{y}}\simeq \CP^2_{B,U}|_{U-\ul{y}},$$
such that the induced isomorphism of $T$-bundles 
$$\CP^1_{T,U}|_{U-\ul{y}}\simeq \CP^2_{T,U}|_{U-\ul{y}}$$
extends to all of $U$, and such that the induced isomorphism of the $G$-bundles
$$\CP^1_G|_{U-\ul{y}}\simeq \CP^2_G|_{U-\ul{y}}$$
extends to all of $S\times X-\ul{y}$. 

\medskip

We let $\Maps(S,\bN)$ be the quotient of the above prestack
by the equivalence relation defined in a way similar to the case of $\Bun_G^{B\on{-gen}}$. 

\medskip

We have the following assertion:

\begin{prop}  \label{p:groupoid contractible}
The fibers of the groupoid $\bN$ are homologically 
contractible, i.e., the functors 
$$p_1^\dagger,p_2^\dagger:\Dmod\left(\left(\Bun_G^{B\on{-gen}}\times \Ran(X)\right){}_{\on{good}}\right)\to \Dmod(\bN)$$
are fully faithful, where $p_1$ and $p_2$ are the two projections 
$\bN\to \left(\Bun_G^{B\on{-gen}}\times \Ran(X)\right){}_{\on{good}}$. 
\end{prop}

This proposition essentially follows from the fact that the group $N$ is homologically contractible.

\sssec{}

The groupoid $\bN_{\CQ_G}$ (resp., $\bN_{\CQ_{G,G}}$) acting on $\CQ_G$ (resp., $\CQ_{G,G}$)
is defined similarly. 

\medskip

Note that we have a Cartesian diagram
$$
\CD
\left(\CQ_G\times \Ran(X)\right){}_{\on{good}}   @<{p_1}<<  \bN_{\CQ_G}  @>{p_2}>>  \left(\CQ_G\times \Ran(X)\right){}_{\on{good}}   \\
@VVV  @VVV   @VVV   \\
\left(\CQ_{G,G}\times \Ran(X)\right){}_{\on{good}}   @<{p_1}<<  \bN_{\CQ_{G,G}}  @>{p_2}>>  \left(\CQ_{G,G}\times \Ran(X)\right){}_{\on{good}}    \\
@VVV   @VVV   @VVV   \\
\left(\Bun_G^{B\on{-gen}}\times \Ran(X)\right){}_{\on{good}}  @<{p_1}<<  \bN @>{p_2}>>  \left(\Bun_G^{B\on{-gen}}\times \Ran(X)\right){}_{\on{good}}.
\endCD
$$

From \propref{p:groupoid contractible} we obtain the corresponding assertion for
$\bN_{\CQ_G}$ (resp., $\bN_{\CQ_{G,G}}$).



\ssec{The character}    \label{ss:char}

\sssec{}

We now consider the groupoid $\bN_{\CQ_{G,G}}$ and we claim that it admits
a canonically defined homomorphism $\chi$ to $\BG_a$. 

\medskip

In fact, there are homomorphisms $\chi_i$, one for each simple root $\alpha_i$ of $G$, and we will
let $\chi$ be their sum. 

\sssec{}

For a simple root $\alpha_i$, let 
$B_i\simeq \BG_m\ltimes \BG_a$ be the quotient group of $B$ by $N(P_i)$ (the unipotent radical
of the sub-minimal parabolic $P_i$) and $Z_{M_i}$ (the center of the Levi $M_i$ of $P_i$). 
In particular, the map $T\to \BG_m$ is given by the simple root $\alpha_i$.

\medskip

For a point $(\CP_G,U,\alpha,\gamma)$ of $\CQ_{G,G}$ we let $\CP_{B_i,U}$ denote the
induced $B_i$-bundle defined over $U$. Note that the data of $\gamma$ identifies 
the line bundle corresponding to $B_i\to \BG_m$ with $\omega_X|_U$. Hence, we can think
of $\CP_{B_i,U}$ as a short exact sequence of vector bundles
$$0\to \omega_U \to \CF_i\to \CO_{U}\to 0.$$

\sssec{}

A point $$((\CP^1_G,U,\alpha^1,\gamma^1),(\CP^2_G,U,\alpha^2,\gamma^2),\ul{y},\beta)$$
of $\bN_{\CQ_{G,G}}$ defines an isomorphism of short exact sequences
$$
\CD
0 @>>> \omega_X|_{U-\ul{u}} @>>>  \CF^1_i|_{U-\ul{y}}   @>>>  \CO_{U-\ul{y}}  @>>> 0 \\
& & @V\on{id}VV   @V{\beta_i}VV  @VV\on{id}V  & & \\
0 @>>> \omega_X|_{U-\ul{y}} @>>>  \CF^2_i|_{U-\ul{y}}   @>>>  \CO_{U-\ul{y}} @>>> 0,
\endCD
$$
and hence a section of the quasi-coherent sheaf
\begin{equation} \label{e:UV quotient}
\omega_X|_{U-\ul{y}}/\omega_X|_{U}\simeq \omega_X|_{S\times X-\ul{y}}/\omega_X|_{S\times X}\simeq
(\CO_S\boxtimes \omega_X)(\infty\cdot \ul{y})/(\CO_S\boxtimes \omega_X),
\end{equation}
where we think of $\ul{y}$ as a relative Cartier divisor $D\subset S\times X$ over $S$. 

\medskip

Now, the residue map assigns to sections of \eqref{e:UV quotient} a section of $\CO_S$,
i.e., a map $S\to \BG_a$.

\sssec{}

By composing, the above character $\chi$ on $\bN_{\CQ_{G,G}}$ gives rise to one
on $\bN_{\CQ_G}$. We will not distinguish the two notationally. 

\ssec{The Whittaker category}  \label{ss:Whit}

We are finally able to define the main actor for this section, the Whittaker category for $G$.

\sssec{}

First, we consider the equivariant category 
$$\Dmod\left(\left(\CQ_G\times \Ran(X)\right){}_{\on{good}}\right)^{\bN_{\CQ_G},\chi}$$
of $\Dmod\left(\left(\CQ_G\times \Ran(X)\right){}_{\on{good}}\right)$ with respect to the
groupoid $\bN_{\CQ_G}$ against the character $\chi$.

\medskip

In other words, we consider the simplicial object $\bN_{\CQ_G}^{\bDelta}$ of $\on{PreStk}$ corresponding
to the groupoid $\bN_{\CQ_G}$.  We consider the co-simplicial category $\Dmod(\bN_{\CQ_G}^{\bDelta})$, and its twist, denoted 
$$\Dmod(\bN_{\CQ_G}^{\bDelta})^\chi,$$ corresponding to the pull-back by means of $\chi$ of the 
exponential D-module on $\BG_a$. By definition,
$$\Dmod\left(\left(\CQ_G\times \Ran(X)\right){}_{\on{good}}\right)^{\bN_{\CQ_G},\chi}:=
\on{Tot}(\Dmod(\bN_{\CQ_G}^{\bDelta})^\chi).$$

\medskip

The following results from \propref{p:groupoid contractible}: 

\begin{prop}  \label{p:embed Whit ff Ran}
The forgetful functor
$$\Dmod\left(\left(\CQ_G\times \Ran(X)\right){}_{\on{good}}\right)^{\bN_{\CQ_G},\chi}\to
\Dmod\left(\left(\CQ_G\times \Ran(X)\right){}_{\on{good}}\right)$$
is fully faithful.
\end{prop}

\sssec{}

We define the Whittaker category $\on{Whit}(G)$ to be the full subcategory of $\Dmod(\CQ_G)$
equal to the preimage of
$$\Dmod\left(\left(\CQ_G\times \Ran(X)\right){}_{\on{good}}\right)^{\bN_{\CQ_G},\chi}\subset 
\Dmod\left(\left(\CQ_G\times \Ran(X)\right){}_{\on{good}}\right)$$
under the pull-back functor
$$\Dmod(\CQ_G)\to \Dmod\left(\left(\CQ_G\times \Ran(X)\right){}_{\on{good}}\right).$$

\medskip

In other words,

$$\on{Whit}(G):=\Dmod(\CQ_G)\underset{\Dmod\left(\left(\CQ_G\times \Ran(X)\right){}_{\on{good}}\right)}
\times\Dmod\left(\left(\CQ_G\times \Ran(X)\right){}_{\on{good}}\right)^{\bN_{\CQ_G},\chi}.$$

\sssec{}  \label{sss:whit techn}

Consider the fully faithful embedding
$$\on{Whit}(G)\hookrightarrow \Dmod(\CQ_G).$$

One shows that it admits a right adjoint; we will denote it by $\on{Av}^{\bN,\chi}$.  

\medskip

In addition, one shows:

\begin{prop} \label{p:Hecke preserves Whit}
The action of the monoidal category 
$\Rep(\cG)_{\Ran(X)}$ on $\Dmod(\CQ_{G})$ preserves 
the full subcategory
$$\on{Whit}(G)\subset \Dmod(\CQ_{G})$$
and commutes with the functor $\on{Av}^{\bN,\chi}$.
\end{prop}

\sssec{}  \label{sss:vacuum Whit}

The category $\on{Whit}(G)$ contains a distinguished object that we shall denote by $\CW_{\on{vac}}$:

\medskip

Analogously to the map $\imath_B:\Bun_B\to \Bun_G^{B\on{-gen}}$, we have a canonically defined map
$$\imath_{N^\omega}:\Bun_{N^\omega}\to \Bun_G^{N\on{-gen}}.$$

Note that, analogously to \secref{ss:char}, there exists a canonically defined map
$$\Bun_{N^\omega}\to \BG_a.$$

Let $\overset{\circ}\CW_{\on{vac}}\in \Dmod(\Bun_{N^\omega})$ denote the pullback
of the exponential D-module on $\BG_a$ under this map. 

\medskip

The object $\CW_{\on{vac}}\in \Dmod(\Bun_G^{N\on{-gen}})$ is defined by
$$\CW_{\on{vac}}:=(\imath_{N^\omega})_\dagger(\overset{\circ}\CW_{\on{vac}}),$$
where for a morphism $f$ we denote by $f_\dagger$ the (partially defined) left adjoint of $f^\dagger$;
one shows that the (partially defined) functor $(\imath_{N^\omega})_\dagger$ \emph{is} defined
on the object $\overset{\circ}\CW_{\on{vac}}$ due to the holonomicity property of the latter. 

\sssec{A variant}

We define the category $\on{Whit}(G,G)$ in a similar way to $\on{Whit}(G)$, using the prestack $\CQ_{G,G}$ instead of $\CQ_G$. 

\medskip

As in \secref{sss:whit techn}, the fully faithful embedding
$$\on{Whit}(G)\hookrightarrow \Dmod(\CQ_G).$$
it admits a right adjoint, and the analog of \propref{p:Hecke preserves Whit} holds. 

\medskip

Let us remind that the categories $\on{Whit}(G)$ and $\on{Whit}(G,G)$ are geometric counterparts of the spaces of functions on
$G(\BA)/G(\BO)$ and $Z^0_G(K)\backslash G(\BA)/G(\BO)$, respectively, that are equivariant with respect to $N(\BA)$
against the character $\chi$. In particular, we have a naturally defined pullback functor 
$$\on{Whit}(G,G)\to \on{Whit}(G).$$

Now, it follows from \propref{p:pull-back fully faithful} that this functor is fully faithful.

\ssec{The functor of Whittaker coefficient and Poincar\'e series}  \label{ss:Whit coeff}

In this subsection we will relate the Whittaker categories $\on{Whit}(G)$ and $\on{Whit}(G,G)$ to
the main object on the geometric side, the category $\Dmod(\Bun_G)$.

\sssec{}

Let $\fr_G$ (resp., $\fr_{G,G}$) denote the forgetful map $\CQ_G\to \Bun_G$
(resp., $\CQ_{G,G}\to \Bun_G$). In particular, we obtain the functors
$$(\fr_G)^\dagger:\Dmod(\Bun_G)\to \Dmod(\CQ_G) \text{ and }
(\fr_{G,G})^\dagger:\Dmod(\Bun_G)\to \Dmod(\CQ_{G,G}).$$

\sssec{}

We denote the composed functors 
$$\on{Av}^{\bN,\chi}\circ (\fr_G)^\dagger:\Dmod(\Bun_G)\to \on{Whit}(G)$$ and 
$$\on{Av}^{\bN,\chi}\circ (\fr_{G,G})^\dagger:\Dmod(\Bun_G)\to \on{Whit}(G,G)$$
by $\on{coeff}_G$ and $\on{coeff}_{G,G}$, respectively. 

\medskip

These are the two closely related versions
of the \emph{functor of Whittaker coefficient}. 

\sssec{}

By \propref{p:Hecke preserves Whit}, the functor $\on{coeff}_G$ (resp., $\on{coeff}_{G,G}$)
intertwines the actions of $\Rep(\cG)_{\Ran(X)}$ on $\Dmod(\Bun_G)$ and 
$\on{Whit}(G)$ (resp., $\on{Whit}(G,G)$).

\sssec{}   \label{sss:1st Whittaker coefficient}

The functor $(\fr_G)^\dagger:\Dmod(\Bun_G)\to \Dmod(\CQ_G)$ does not in general admit a left adjoint. 
However, one shows that the (partially defined) left adjoint $(\fr_G)_\dagger$ \emph{is defined} on the full subcategory
$\on{Whit}(G)\subset \Dmod(\CQ_G)$.

\medskip

We denote the resulting functor $\on{Whit}(G)\to \Dmod(\Bun_G)$ by $\Poinc_G$, and refer to it as the functor
of Poincar\'e series. By construction, this functor is the left adjoint of the functor $\on{coeff}_G$.

\medskip

In particular, we obatin a canonically defined object
$$\Poinc_G(\CW_{\on{vac}})\in \Dmod(\Bun_G),$$ 
where $\CW_{\on{vac}}$ is as in \secref{sss:vacuum Whit}.

\ssec{Digression: ``unital" categories over the Ran space}

Our next goal is to give a description of the categories $\on{Whit}(G)$ and $\on{Whit}(G,G)$ in spectral terms. 
This subsection contains some preliminaries needed in order to describe the spectral side.  

\medskip

These preliminaries
have to do with the fact that the symmetric monoidal categories $\Dmod(\Ran(X))$ and $\Rep(\cG)_{\Ran(X)}$
are non-unital, and in this subsection we will show how to modify them to make them unital. \footnote{The reader
who is afraid of being overwhelmed by the notation can skip this subsection and return to it when necessary.} 

\sssec{}

We note that a group homomorphism $G_1\to G_2$ gives rise to a symmetric monoidal
functor
$$\Rep(G_2)_{\Ran(X)}\to \Rep(G_1)_{\Ran(X)}.$$

In particular, taking $\cG_1=\cG$ and $\cG_2=\{1\}$, we obtain a symmetric monoidal functor
$$\Dmod(\Ran(X))\simeq \QCoh(\Ran(X))\to \Rep(\cG)_{\Ran(X)}.$$

\medskip

By taking $G_1=\cG$ and $G_2=\cG/[\cG,\cG]$ we obtain a symmetric monoidal functor
$$\Rep(\cG/[\cG,\cG])_{\Ran(X)}\to \Rep(\cG)_{\Ran(X)}.$$

\sssec{}  \label{sss:unit}

Consider the symmetric monoidal functor 
$$\Dmod(\Ran(X))\overset{p_\dagger}\longrightarrow \Dmod(\on{pt})=\Vect$$ 
(the left adjoint to the pull-back functor $p^\dagger$). It can also be viewed as the functor 
$\on{Loc}_{\{1\},\on{spec}}$, where $\{1\}$ is the trivial group. Consider the category
$$\Rep(\cG)_{\Ran(X)}^{\on{unital}}:= \Rep(\cG)_{\Ran(X)}\underset{\Dmod(\Ran(X))}\otimes \Vect.$$

One can show that the functor
$$\Vect\simeq  \Dmod(\Ran(X))\underset{\Dmod(\Ran(X))}\otimes \Vect\to 
\Rep(\cG)_{\Ran(X)}\underset{\Dmod(\Ran(X))}\otimes \Vect=:\Rep(\cG)_{\Ran(X)}^{\on{unital}}$$
defines a \emph{unit} for the symmetric monoidal structure on $\Rep(\cG)_{\Ran(X)}^{\on{unital}}$; we
shall denote the corresponding unit object by 
$${\bf 1}_{\Rep(\cG)_{\Ran(X)}^{\on{unital}}}\in \Rep(\cG)_{\Ran(X)}^{\on{unital}}.$$

I.e., unlike $\Rep(\cG)_{\Ran(X)}$, the symmetric monodical category $\Rep(\cG)_{\Ran(X)}^{\on{unital}}$
is unital.

\sssec{}

It follows from the definition that the symmetric monoidal functor
$$\on{Loc}_{\cG,\on{spec}}:\Rep(\cG)_{\Ran(X)}\to \QCoh(\LocSys_\cG)$$
canonically factors as
\begin{equation} \label{e:factoring locs 1}
\Rep(\cG)_{\Ran(X)}\to \Rep(\cG)_{\Ran(X)}^{\on{unital}} \to
\QCoh(\LocSys_\cG).
\end{equation}

We denote the resulting functor
$$\Rep(\cG)_{\Ran(X)}^{\on{unital}} \to  \QCoh(\LocSys_\cG)$$
by $\on{Loc}^{\on{unital}}_{\cG,\on{spec}}$.

\medskip

Passing to right adjoints in \eqref{e:factoring locs 2}, we obtain the functors
$$\QCoh(\LocSys_\cG)\to  \Rep(\cG)_{\Ran(X)}^{\on{unital}} \to \Rep(\cG)_{\Ran(X)},$$
all of which are fully faithful by  \propref{p:local to global LocSys}.

\medskip

We shall denote the resulting (fully faithful) functor
$$\QCoh(\LocSys_\cG)\to  \Rep(\cG)_{\Ran(X)}^{\on{unital}}$$
by $\on{co-Loc}^{\on{unital}}_{\cG,\on{spec}}$.

\sssec{Variant}

Consider now the symmetric monoidal functor
$$\on{Loc}_{\cG/[\cG,\cG],\on{spec}}:\Rep(\cG/[\cG,\cG])_{\Ran(X)}\to \QCoh(\LocSys_{\cG/[\cG,\cG]}).$$

Consider also the category 
$$\Rep(\cG)_{\Ran(X)}^{\on{unital}_{\cG/[\cG,\cG]}}:=\Rep(\cG)_{\Ran(X)}\underset{\Rep(\cG/[\cG,\cG])_{\Ran(X)}}\otimes 
\QCoh(\LocSys_{\cG/[\cG,\cG]}),$$
and the functor
\begin{multline}  \label{e:rel between 1 and Z}
\Rep(\cG)_{\Ran(X)}^{\on{unital}} \simeq \\
\simeq \Rep(\cG)_{\Ran(X)}\underset{\Rep(\cG/[\cG,\cG])_{\Ran(X)}}\otimes  \Rep(\cG/[\cG,\cG])_{\Ran(X)}
\underset{\Dmod(\Ran(X))}\otimes \Vect \overset{\on{Id}\otimes \on{Loc}^{\on{unital}}_{\cG/[\cG,\cG],\on{spec}}}\longrightarrow \\
\to \Rep(\cG)_{\Ran(X)}\underset{\Rep(\cG/[\cG,\cG])_{\Ran(X)}}\otimes \QCoh(\LocSys_{\cG/[\cG,\cG]})=:
\Rep(\cG)_{\Ran(X)}^{\on{unital}_{\cG/[\cG,\cG]}}.
\end{multline}

\medskip

It follows from the construction that the functor
$$\on{Loc}^{\on{unital}}_{\cG,\on{spec}}:\Rep(\cG)_{\Ran(X)}^{\on{unital}}\to \QCoh(\LocSys_\cG)$$
introduced above canonically factors as
\begin{equation} \label{e:factoring locs 2}
\Rep(\cG)_{\Ran(X)}^{\on{unital}}
\overset{\text{\eqref{e:rel between 1 and Z}}}\longrightarrow 
\Rep(\cG)_{\Ran(X)}^{\on{unital}_{\cG/[\cG,\cG]}}
 \to \QCoh(\LocSys_\cG),
\end{equation}

We denote the resulting functor
$$\Rep(\cG)_{\Ran(X)}^{\on{unital}_{\cG/[\cG,\cG]}}\to \QCoh(\LocSys_\cG)$$
by $\on{Loc}^{\on{unital}_{\cG/[\cG,\cG]}}_{\cG,\on{spec}}$.

\medskip

Passing to right adjoints in \eqref{e:factoring locs 2}, we obtain functors
\begin{equation} 
\QCoh(\LocSys_\cG)\to \Rep(\cG)_{\Ran(X)}^{\on{unital}_{\cG/[\cG,\cG]}}\to 
\Rep(\cG)_{\Ran(X)}^{\on{unital}}
\end{equation}
all of which are fully faithful by \propref{p:local to global LocSys}.

\medskip

We denote the resulting (fully faithful) functor
$$\QCoh(\LocSys_\cG)\to \Rep(\cG)_{\Ran(X)}^{\on{unital}_{\cG/[\cG,\cG]}}$$
by $\on{co-Loc}^{\on{unital}_{\cG/[\cG,\cG]}}_{\cG,\on{spec}}$.

\ssec{Description of the Whittaker category in spectral terms}

A key feature of the Whittaker categories $\on{Whit}(G)$ and $\on{Whit}(G,G)$, and the reason for
why the figure so prominently in geometric Langlands, is that these categories can be directly described
in terms of the spectral side of the correspondence. 

\sssec{}

The following assertion is a geometric version of the Casselman-Shalika formula. 
It expresses the categories $\on{Whit}(G)$ and $\on{Whit}(G,G)$, respectively,
in terms of the Langlands dual group. 

\begin{qthm} \label{t:Cass Shal}   \noindent

\smallskip

\noindent{\em(a)} 
There exists a canonical equivalence
$$\BL^{\on{Whit}}_G:\Rep(\cG)_{\Ran(X)}^{\on{unital}}\to \on{Whit}(G),$$
compatible with the action of the monoidal category $\Rep(\cG)_{\Ran(X)}$. 

\smallskip

\noindent{\em(b)} There is a canonical equivalance
$$\Rep(\cG)_{\Ran(X)}^{\on{unital}_{\cG/[\cG,\cG]}}\to\on{Whit}(G,G),$$
compatible with the action of the monoidal category $\Rep(\cG)_{\Ran(X)}$. 

\smallskip

\noindent{\em(c)} We have a commutative diagram
\begin{equation} \label{e:two coLocs}
\CD
\Rep(\cG)_{\Ran(X)}^{\on{unital}} @>{\BL^{\on{Whit}}_G}>>  \on{Whit}(G)  \\
@AAA     @AAA   \\
\Rep(\cG)_{\Ran(X)}^{\on{unital}_{\cG/[\cG,\cG]}}
@>>>  \on{Whit}(G,G),
\endCD
\end{equation}
where the left vertical arrow is the right adjoint of \eqref{e:rel between 1 and Z}.

\end{qthm}

This quasi-theorem is very close to being a theorem and is being worked out by D.~Beraldo. 

\medskip

We shall denote the composed functor
$$\QCoh(\LocSys_\cG)\overset{\on{co-Loc}^{\on{unital}_{\cG/[\cG,\cG]}}_{\cG,\on{spec}}}\longrightarrow   
\Rep(\cG)_{\Ran(X)}^{\on{unital}_{\cG/[\cG,\cG]}}\overset{\sim}\to \on{Whit}(G,G)$$
by $\BL^{\on{Whit}}_{G,G}$. By the above, $\BL^{\on{Whit}}_{G,G}$ is fully faithful.

\sssec{}

We will now formulate Property $\on{Wh}$ (``Wh" stands for Whittaker) of the geometric Langlands functor $\BL_G$.
It is a particular case of Property $\on{Wh}^{\on{ext}}$, formulated in \secref{sss:cond W ext}:

\medskip

\noindent{\bf Property Wh:} {\it We shall say that the functor $\BL_G$ satisfies \emph{Property $\on{Wh}$} if 
the following diagram is commutative:
\begin{equation} \label{e:cond W}
\CD
\Rep(\cG)_{\Ran(X)}^{\on{unital}}  @>{\BL^{\on{Whit}}_{G}}>>  \on{Whit}(G)  \\
@A{\on{co-Loc}^{\on{unital}}_{\cG,\on{spec}}\circ \Psi_\cG}AA    @AA{\on{coeff}_{G}}A   \\
\IndCoh_{\on{Nilp}^{\on{glob}}_\cG}(\LocSys_\cG)   @>{\BL_G}>>  \Dmod(\Bun_G).
\endCD
\end{equation}}

We remind that the functor $\Psi_\cG$ appearing in the left vertical arrow in \eqref{e:cond W} is the right adjoint
of the fully faithful embedding
$$\QCoh(\LocSys_\cG)\overset{\Xi_\cG} \longrightarrow \IndCoh_{\on{Nilp}^{\on{glob}}_\cG}(\LocSys_\cG).$$

\sssec{}  \label{sss:1st Whittaker coefficient and Langlands}

By passing to left adjoints in the diagram \eqref{e:cond W}, from Property $\on{Wh}$ we obtain a commutative diagram
\begin{equation} \label{e:cond Poinc}
\CD
\Rep(\cG)_{\Ran(X)}^{\on{unital}}  @>{\BL^{\on{Whit}}_{G}}>>  \on{Whit}(G)  \\
@V{\Xi_\cG\circ \on{Loc}^{\on{unital}}_{\cG,\on{spec}}}VV    @VV{\Poinc_{G}}V   \\
\IndCoh_{\on{Nilp}^{\on{glob}}_\cG}(\LocSys_\cG)   @>{\BL_G}>>  \Dmod(\Bun_G).
\endCD
\end{equation}

As part of the construction of the equivalence of Quasi-Theorem \ref{t:Cass Shal}, we have that the object
$\CW_{\on{vac}}\in \on{Whit}(G)$ identifies with
$$\BL^{\on{Whit}}_{G}({\bf 1}_{\Rep(\cG)^{\on{unital}}_{\Ran(X)}}),$$
where we remind that ${\bf 1}_{\Rep(\cG)^{\on{unital}}_{\Ran(X)}}$ is the unit object in the monoidal category $\Rep(\cG)^{\on{unital}}_{\Ran(X)}$, 
see \secref{sss:unit}. 

\sssec{}

In particular, from \eqref{e:cond Poinc}, we obtain:
\begin{equation} \label{e:identify structure sheaf}
\BL_G(\Xi_\cG(\CO_{\LocSys_\cG}))\simeq \Poinc_G(\CW_{\on{vac}}).
\end{equation} 

So, the object on the geometric side that corresponds to 
$$\Xi_\cG(\CO_{\LocSys_\cG})\in \IndCoh_{\on{Nilp}^{\on{glob}}_\cG}(\LocSys_\cG)$$
is
$$\Poinc_G(\CW_{\on{vac}})\in \Dmod(\Bun_G).$$

\sssec{}

Note that since the vertical arrows in the diagram \eqref{e:two coLocs} are fully faithful, we can reformulate Property $\on{Wh}$ 
as the commutativity of the diagram
\begin{equation} \label{e:cond W modified}
\CD
\QCoh(\LocSys_\cG) @>{\BL^{\on{Whit}}_{G,G}}>>  \on{Whit}(G,G)  \\
@A{\Psi_\cG}AA    @AA{\on{coeff}_{G,G}}A   \\
\IndCoh_{\on{Nilp}^{\on{glob}}_\cG}(\LocSys_\cG)   @>{\BL_G}>>  \Dmod(\Bun_G).
\endCD
\end{equation} 

\begin{rem}
Note that if we believe in \conjref{c:GL}, we obtain a commutative diagram
$$
\CD
\QCoh(\LocSys_\cG)  @>{\BL^{\on{Whit}}_{G,G}}>>  \on{Whit}(G,G)  \\
@A{\Psi_\cG}AA    @AA{\on{coeff}_{G,G}}A   \\
\IndCoh_{\on{Nilp}^{\on{glob}}_\cG}(\LocSys_\cG)   @>{\BL_G}>>  \Dmod(\Bun_G) \\
@A{\Xi_\cG}AA   @AAA  \\
\QCoh(\LocSys_\cG)   @>>>  \Dmod(\Bun_G)_{\on{temp}},
\endCD
$$
where the composed left vertical arrow is the identity functor. Hence, 
the composed functor
$$\Dmod(\Bun_G)_{\on{temp}}\hookrightarrow \Dmod(\Bun_G)\overset{\on{coeff}_{G,G}}\to  \on{Whit}(G,G)$$
is fully faithful. I.e., the tempered category is Whittaker non-degenerate in the strong sense that not only does
the functor $\on{coeff}_G$ not annihilate anything, but it is actually fully faithful.
\end{rem}

\section{Parabolic induction}  \label{s:parabolic}

In this subsection we study how the automorphic category $\Dmod(\Bun_G)$ can be related to the
corresponding categories for proper Levi subgroups of $G$, and a similar phenomenon on the spectral
side of Langlands correspondence.

\ssec{The space of generic parabolic reductions}

In this subsection we will introduce the ``parabolically induced" category, denoted $\on{I}(G,P)$. 

\sssec{}

Let $P\subset G$ be a parabolic. Let $M$ denote its Levi quotient. 

\medskip

We define the prestack $\Bun_G^{P\on{-gen}}$ in the
same way as we defined $\Bun_G^{B\on{-gen}}$, substituting $P$ for $B$. 

\medskip

We let $\sfp^{\on{enh}}_P$ denote the natural forgetful map $\Bun_G^{P\on{-gen}}\to \Bun_G$, and by
$\imath_P$ the map 
$$\Bun_P\to \Bun_G^{P\on{-gen}}.$$

As in the case of $P=B$, the map $\imath_P$ defines an isomorphism at the level of groupoids
of field-valued points. In particular, the groupoid of $k$-points of $\Bun_G^{P\on{-gen}}$ identifies
canonically with the double quotient
$$P(K)\backslash G(\BA)/G(\BO).$$

\medskip

From here one deduces: 

\begin{lem}   \label{l:restr to strata}
The forgetful functor
$(\imath_P)^\dagger:\Dmod(\Bun_G^{P\on{-gen}})\to \Dmod(\Bun_P)$
is conservative.
\end{lem}

\sssec{}

Let $N(P)$ denote the unipotent radical of $P$. To it we associate a groupoid that we denote by
$\bNP$ acting on $\left(\Bun_G^{P\on{-gen}}\times \Ran(X)\right){}_{\on{good}}$ in the same way as we defined the groupoid $\bN$
acting on $\left(\Bun_G^{B\on{-gen}}\times \Ran(X)\right){}_{\on{good}}$.

\medskip

We consider the $\bNP$-equivariant category of $\left(\Bun_G^{P\on{-gen}}\times \Ran(X)\right){}_{\on{good}}$, i.e.,
$$\Dmod\left(\left(\Bun_G^{P\on{-gen}}\times \Ran(X)\right){}_{\on{good}}\right)^{\bNP}:=
\on{Tot}\left(\Dmod(\bNP^{\bDelta})\right),$$
where $\bNP^{\bDelta}$ is the simplicial object of $\on{PreStk}$ corresponding to the groupoid $\bNP$. 

\medskip

As in \propref{p:embed Whit ff Ran}, we have:

\begin{prop}
The forgetful functor 
$$\Dmod\left(\left(\Bun_G^{P\on{-gen}}\times \Ran(X)\right){}_{\on{good}}\right)^{\bNP}\to 
\Dmod\left(\left(\Bun_G^{P\on{-gen}}\times \Ran(X)\right){}_{\on{good}}\right)$$
is fully faithful.
\end{prop}

\sssec{}

We define $\on{I}(G,P)$ as the full subcategory of $\Dmod(\Bun_G^{P\on{-gen}})$ equal to the preimage
of 
$$\Dmod\left(\left(\Bun_G^{P\on{-gen}}\times \Ran(X)\right){}_{\on{good}}\right)^{\bNP}\subset 
\Dmod\left(\left(\Bun_G^{P\on{-gen}}\times \Ran(X)\right){}_{\on{good}}\right)$$
under the pull-back functor
$$\Dmod(\Bun_G^{P\on{-gen}})\to \Dmod\left(\left(\Bun_G^{P\on{-gen}}\times \Ran(X)\right){}_{\on{good}}\right).$$

I.e., 
\begin{multline*}
\on{I}(G,P):= \\
=\Dmod(\Bun_G^{P\on{-gen}}) \underset{\Dmod\left(\left(\Bun_G^{P\on{-gen}}\times \Ran(X)\right){}_{\on{good}}\right)}\times
\Dmod\left(\left(\Bun_G^{P\on{-gen}}\times \Ran(X)\right){}_{\on{good}}\right)^{\bNP}.
\end{multline*}

\begin{rem}
The category $\on{I}(G,P)$ is the geometric counterpart of the space of functions on the double quotient
$$M(K)\cdot N(P)(\BA)\backslash G(\BA)/G(\BO).$$
\end{rem}

\sssec{}  \label{sss:Hecke action on PS}

As in the case of $\Whit(G)$, one shows that the fully faithful embedding
$$\on{I}(G,P)\subset \Dmod(\Bun_G^{P\on{-gen}})$$ 
admits a right adjoint, that we denote by $\on{Av}^{\bNP}$. 

\medskip

As in \secref{sss:Hecke action on BunBrat}, we have a canonically defined action of the monoidal
category $\Rep(\cG)_{\Ran(X)}$ on $\Dmod(\Bun_G^{P\on{-gen}})$, and as in \propref{p:Hecke preserves Whit},
this action preserves the full subcategory 
$$\on{I}(G,P)\subset \Dmod(\Bun_G^{P\on{-gen}})$$
and commutes with the functor $\on{Av}^{\bNP}$. 

\ssec{A strata-wise description of the parabolic category}

One can describe the full subcategory $\on{I}(G,P)\subset \Dmod(\Bun_G^{P\on{-gen}})$ explicitly via the
morphism
$$\iota_P:\Bun_P\to \Bun_G^{P\on{-gen}}.$$

This is the subject of the present subsection. 

\sssec{}

Let $\sfp_P$ and $\sfq_P$ denote the natural forgetful maps from $\Bun_P$
to $\Bun_G$ and $\Bun_M$, respectively. For instance, we have:
$$\sfp_P=\sfp^{\on{enh}}_P\circ \imath_P.$$

Note that the map $\sfq_P$ is smooth. Hence, the functor
$$(\sfq_P)^\bullet:\Dmod(\Bun_M)\to \Dmod(\Bun_P)$$
(the Verdier conjugate of $(\sfq_P)^\dagger$) is well-defined \footnote{Since $\sfq_P$ is smooth, the functors $(\sfq_P)^\bullet$ and
$(\sfq_P)^\dagger$ are in fact isomorphic up to a cohomological shift, which depends on the connected
component of $\Bun_M$.}. Note that the fibers of $\sfq_P$ are contractible, so the functor $(\sfq_P)^\bullet$
is fully faithful.

\sssec{}

We have:

\begin{lem}  \label{l:equiv strata-wise}
The category $\on{I}(G,P)$ fits into a pull-back square:
\begin{equation}  \label{e:equiv strata-wise}
\CD
\on{I}(G,P)  @>>> \Dmod(\Bun_G^{P\on{-gen}}) \\
@VVV @VV{(\imath_P)^\dagger}V   \\
\Dmod(\Bun_M)  @>{(\sfq_P)^\bullet}>>  \Dmod(\Bun_P).
\endCD
\end{equation} 
\end{lem}

In other words, the above lemma says that an object $\CM\in \Dmod(\Bun_G^{P\on{-gen}})$
belongs to $\on{I}(G,P)$ if and only if $(\imath_P)^\dagger(\CM)\in \Dmod(\Bun_P)$ belongs
to the essential image of the functor $(\sfq_P)^\bullet$. 

\sssec{}  \label{sss:restr PS on strata}

We denote the resulting (conservative) functor $\on{I}(G,P) \to \Dmod(\Bun_M)$ by 
$(\imath_M)^\dagger$.

\medskip

One shows that the square obtained by passing to right adjoints along the
horizontal arrows in \eqref{e:equiv strata-wise} is also commutative:

\begin{equation} \label{e:av strata-wise}
\CD
\on{I}(G,P)  @<{\on{Av}^{\bNP}}<< \Dmod(\Bun_G^{P\on{-gen}}) \\
@V{(\imath_M)^\dagger}VV @VV{(\imath_P)^\dagger}V   \\
\Dmod(\Bun_M)  @<{(\sfq_P)_\bullet}<<  \Dmod(\Bun_P).
\endCD
\end{equation}

\sssec{}

In addition, one shows that the partially defined left adjoint $(\imath_P)_\dagger$ of $(\imath_P)^\dagger$ 
is defined on the essential image of $(\sfq_P)^\bullet$. We denote the resulting functor
$\Dmod(\Bun_M)\to \on{I}(G,P)$ by $(\imath_M)_\dagger$.

\medskip

By passing to left adjoints in \eqref{e:av strata-wise}, we obtain a commutative diagram 
\begin{equation}  \label{e:equiv strata-wise !}
\CD
\on{I}(G,P)  @>>> \Dmod(\Bun_G^{P\on{-gen}}) \\
@A{(\imath_M)_\dagger}AA  @AA{(\imath_P)_\dagger}A   \\
\Dmod(\Bun_M)  @>{(\sfq_P)^\bullet}>>  \Dmod(\Bun_P).
\endCD
\end{equation}

\ssec{The ``enhanced" constant term and Eisenstein functors}

As in the classical theory of automorphic functions, the parabolic category $\on{I}(G,P)$
is related to the automorphic category $\Dmod(\Bun_G)$ by a pair of functors, called
``constant term" and ``Eisenstein series." 

\sssec{}

We define the functor of \emph{enhanced constant term}
$$\on{CT}_P^{\on{enh}}:\Dmod(\Bun_G)\to \on{I}(G,P)$$
as the composition 
$$\on{CT}_P^{\on{enh}}=\on{Av}^{\bNP}\circ (\sfp^{\on{enh}}_P)^\dagger.$$

\sssec{}

We claim that the functor $\on{CT}_P^{\on{enh}}$ admits a left adjoint. This follows
from the next lemma:

\begin{lem}  \label{l:p! defined}
The partially defined left adjoint $(\sfp^{\on{enh}}_P)_\dagger$ of $(\sfp^{\on{enh}}_P)^\dagger$ is defined on the full
subcategory $\on{I}(G,P)\subset \Dmod(\Bun_G^{P\on{-gen}})$.
\end{lem}

Thus, the functor
$$\Eis_P^{\on{enh}}:=(\sfp^{\on{enh}}_P)_\dagger|_{\on{I}(G,P)},\quad \on{I}(G,P)\to \Dmod(\Bun_G)$$
is well-defined and provides a left adjoint to $\on{CT}_P^{\on{enh}}$.

\medskip

We will refer to $\Eis_P^{\on{enh}}$ as the functor of \emph{enhanced Eisenstein series}. 

\sssec{}

Consider now the diagram
$$
\xy
(-15,0)*+{\Bun_G}="X";
(15,0)*+{\Bun_M.}="Y";
(0,15)*+{\Bun_P}="Z";
{\ar@{->}_{\sfp_P} "Z";"X"};
{\ar@{->}^{\sfq_P}  "Z";"Y"};
\endxy
$$

We define the \emph{usual} constant term and Eisenstein functors as follows:

$$\on{CT}_P=(\sfq_P)_\bullet \circ (\sfp_P)^\dagger,$$
where $(\sfq_P)_\bullet$ is the right adjoint of the functor $(\sfq_P)^\bullet$ (i.e., $(\sfq_P)_\bullet$ is the functor
of usual direct image for D-modules).

\sssec{}

The functor $\Eis_P$ (called the usual functor of Eisenstein series), left adjoint to $\on{CT}_P$, 
is described as 
$$(\sfp_P)_\dagger\circ (\sfq_P)^\bullet.$$

\medskip

The functor $(\sfp_P)_\dagger$ is the partially defined left adjoint to $(\sfp_P)^\dagger$,
and as in \lemref{l:p! defined} one shows that it is defined on the essential image of $(\sfq_P)^\bullet$.

\sssec{}

From \eqref{e:av strata-wise} we obtain that the functor $\on{CT}_P$ can be expressed through $\on{CT}_P^{\on{enh}}$
as 
$$\on{CT}_P\simeq (\imath_M)^\dagger\circ \on{CT}_P^{\on{enh}}.$$

Similarly, from \eqref{e:equiv strata-wise !}, we obtain that the functor $\Eis_P$ can be expressed through $\Eis_P^{\on{enh}}$ as 
$$\Eis_P\simeq \Eis_P^{\on{enh}}\circ (\imath_M)_\dagger.$$

\ssec{Spectral Eisenstein series}

The functors of constant term and Eisenstein series on the geometric side have their respective counterparts on the
spectral side. In this subsection we will introduce the spectral counterparts of the naive functors $\Eis_P$ and $\on{CT}_P$;
their enhanced versions will be introduced in \secref{ss:spectral parabolic}. 

\sssec{}

Consider the derived stack $\LocSys_\cP$ and the diagram
$$
\xy
(-15,0)*+{\LocSys_\cG}="X";
(15,0)*+{\LocSys_\cM.}="Y";
(0,15)*+{\LocSys_\cP}="Z";
{\ar@{->}_{\sfp_{\cP,\on{spec}}} "Z";"X"};
{\ar@{->}^{\sfq_{\cP,\on{spec}}}  "Z";"Y"};
\endxy
$$

We note that the morphism $\sfq_{\cP,\on{spec}}$ is quasi-smooth (i.e., its geometric fibers are quasi-smooth), 
and in particular of finite Tor dimension. Hence, the functor
$$\sfq_{\cP,\on{spec}}^{\IndCoh,*}:\IndCoh(\LocSys_\cM)\to \IndCoh(\LocSys_\cP),$$
is well-defined, see \secref{sss:! pull-back}.

\medskip

We also note that the morphism $\sfp_{\cP,\on{spec}}$ is schematic and proper. Hence, the functor
$$\sfp_{\cP,\on{spec}}^!:\IndCoh(\LocSys_\cG)\to \IndCoh(\LocSys_\cP),$$
right adjoint to $(\sfp_{\cP,\on{spec}})^{\IndCoh}_*$, is well-defined and is continuous, see again \secref{sss:! pull-back}.

\sssec{}

We let $\on{Nilp}^{\on{glob}}_\cP$ be the conical Zariski-closed subset of $\Sing(\LocSys_\cP)$ that corresponds to
pairs $(\sigma,A)$, where $\sigma$ is a $\cP$-local system, and $A$ is a horizontal section of
$\check\fp^*_\sigma$ that belongs to $\check\fm^*_\sigma\subset \check\fp^*_\sigma$, and is nilpotent as a section of
$\check\fm^*_\sigma$. 

\medskip

We consider the corresponding category 
$$\IndCoh_{\on{Nilp}^{\on{glob}}_\cP}(\LocSys_\cP)\subset \IndCoh(\LocSys_\cP).$$

\medskip

The following is shown in \cite[Propositions 13.2.6]{AG}:

\begin{lem}  \hfill

\smallskip

\noindent{\em(a)} The functor $\sfq_{\cP,\on{spec}}^{\IndCoh,*}:\IndCoh(\LocSys_\cM)\to \IndCoh(\LocSys_\cP)$
sends the subcategory 
$\IndCoh_{\on{Nilp}^{\on{glob}}_\cM}(\LocSys_\cM)$ to the subcategory 
$\IndCoh_{\on{Nilp}^{\on{glob}}_\cP}(\LocSys_\cP)$.

\smallskip

\noindent{\em(b)} The functor $(\sfp_{\cP,\on{spec}})^{\IndCoh}_*:\IndCoh(\LocSys_\cP)\to \IndCoh(\LocSys_\cG)$,
sends the subcategory $\IndCoh_{\on{Nilp}^{\on{glob}}_\cP}(\LocSys_\cP)$
to the subcategory $\IndCoh_{\on{Nilp}^{\on{glob}}_\cG}(\LocSys_\cG)$.

\end{lem}

\sssec{}

Hence, we obtain well-defined functors 
$$\sfq_{\cP,\on{spec}}^{\IndCoh,*}:\IndCoh_{\on{Nilp}^{\on{glob}}_\cM}(\LocSys_\cM)\to \IndCoh_{\on{Nilp}^{\on{glob}}_\cP}(\LocSys_\cP)$$ 
and 
$$(\sfp_{\cP,\on{spec}})^{\IndCoh}_*:\IndCoh_{\on{Nilp}^{\on{glob}}_\cP}(\LocSys_\cP)\to \IndCoh_{\on{Nilp}^{\on{glob}}_\cG}(\LocSys_\cG),$$
that admit (continuous) right adjoints
$$\IndCoh_{\on{Nilp}^{\on{glob}}_\cM}(\LocSys_\cM)\leftarrow \IndCoh_{\on{Nilp}^{\on{glob}}_\cP}
(\LocSys_\cP):(\sfq_{\cP,\on{spec}})_*^{\IndCoh}$$
and
$$\IndCoh_{\on{Nilp}^{\on{glob}}_\cP}(\LocSys_\cP)\leftarrow \IndCoh_{\on{Nilp}^{\on{glob}}_\cG}(\LocSys_\cG):\sfp_{\cP,\on{spec}}^!,$$
respectively.

\sssec{}

We define the spectral Eisenstein series functor
$$\Eis_{\cP,\on{spec}}:\IndCoh_{\on{Nilp}^{\on{glob}}_\cM}(\LocSys_\cM)\to \IndCoh_{\on{Nilp}^{\on{glob}}_\cG}(\LocSys_\cG)$$
as
$$\Eis_{\cP,\on{spec}}:=(\sfp_{\cP,\on{spec}})_*^{\IndCoh}\circ \sfq_{\cP,\on{spec}}^{\IndCoh,*}.$$

We introduce the spectral constant term functor 
$$\on{CT}_{\cP,\on{spec}}:\IndCoh_{\on{Nilp}^{\on{glob}}_\cG}(\LocSys_\cG)\to \IndCoh_{\on{Nilp}^{\on{glob}}_\cM}(\LocSys_\cM)$$
as
$$\on{CT}_{\cP,\on{spec}}:=(\sfq_{\cP,\on{spec}})_*^{\IndCoh}\circ \sfp_{\cP,\on{spec}}^!.$$

By construction, $\on{CT}_{\cP,\on{spec}}$ is the right adjoint of $\Eis_{\cP,\on{spec}}$.

\sssec{}

In addition to the adjoint pair
$$\Eis_{\cP,\on{spec}}:\IndCoh_{\on{Nilp}^{\on{glob}}_\cM}(\LocSys_\cM)\rightleftarrows \IndCoh_{\on{Nilp}^{\on{glob}}_\cG}(\LocSys_\cG):
\on{CT}_{\cP,\on{spec}}$$
we shall also consider the corresponding adjoint pair
$$\Eis_{\cP,\on{spec}}\circ \Xi_\cM:
\QCoh(\LocSys_\cM)\rightleftarrows \IndCoh_{\on{Nilp}^{\on{glob}}_\cG}(\LocSys_\cG):
\Psi_\cM\circ \on{CT}_{\cP,\on{spec}}.$$

In a certain sense the above two adjoint pairs end up retaining the same information. More precisely, we have
the following result of \cite[Corollary 13.3.10 and Theorem 13.3.6]{AG}:

\begin{prop}  \label{p:Eis generation}  \hfill

\smallskip

\noindent{\em(a)}
The essential images of the functors 
$$\Eis_{\cP,\on{spec}}\circ \Xi_\cM:\QCoh(\LocSys_\cM)\to \IndCoh_{\on{Nilp}^{\on{glob}}_\cG}(\LocSys_\cG)$$
for all parabolics $P$ (including $P=G$) generate $\IndCoh_{\on{Nilp}^{\on{glob}}_\cG}(\LocSys_\cG)$.

\smallskip

\noindent{\em(b)}
The essential images of the functors 
$$\Eis_{\cP,\on{spec}}\circ \Xi_\cM:\QCoh(\LocSys_\cM)\to \IndCoh_{\on{Nilp}^{\on{glob}}_\cG}(\LocSys_\cG)$$
for all \emph{proper} parabolics generate the full subcategory equal to the kernel of the restriction functor
$$\IndCoh_{\on{Nilp}^{\on{glob}}_\cG}(\LocSys_\cG)\to \IndCoh_{\on{Nilp}^{\on{glob}}_\cG}(\LocSys^{\on{irred}}_\cG)\simeq
\QCoh(\LocSys^{\on{irred}}_\cG).$$

\end{prop}

\sssec{}   \label{sss:cond E}

We can now formulate Property $\on{Ei}$  (``Ei'' stands for Eisenstein) of the compatibility of the geometric 
Langlands functor for the group $G$ and its Levi subgroups. It is a particular case of Porperty $\on{Ei}^{\on{enh}}$ 
formulated in \secref{sss:cond E rat}.

\medskip

\noindent{\bf Property $\mathbf{Ei}$:} {\it We shall say that the functor $\BL_G$ satisfies \emph{Porperty $\on{Ei}$} if
the following diagram of functors commutes:
\begin{equation} \label{e:cond E}
\CD
\IndCoh_{\on{Nilp}^{\on{glob}}_\cM}(\LocSys_{\cM})  @>{\BL_M}>>  \Dmod(\Bun_M)  \\
@V{-\otimes \check\fl_{M,G}}VV @VV{-\otimes \fl_{M,G}}V  \\
\IndCoh_{\on{Nilp}^{\on{glob}}_\cM}(\LocSys_{\cM})  & &  \Dmod(\Bun_M)  \\
@V{\Eis_{\cP,\on{spec}}}VV    @VV{\Eis_P}V   \\
\IndCoh_{\on{Nilp}^{\on{glob}}_\cG}(\LocSys_\cG)  @>{\BL_G}>>  \Dmod(\Bun_G)
\endCD
\end{equation}
where $-\otimes \check\fl_{M,G}$ and $-\otimes \fl_{M,G}$ are auto-equivalences defined below.}

\medskip

The functor $-\otimes \check\fl_{M,G}$ appearing in the statement of Property $\on{Ei}$ 
is that given by tensor product by a line bundle $\check\fl_{M,G}$ on $\LocSys_\cM$ equal to
the pull-back under
$$\LocSys_\cM \overset{2\check{\rho}_P}\to \LocSys_{\BG_m}\to \Bun_{\BG_m}=\on{Pic}$$
of the line bundle on $\on{Pic}$ corresponding to the canonical line bundle $\omega^{\frac{1}{2}}_X$ on $X$,
and where $2\check\rho_P$ is the character $\cM\to \BG_m$  correspinding to the
determinant of the adjoint action on $\check\fp$.

\medskip

The functor $-\otimes \fl_{M,G}$ is given by tensor product by the (constant) D-module
on $\Bun_M$, which on the connected component $\Bun_M^{\check\mu}$ corresponding
to $\check\mu:M\to \BG_m$ is given by the cohomological shift by 
$\langle 2\rho_P,2(g-1)\check\rho_P-\check\mu\rangle$. 

\sssec{}

By adjunction, Property $\on{Ei}$ implies that the following diagram of functors commutes as well:
\begin{equation} \label{e:cond CT}
\CD
\IndCoh_{\on{Nilp}^{\on{glob}}_\cM}(\LocSys_{\cM})  @>{\BL_M}>>  \Dmod(\Bun_M)  \\
@A{-\otimes \check\fl_{M,G}^{-1}}AA @AA{-\otimes \fl_{M,G}^{-1}}A  \\
\IndCoh_{\on{Nilp}^{\on{glob}}_\cM}(\LocSys_{\cM})  & &  \Dmod(\Bun_M)  \\
@A{\on{CT}_{\cP,\on{spec}}}AA    @AA{\on{CT}_P}A   \\
\IndCoh_{\on{Nilp}^{\on{glob}}_\cG}(\LocSys_\cG)  @>{\BL_G}>>  \Dmod(\Bun_G)
\endCD
\end{equation}

\ssec{The spectral parabolic category}  \label{ss:spectral parabolic}

The goal of this subsection is to define a spectral counterpart  of the category $\on{I}(G,P)$
and the functors $\Eis_P^{\on{enh}}$ and $\on{CT}_P^{\on{enh}}$.

\sssec{}

Consider the groupoid 
$$
\xy
(-15,0)*+{\LocSys_\cP}="X";
(15,0)*+{\LocSys_\cP}="Y";
(0,15)*+{\LocSys_\cP\underset{\LocSys_\cG}\times \LocSys_\cP}="Z";
{\ar@{->}_{p_1} "Z";"X"};
{\ar@{->}^{p_2}  "Z";"Y"};
\endxy
$$
over $\LocSys_\cP$.

\medskip

Since the map $p_i$ ($i=1,2$) is schematic and proper, we have an adjoint pair of (continuous) functors
$$(p_i)^{\IndCoh}_*: \IndCoh(\LocSys_\cP\underset{\LocSys_\cG}\times \LocSys_\cP)\to \IndCoh(\LocSys_\cP):p_i^!.$$

\sssec{}

We let 
\begin{equation} \label{e:cat on diag}
\IndCoh(\LocSys_\cP\underset{\LocSys_\cG}\times \LocSys_\cP)_\Delta \hookrightarrow
\IndCoh(\LocSys_\cP\underset{\LocSys_\cG}\times \LocSys_\cP)
\end{equation}
denote the full subcategory consisting of objects that are set-theoretically supported on the image
of the diagonal embedding
$$\LocSys_\cP\to \LocSys_\cP\underset{\LocSys_\cG}\times \LocSys_\cP.$$

\medskip

We let $(p_{i,\Delta})^{\IndCoh}_*$ denote the restriction of $(p_i)^{\IndCoh}_*$
to the subcategory \eqref{e:cat on diag}. We let $p_{i,\Delta}^!$ denote the right adjoint of $(p_{i,\Delta})^{\IndCoh}_*$, which is
isomorphic to the composition of $p_i^!$ and the right adjoint to the embedding \eqref{e:cat on diag}.

\sssec{}

The structure of groupoid on
$\LocSys_\cP\underset{\LocSys_\cG}\times \LocSys_\cP$ endows the endo-functor
$$(p_{2,\Delta})^{\IndCoh}_*\circ p_{1,\Delta}^!$$
of $\IndCoh(\LocSys_\cP)$
with a structure of monad. We shall denote this monad by $\sF_{\cP}$. 

\medskip

We have:

\begin{lem} \label{l:monad preserves}
Let $\CN\subset \Sing(\LocSys_\cP)$ be any conical Zariski-closed subset. Then the functor
$\sF_{\cP}$ sends the full subcategory
$$\IndCoh_\CN(\LocSys_\cP)\subset \IndCoh(\LocSys_\cP)$$
to itself. 
\end{lem}

\noindent (The lemma holds more generally when $\LocSys_\cG$ and $\LocSys_\cP$ are replaced by 
arbitrary quasi-smooth algebraic stacks.) 

\sssec{}  \label{sss:monad and Hecke}

By construction, the action of the monad $\sF_{\cP}$ on the category $\IndCoh(\LocSys_\cP)$ 
commutes with the action of the (symmetric) monoidal category $\QCoh(\LocSys_\cG)$, where the latter 
acts on $\IndCoh(\LocSys_\cP)$ via the (symmetric) monoidal functor $$\sfp_{\cP,\on{spec}}^*:\QCoh(\LocSys_\cG)\to
\QCoh(\LocSys_\cP)$$ and the canonical action of $\QCoh(\LocSys_\cP)$ on $\IndCoh(\LocSys_\cP)$.

\sssec{}

We consider the category $$\sF_{\cP}\mod(\IndCoh(\LocSys_\cP))$$ of $\sF_{\cP}$-modules in $\IndCoh(\LocSys_\cP)$.
We let
$$\ind_{\sF_{\cP}}:\IndCoh(\LocSys_\cP)\rightleftarrows \sF_{\cP}\mod(\IndCoh(\LocSys_\cP)):\oblv_{\sF_{\cP}}$$
be the corresponding adjoint pair of forgetful and induction functors.

\medskip

By \lemref{l:monad preserves} we also have well-defined full subcategories
\begin{multline*}
\sF_{\cP}\mod(\QCoh(\LocSys_\cP))\subset \sF_{\cP}\mod(\IndCoh_{\on{Nilp}^{\on{glob}}_\cP}(\LocSys_\cP))\subset \\
\subset \sF_{\cP}\mod(\IndCoh(\LocSys_\cP))
\end{multline*}
and the functors
$$\ind_{\sF_{\cP}}:\QCoh(\LocSys_\cP)\rightleftarrows \sF_{\cP}\mod(\QCoh(\LocSys_\cP)):\oblv_{\sF_{\cP}}$$
and
$$\ind_{\sF_{\cP}}:\IndCoh_{\on{Nilp}^{\on{glob}}_\cP}(\LocSys_\cP)\rightleftarrows 
\sF_{\cP}\mod(\IndCoh_{\on{Nilp}^{\on{glob}}_\cP}(\LocSys_\cP)):\oblv_{\sF_{\cP}}$$
that commute with the corresponding fully faithful embeddings and their right adjoints, denoted
$(\Xi_\cP,\Psi_\cP)$, respectively. 

\medskip

In particular, we have the following commutative daigarms
$$
\CD
\IndCoh_{\on{Nilp}^{\on{glob}}_\cP}(\LocSys_\cP)  @>{\ind_{\sF_{\cP}}}>>
\sF_{\cP}\mod(\IndCoh_{\on{Nilp}^{\on{glob}}_\cP}(\LocSys_\cP))  \\
@A{\Xi_\cP}AA    @AA{\Xi_\cP}A  \\
\QCoh(\LocSys_\cP)  @>{\ind_{\sF_{\cP}}}>>
\sF_{\cP}\mod(\QCoh(\LocSys_\cP)) 
\endCD
$$
and
$$
\CD
\IndCoh_{\on{Nilp}^{\on{glob}}_\cP}(\LocSys_\cP)  @>{\ind_{\sF_{\cP}}}>>
\sF_{\cP}\mod(\IndCoh_{\on{Nilp}^{\on{glob}}_\cP}(\LocSys_\cP))  \\
@V{\Psi_\cP}VV    @VV{\Psi_\cP}V  \\
\QCoh(\LocSys_\cP)  @>{\ind_{\sF_{\cP}}}>>
\sF_{\cP}\mod(\QCoh(\LocSys_\cP)). 
\endCD
$$

\medskip

Finally, it follows from \secref{sss:monad and Hecke} that the category $\sF_{\cP}\mod(\IndCoh_{\on{Nilp}^{\on{glob}}_\cP}(\LocSys_\cP))$
is naturally acted on by the monoidal category $\QCoh(\LocSys_\cG)$, and the functors $\ind_{\sF_{\cP}}$ and $\oblv_{\sF_{\cP}}$
commute with this action.

\sssec{}

Consider again the functors
$$(\sfp_{\cP,\on{spec}})^{\IndCoh}_*:
\IndCoh_{\on{Nilp}^{\on{glob}}_\cP}(\LocSys_\cP)\rightleftarrows 
\IndCoh_{\on{Nilp}^{\on{glob}}_\cG}(\LocSys_\cG): \sfp_{\cP,\on{spec}}^!.$$

It follows from the definitions that the functor $\sfp_{\cP,\on{spec}}^!$
canonically factors as a composition
\begin{multline*}
\IndCoh_{\on{Nilp}^{\on{glob}}_\cG}(\LocSys_\cG)\to
\sF_{\cP}\mod(\IndCoh_{\on{Nilp}^{\on{glob}}_\cP}(\LocSys_\cP))\overset{\oblv_{\sF_{\cP}}}\longrightarrow \\
\to \IndCoh_{\on{Nilp}^{\on{glob}}_\cP}(\LocSys_\cP).
\end{multline*}

We denote the resulting functor
$$\IndCoh_{\on{Nilp}^{\on{glob}}_\cG}(\LocSys_\cG)\to
\sF_{\cP}\mod(\IndCoh_{\on{Nilp}^{\on{glob}}_\cP}(\LocSys_\cP))$$
by $\on{CT}^{\on{enh}}_{\cP,\on{spec}}$.

\sssec{}

It follows formally from the Barr-Beck-Lurie theorem (see \cite[Proposition 3.1.1]{DG}) that there exists a canonically defined functor,
to be denoted $\Eis^{\on{enh}}_{\cP,\on{spec}}$, 
$$\sF_{\cP}\mod(\IndCoh_{\on{Nilp}^{\on{glob}}_\cP}(\LocSys_{\cP}))\to \IndCoh_{\on{Nilp}^{\on{glob}}_\cG}(\LocSys_\cG)$$
equipped with an isomorphism
$$\Eis^{\on{enh}}_{\cP,\on{spec}}\circ \ind_{\sF_{\cP}}\simeq 
(\sfp_{\cP,\on{spec}})^{\IndCoh}_*.$$

Furthermore, the functor $\Eis^{\on{enh}}_{\cP,\on{spec}}$ is the \emph{left} adjoint of
$\on{CT}^{\on{enh}}_{\cP,\on{spec}}$.

\sssec{}

By construction, the functors $\on{CT}^{\on{enh}}_{\cP,\on{spec}}$ and $\Eis^{\on{enh}}_{\cP,\on{spec}}$
intertwine the monoidal actions of $\QCoh(\LocSys_\cG)$ on $\sF_{\cP}\mod(\IndCoh_{\on{Nilp}^{\on{glob}}_\cP}(\LocSys_{\cP}))$
and $\IndCoh_{\on{Nilp}^{\on{glob}}_\cG}(\LocSys_\cG)$, respectively.

\sssec{}

We proclaim the category $\sF_{\cP}\mod(\IndCoh_{\on{Nilp}^{\on{glob}}_\cP}(\LocSys_{\cP}))$, equipped with the adjoint functors
$$\Eis^{\on{enh}}_{\cP,\on{spec}}:\sF_{\cP}\mod(\IndCoh_{\on{Nilp}^{\on{glob}}_\cP}(\LocSys_{\cP}))\rightleftarrows
\IndCoh_{\on{Nilp}^{\on{glob}}_\cG}(\LocSys_\cG):\on{CT}^{\on{enh}}_{\cP,\on{spec}}$$
to be the spectral counterpart of the category $\on{I}(G,P)$ equipped with the adjoint functors
$$\Eis^{\on{enh}}:\on{I}(G,P)\rightleftarrows \Dmod(\Bun_G):\on{CT}_P^{\on{enh}}.$$

\ssec{Compatibility of Langlands correspondence with parabolic induction}

For the duration of this subsection we will assume the validity of \conjref{c:GL} 
for the reductive group $M$. In particular, this is unconditional for $P=B$, in which case $M$ is a torus,
and \conjref{c:GL} amounts to Fourier-Mukai transform.

\medskip

The key observation is that although the categories $$\on{I}(G,P) \text{ and }
\sF_{\cP}\mod(\IndCoh_{\on{Nilp}^{\on{glob}}_\cP}(\LocSys_\cP))$$ cannot be recovered purely in terms of the reductive
group $M$ (i.e., we need to know how it is realized as a Levi of $G$), this additional $G$-information is manageable, and so we
can relate these categories by just knowing Langlands correspondence for $M$. 

\sssec{}

We have:


\begin{qthm}  \label{t:principal series}  \hfill

\smallskip

\noindent{\em(a)}
There exists a canonically defined equivalence of categories 
$$\BL_P:\sF_{\cP}\mod(\IndCoh_{\on{Nilp}^{\on{glob}}_\cP}(\LocSys_\cP))\to \on{I}(G,P)$$
that makes the following diagram commute:
\begin{equation} \label{e:principal series}
\CD
\IndCoh_{\on{Nilp}^{\on{glob}}_\cM}(\LocSys_\cM) @>{\BL_M}>> \Dmod(\Bun_M) \\
@V{-\otimes \check\fl_{M,G}}VV  @VV{-\otimes \fl_{M,G}}V   \\
\IndCoh_{\on{Nilp}^{\on{glob}}_\cM}(\LocSys_\cM) & & \Dmod(\Bun_M) \\
@V{\ind_{\sF_{\cP}}\circ (\sfq_{\cP,\on{spec}})^*}VV   @VV{(\imath_M)_\dagger}V  \\
\sF_{\cP}\mod(\IndCoh_{\on{Nilp}^{\on{glob}}_\cP}(\LocSys_\cP))  @>{\BL_P}>> \on{I}(G,P),
\endCD
\end{equation}
where $-\otimes \check\fl_{M,G}$ and $-\otimes \fl_{M,G}$ are the auto-equivalences defined in \secref{sss:cond E}.

\smallskip

\noindent{\em(b)} The equivalence $\BL_P$ is compatible with the action of the category $\Rep(\cG)_{\Ran(X)}$,
where

\begin{itemize}

\item $\Rep(\cG)_{\Ran(X)}$ acts on $\sF_{\cP}\mod(\IndCoh_{\on{Nilp}^{\on{glob}}_\cP}(\LocSys_\cP))$ via
the symmetric monoidal functor
$$\on{Loc}_{\cG,\on{spec}}:\Rep(\cG)_{\Ran(X)}\to \QCoh(\LocSys_\cG);$$

\item $\Rep(\cG)_{\Ran(X)}$ acts on $\on{I}(G,P)$ as in \secref{sss:Hecke action on PS}.

\end{itemize}

\end{qthm}

\sssec{}

In the case of $P=B$, Quasi-Theorem \ref{t:principal series} is work-in-progress by S.~Raskin. The idea of the proof,
applicable to any $P$, is the following: 

\medskip

The composite functors 
\begin{equation} \label{e:parabolic monad 1}
\IndCoh_{\on{Nilp}^{\on{glob}}_\cM}(\LocSys_\cM)\to \sF_{\cP}\mod(\IndCoh_{\on{Nilp}^{\on{glob}}_\cP}(\LocSys_\cP))
\end{equation}
and 
\begin{equation} \label{e:parabolic monad 2}
\Dmod(\Bun_M)\to  \on{I}(G,P)
\end{equation}
appearing in \eqref{e:principal series} admit continuous and consertavative right adjoints, which, up to
twists by line bundles, are given by
$$(\sfq_{\cP,\on{spec}})^{\IndCoh}_* \circ \oblv_{\sF(P)} \text{ and } \iota_M^\dagger,$$
respectively. Hence,
by the Barr-Beck-Lurie theorem, the statement of Quasi-Theorem  \ref{t:principal series} amounts to comparing the
monads corresponding to the composition of the functors in \eqref{e:parabolic monad 1} and \eqref{e:parabolic monad 2}, 
and their respective right adjoints. 

\medskip

One shows that the monad on the geometric side, i.e., $\Dmod(\Bun_M)$, is given by the action of an algebra object
in the monoidal category $\Dmod(\on{Hecke}(M)_{\Ran(X)})$ that comes via the functor 
$\on{Sat}(G)_{\Ran(X)}$ from a canonically defined algebra object of the monoidal category 
$\IndCoh(\on{Hecke}(\cG,\on{spec})^{\on{loc}}_{\Ran(X)})$, see \secref{sss:Hecke spec Ran}.
\footnote{We emphasize that the above algebra object of $\IndCoh(\on{Hecke}(\cG,\on{spec})^{\on{loc}}_{\Ran(X)})$ 
\emph{does not} come from an algebra object of $\Rep(\cG)_{\Ran(X)}$ via the functor $\to$ in \eqref{e:small to big Hecke},
so here one really needs to use the full derived Satake equivalence for the group $M$.}

\medskip

One then uses Bezrukavnikov's theory of \cite{Bez} that describes various
categories of D-modules on the affine Grassmannian in terms of the Langlands dual group to
match the resulting monad with one appearing on the spectral side.

\sssec{}    \label{sss:cond E rat}

We can now state Property $\on{Ei}^{\on{enh}}$ of the geometric Langlands functor $\BL_G$ in \conjref{c:GL}.

\medskip

\noindent{\bf Property $\mathbf{Ei}^{\mathbf{enh}}$:} {\it We shall say that the functor $\BL_G$ satisfies
\emph{Property $\on{Ei}^{\on{enh}}$} if the following diagram of functors
commutes:
\begin{equation} \label{e:cond E rat}
\CD
\sF_{\cP}\mod(\IndCoh_{\on{Nilp}^{\on{glob}}_\cP}(\LocSys_\cP)) @>{\BL_P}>>  \on{I}(G,P)   \\
@V{\Eis^{\on{enh}}_{\cP,\on{spec}}}VV    @VV{\Eis_P^{\on{enh}}}V   \\
\IndCoh_{\on{Nilp}^{\on{glob}}_\cG}(\LocSys_\cG)  @>{\BL_G}>>  \Dmod(\Bun_G).
\endCD
\end{equation}}

\sssec{}

Note that by passing to right adjoint functors in \eqref{e:cond E rat} we obtain the following commutative diagram

\begin{equation} \label{e:CT spec rat}
\CD
\sF_{\cP}\mod(\IndCoh_{\on{Nilp}^{\on{glob}}_\cP}(\LocSys_\cP)) @>{\BL_P}>>  \on{I}(G,P)   \\
@A{\on{CT}^{\on{enh}}_{\cP,\on{spec}}}AA    @AA{\on{CT}_P^{\on{enh}}}A   \\
\IndCoh_{\on{Nilp}^{\on{glob}}_\cG}(\LocSys_\cG)  @>{\BL_G}>>  \Dmod(\Bun_G).
\endCD
\end{equation}

\sssec{}

Finally, we note that Property $\on{Ei}$ stated in \secref{sss:cond E} is a formal consequence of
Property $\on{Ei}^{\on{enh}}$: the commutative diagram \eqref{e:cond E} is obtained by
concatenating \eqref{e:cond E rat} and \eqref{e:principal series}.

\ssec{Eisenstein and constant term compatibility}

Let now $P$ and $P'$ be two parabolic subgroups, and let us assume the validity of \conjref{c:GL} for
the Levi quotient $M'$ as well. 

\sssec{}

By concatenating diagrams \eqref{e:cond E rat} (for $P'$) and 
\eqref{e:CT spec rat} (for$P$) we obtain the following commutative diagram:

\begin{equation} \label{e:CT and Eis}
\CD
\sF_{\cP}\mod(\IndCoh_{\on{Nilp}^{\on{glob}}_\cP}(\LocSys_\cP))  @>{\BL_P}>>  \on{I}(G,P)   \\
@A{\on{CT}^{\on{enh}}_{\cP,\on{spec}}}AA      @AA{\on{CT}_P^{\on{enh}}}A    \\
\IndCoh_{\on{Nilp}^{\on{glob}}_\cG}(\LocSys_\cG)    & &   \Dmod(\Bun_G) \\
@A{\Eis^{\on{enh}}_{\cP',\on{spec}}}AA    @AA{\Eis_{P'}^{\on{enh}}}A   \\
\sF_{\cP'}\mod(\IndCoh_{\on{Nilp}_{\on{glob},\cP'}}(\LocSys_{\cP'}))  @>{\BL_{P'}}>>  \on{I}(G,P')
\endCD
\end{equation}

\medskip

We have (again, assuming the validity of \conjref{c:GL} for $M$ and $M'$):

\begin{qthm} \label{t:CT and Eis}
The diagram \eqref{e:CT and Eis} commutes unconditionally (i.e., without assuming the validity of
\conjref{c:GL} for $G$).
\end{qthm}


\begin{rem}

One shows that both functors in \eqref{e:CT and Eis} corresponding to the 
vertical arrows admit natural filtrations indexed by the poset
$$W_M\backslash W/W_{M'},$$
where $W$ is the Weyl group of $G$, and $W_M$ and $W_{M'}$ are the Weyl groups of
$M$ and $M'$, respectively. 

\medskip

In order to prove Quasi-Theorem \ref{t:CT and Eis}, one needs to identify the corresponding subquotients
on both sides (via the equivalence of Quasi-Theorem \ref{t:principal series} for $M$ and $M'$), and then
show that these subquotients glue in the same way on both sides. 

\medskip

For $G=GL_2$ (and when $M=M'=T$) the first step follows easily from Quasi-Theorem \ref{t:principal series},
and the second step is an explicit calculation of a class in an appropriate $\Ext^1$ group. 

\end{rem}
 
\section{The degenerate Whittaker model}     \label{s:deg Whit}

This section develops a variant of the category $\Whit(G,G)$, denoted $\Whit(G,P)$ (for a fixed parabolic $P$), 
where we impose an equivariance condition with respect to a character of $N$, which is no longer non-degenerate, but 
is trivial on $N(P)$ and non-degenerate on $N(M)$, where $N(M)=N\cap M$ is the unipotent radical of the Borel subgroup
of $M$.  

\medskip

The reason that $\Whit(G,P)$ is necessary to consider is that these categories, when $P$ runs through the poset of
standard parabolics, comprise the extended Whittaker category introduced in the next section, and which will be of central importance
for the proof of \conjref{c:GL}.

\medskip

That said, we should remark that the present section does not contain any substantially new ideas. Furthermore, 
the material discussed here is relevant only for groups of semi-simple rank $>1$, because the case $P=G$ is covered by
\secref{s:Whit}, and in the case $P=B$ we have $\Whit(G,P)=\on{I}(G,B)$. So the reader may prefer to skip this section 
on the first pass. 

\ssec{Degenerate Whittaker categories}

The degenerate Whittaker category $\on{Whit}(G,P)$ defined in this subsection is the geometric counterpart of the space 
of functions on the double quotient $$Z^0_M(K)\backslash G(\BA)/G(\BO)$$ that are equivariant with respect to $N(\BA)$
against a character that factors via the surjection $N(\BA)\to N(M)(\BA)$ and a non-degenerate
character of $N(M)(\BA)$, trivial on $N(M)(K)$. 

\sssec{}

We define the prestack $\CQ_{G,P}$ in a way similar to $\CQ_{G,G}$. It classifies the data of
$$(\CP_G,U,\alpha,\gamma),$$
where $(\CP_G,U,\alpha)$ has the same meaning as for $\CQ_{G,G}$ (i.e., it defines a point of
$\Bun_G^{B\on{-gen}}$), but $\gamma$ is now an identification of bundles with respect to the torus
$T/Z^0_M$, one bundle being induced from $\CP_{T,U}$, and the other from $\check\rho(\omega_X)$.

\medskip

Equivalently, $\CQ_{G,P}$ is the quotient of $\CQ_Q$ by the action of $\bMaps(X,Z^0_M)^{\on{gen}}$.

\medskip

(Note that when $G$ has a connected center, the data of $\gamma$ amounts to an isomorphism
$\alpha_i(\CP_{T,U})\simeq \omega_X$ for every simple root $\alpha_i$ of $M$.)

\medskip

A choice of a generic trivialization of $\omega^{\frac{1}{2}}_X$ identifies the 
groupoid on $k$-points of $\CQ_{G,P}$ with the 
double quotient
$$Z_M^0(K)\cdot N(K)\backslash G(\BA)/G(\BO).$$

\medskip

By construction, if $P=G$, we have $\CQ_{G,P}=\CQ_{G,G}$ (so the notation is consistent). When
$P=B$, we have $\CQ_{G,P}=\Bun_G^{B\on{-gen}}$. 

\medskip

We let $\fr_{G,P}$ denote the forgetful map  $\CQ_{G,P}\to \Bun_G$. 

\sssec{}

The groupoid $\bN$ acting $\left(\Bun_G^{B\on{-gen}}\times \Ran(X)\right){}_{\on{good}}$ gives rise to a groupoid that 
we denote $\bN_{\CQ_{G,P}}$ over $\left(\CQ_{G,P}\times \Ran(X)\right){}_{\on{good}}$ so that the diagram
$$
\CD
\left(\CQ_{G,P}\times \Ran(X)\right){}_{\on{good}} @<{p_1}<<  \bN_{\CQ_{G,P}} @>{p_2}>> \left(\CQ_{G,P}\times \Ran(X)\right){}_{\on{good}}  \\
@VVV   @VVV  @VVV  \\
\left(\Bun_G^{B\on{-gen}}\times \Ran(X)\right){}_{\on{good}} @<{p_1}<<  \bN @>{p_2}>> \left(\Bun_G^{B\on{-gen}}\times \Ran(X)\right){}_{\on{good}}
\endCD
$$
is Cartesian. 

\medskip

The groupoid $\bN_{\CQ_{G,P}}$ is endowed with a canonically defined character that we denote
$\chi_P$. The definition of $\chi_P$ is similar to that of $\chi$ with the difference that we only use
the simple roots that lie in $M$.

\sssec{}

We consider the twisted $\bN_{\CQ_{G,P}}$-equivariant category of $\Dmod\left(\left(\CQ_{G,P}\times \Ran(X)\right){}_{\on{good}}\right)$, denoted 
$\Dmod\left(\left(\CQ_{G,P}\times \Ran(X)\right){}_{\on{good}}\right)^{\bN_{\CQ_{G,P}},\chi_P}$.

\medskip

As in \propref{p:embed Whit ff Ran}, the forgetful functor 
$$\Dmod\left(\left(\CQ_{G,P}\times \Ran(X)\right){}_{\on{good}}\right)^{\bN_{\CQ_{G,P}},\chi_P}\to
\Dmod\left(\left(\CQ_{G,P}\times \Ran(X)\right){}_{\on{good}}\right)$$
is fully faithful.

\sssec{}

We define the \emph{degenerate Whittaker category} $\on{Whit}(G,P)$ to be the full subcategory of $\Dmod(\CQ_{G,P})$
equal to the preimage of
$$\Dmod\left(\left(\CQ_{G,P}\times \Ran(X)\right){}_{\on{good}}\right)^{\bN_{\CQ_{G,P}},\chi_P}\subset
\Dmod\left(\left(\CQ_{G,P}\times \Ran(X)\right){}_{\on{good}}\right)$$
under the pull-back functor
$$\Dmod(\CQ_{G,P})\to \Dmod\left(\left(\CQ_{G,P}\times \Ran(X)\right){}_{\on{good}}\right).$$

In other words,
\begin{multline*}
\on{Whit}(G,P):=\\
=\Dmod(\CQ_{G,P})\underset{\Dmod\left(\left(\CQ_{G,P}\times \Ran(X)\right){}_{\on{good}}\right)}\times
\Dmod\left(\left(\CQ_{G,P}\times \Ran(X)\right){}_{\on{good}}\right)^{\bN_{\CQ_{G,P}},\chi_P}.
\end{multline*}

\medskip

Note that for $P=G$ we recover the category $\on{Whit}(G,G)$; for $P=B$, we recover the category
$\on{I}(G,B)$. 

\sssec{}

As in the case of $\on{Whit}(G,G)$, the (fully faithful) forgetful functor 
$$\on{Whit}(G,P)\to \Dmod(\CQ_{G,P})$$
admits a right adjoint that we denote $\on{Av}^{\bN,\chi_P}$. 

\medskip

As in the case of $\on{Whit}(G,G)$, we have a canonical action of the monoidal
category $\Rep(\cG)_{\Ran(X)}$ on $\Dmod(\CQ_{G,P})$, and this action
preserves the full subcategory
$$\on{Whit}(G,P)\subset \Dmod(\CQ_{G,P}).$$

Furthermore, the functor $\on{Av}^{\bN,\chi_P}$ commutes with the $\Rep(\cG)_{\Ran(X)}$-action. 

\sssec{}

We define the functor of \emph{degenerate Whittaker coefficient} 
$$\on{coeff}_{G,P}:\Dmod(\Bun_G)\to \on{Whit}(G,P)$$
by
$$\on{coeff}_{G,P}:=\on{Av}^{\bN,\chi_P}\circ (\fr_{G,P})^\dagger.$$

\ssec{Relation between constant term and degenerate Whittaker coefficient functors}

In this subsection we will show how to express the functor $\on{coeff}_{G,P}$, introduced above, via
the functor of enhanced constant term $\on{CT}_P^{\on{enh}}$, introduced in the previous section. 

\sssec{}

Note that we have a naturally defined forgetful map
$$\fr_{P,M}:\CQ_{G,P}\to \Bun_G^{P\on{-gen}},$$
so that
$$\sfp^{\on{enh}}_P\circ {}\fr_{P,M}=\fr_{G,P}.$$

\medskip

In addition to the groupoid $\bN_{\CQ_{G,P}}$ over $\left(\CQ_{G,P}\times \Ran(X)\right){}_{\on{good}}$, there exists a
canonically defined groupoid $\bNP{}_{\CQ_{G,P}}$ that fits into a Cartesian diagram
$$
\CD
\left(\CQ_{G,P}\times \Ran(X)\right){}_{\on{good}}  @<{p_1}<<  \bNP{}_{\CQ_{G,P}} @>{p_2}>> \left(\CQ_{G,P}\times 
\Ran(X)\right){}_{\on{good}}   \\
@V{\fr_{G,P}}VV   @VVV  @VV{\fr_{G,P}}V  \\
\left(\Bun_G^{P\on{-gen}}\times \Ran(X)\right){}_{\on{good}}  @<{p_1}<<  \bNP @>{p_2}>> \left(\Bun_G^{P\on{-gen}}\times \Ran(X)\right){}_{\on{good}} ,
\endCD
$$

\sssec{}

By a slight abuse of notation, let us denote by 
$$\Dmod(\CQ_{G,P})^{\bNP{}_{\CQ_{G,P}}}\subset \Dmod(\CQ_{G,P})$$
the full subcategory, defined in the same way as $\on{Whit}(G,P)$, when instead of the groupoid 
$\bN_{\CQ_{G,P}}$, we use $\bNP{}_{\CQ_{G,P}}$. We shall denote by the 
$$\on{Av}^{\bNP}:\Dmod(\CQ_{G,P})\to \Dmod(\CQ_{G,P})^{\bNP{}_{\CQ_{G,P}}}$$
the right adjoint to the embedding. 

\begin{rem}
The category $\Dmod(\CQ_{G,P})^{\bNP{}_{\CQ_{G,P}}}$ is the geometric counterpart of the space of functions on
$$Z^0_M(K)\cdot N(P)(\BA)\backslash G(\BA)/G(\BO).$$
\end{rem}

\sssec{}

We have a commutative diagram of functors
\begin{equation} \label{e:rgp}
\CD
 \Dmod(\CQ_{G,P})^{\bNP{}_{\CQ_{G,P}}}  @>>> \Dmod(\CQ_{G,P})   \\
@AAA   @A{(\fr_{G,P})^\dagger}AA     \\
\on{I}(G,P) @>>> \Dmod(\Bun_G^{P\on{-gen}}).
\endCD
\end{equation}

By a slight abuse of notation, we shall denote the resulting functor
$$\on{I}(G,P)\to \Dmod(\CQ_{G,P})^{\bNP{}_{\CQ_{G,P}}}$$
by $(\fr_{G,P})^\dagger$.

\begin{lem}
The functor $(\fr_{G,P})^\dagger:\Dmod(\Bun_G^{P\on{-gen}}) \to \Dmod(\CQ_{G,P})$ is fully faithful.
\end{lem}

\begin{proof}
Follows from the homological contractibility of $\bMaps(X,M/Z^0_M)^{\on{gen}}$. 
\end{proof} 

Hence, we obtain that the above functor
$$(\fr_{G,P})^\dagger:\on{I}(G,P)\to \Dmod(\CQ_{G,P})^{\bNP{}_{\CQ_{G,P}}}$$
is also fully faithful.

\medskip

Finally, we note that the diagram 
\begin{equation} \label{e:rgp adj}
\CD
\Dmod(\CQ_{G,P})^{\bNP{}_{\CQ_{G,P}}}  @<{\on{Av}^{\bNP}}<<  \Dmod(\CQ_{G,P})   \\
@AAA   @A{(\fr_{G,P})^\dagger}AA     \\
\on{I}(G,P)  @<{\on{Av}^{\bNP}}<<   \Dmod(\Bun_G^{P\on{-gen}}),  
\endCD
\end{equation} 
obtained from \eqref{e:rgp} by passing to right adjoints along the horizontal arrows,
is also commutative. 

\sssec{}

There is a canonical map of groupoids $\bNP{}_{\CQ_{G,P}}\to \bN_{\CQ_{G,P}}$, and the restriction
of the character $\chi_P$ under this map is trivial. Hence, we obtain an inclusion of full subcategories of
$\Dmod(\CQ_{G,P})$:
$$\on{Whit}(G,P)\hookrightarrow \Dmod(\CQ_{G,P})^{\bNP{}_{\CQ_{G,P}}}.$$

This inclusion admits a right adjoint obtained by restricting the functor $\on{Av}^{\bN,\chi_P}$ to
$\Dmod(\CQ_{G,P})^{\bNP{}_{\CQ_{G,P}}}$. 

\sssec{}

Hence, we obtain a functor 
$$\on{coeff}_{P,M}:\on{I}(G,P)\to \on{Whit}(G,P),$$
defined as
$$\on{coeff}_{P,M}:=\on{Av}^{\bN,\chi_P}\circ (\fr_{G,P})^\dagger.$$

By construction, the functor $\on{coeff}_{P,M}$ respects the action of the monoidal category
$\Rep(\cG)_{\Ran(X)}$.

\begin{rem}
The functor $\on{coeff}_{P,M}$ is \emph{not} fully faithful. However, it follows from Quasi-Theorem 
\ref{t:deg Whit spec} formulated below that its retsriction to the full
subcategory 
$$\on{I}(G,P)_{\on{temp}}\subset \on{I}(G,P)$$
\emph{is} fully faithful, where $\on{I}(G,P)_{\on{temp}}$ is defined via the pull-back square
$$
\CD 
\on{I}(G,P)_{\on{temp}}  @>>>  \on{I}(G,P)  \\
@VVV    @VV{(\imath_M)^\dagger}V  \\
\Dmod(\Bun_M)_{\on{temp}}  @>>>  \Dmod(\Bun_M)
\endCD
$$
\end{rem}

\begin{rem}
The analog of the functor $\on{coeff}_{P,M}$ at the level of functions takes a function on 
$$M(K)\cdot N(P)(\BA)\backslash G(\BA)/G(\BO)$$ and averages it on the left with respect to 
$N(M)(\BA)/N(M)(K)$ against the character $\chi$. 
\end{rem}

\sssec{}

From \eqref{e:rgp adj} we obtain that there exists a canonical isomorphism of functors
$$\Dmod(\Bun_G)\to \on{Whit}(G,P),$$
namely
\begin{equation} \label{e:deg Whit via CT}
\on{coeff}_{G,P}\simeq \on{coeff}_{P,M}\circ \on{CT}_P^{\on{enh}}.
\end{equation}

An intuitive picture behind the functor $\on{coeff}_{G,P}$ will be suggested
in Remark \ref{r:partial Whit}.

\ssec{A strata-wise description}

\sssec{}

Set
$$\CQ_{P,M}:=\CQ_{G,P}\underset{\Bun_G^{P\on{-gen}}}\times \Bun_P,$$
where the map $\Bun_P\to \Bun_G^{P\on{-gen}}$ is $\imath_P$. Denote the resulting map
$$\CQ_{P,M}\to \Bun_P$$ by ${}'\fr_{P,M}$, and the map $\CQ_{P,M}\to \CQ_{G,P}$
by $'\imath_P$. I.e., we have a Cartesian diagram
$$
\CD
\CQ_{G,P}  @<{'\imath_P}<<  \CQ_{P,M}   \\
@V{\fr_{P,M}}VV    @VV{{}'\fr_{P,M}}V   \\
\Bun_G^{P\on{-gen}} @<{\imath_P}<< \Bun_P.
\endCD
$$

\medskip

We have:

\begin{lem}
There exists a canonically defined Cartesian square:
$$
\CD
\CQ_{P,M}  @>{'\sfq_M}>> \CQ_{M,M}  \\
@V{{}'\fr_{P,M}}VV    @VV{\fr_{M,M}}V \\
\Bun_P  @>{\sfq_M}>> \Bun_M
\endCD
$$
\end{lem}

\sssec{}

Consider the stack $\Bun_P$. The groupoid $\bNP$ gives rise to a groupoid $\bNP{}_{\Bun_P}$ acting on 
$\left(\Bun_P\times \Ran(X)\right)_{\on{good}}$. Consider the corresponding full subcategory 
$$\Dmod(\Bun_P)^{\bNP{}_{\Bun_P}}\subset \Dmod(\Bun_P).$$

\medskip

The groupoid $\bN$ gives rise to a groupoid $\bN_{\CQ_{P,M}}$ acting on 
$\left(\CQ_{P,M}\times \Ran(X)\right)_{\on{good}}$. We let 
$$\Dmod(\CQ_{P,M})^{\bN_{\CQ_{P,M}},\chi_P}\subset \Dmod(\CQ_{P,M})$$
denote the resulting full subcategory. 

\medskip

Consider again the map
$$'\sfq_M:\CQ_{P,M}\to \CQ_{M,M}.$$

This map is smooth and has contractible fibers, and we consider the corresponding fully
faithful functor
$$({}'\sfq_M)^\bullet:\Dmod(\CQ_{M,M})\to \Dmod(\CQ_{P,M}).$$

We have:
\begin{lem} \hfill

\smallskip

\noindent{\em(a)} The functor $(\sfq_M)^\bullet$ defines an equivalence
$$\Dmod(\Bun_M)\to \Dmod(\Bun_P)^{\bNP{}_{\Bun_P}}.$$

\smallskip

\noindent{\em(b)} The functor $({}'\sfq_M)^\bullet$ defines an equivalence
$$\Whit(M,M)\to \Dmod(\CQ_{P,M})^{\bN_{\CQ_{P,M}},\chi_P}.$$

\end{lem}

\sssec{}

The functor 
$$({}'\imath_P)^\dagger:\Dmod(\CQ_{G,P})\to \Dmod(\CQ_{P,M})$$
gives rise to a (conservative) functor
$$\Whit(G,P)\to  \Dmod(\CQ_{P,M})^{\bN_{\CQ_{P,M}},\chi_P}.$$

\medskip

We denote the resulting functor
$$\Whit(G,P)\to  \Dmod(\CQ_{P,M})^{\bN_{\CQ_{P,M}},\chi_P}\simeq \Whit(M,M)$$
by $({}'\imath_M)^\dagger$.

\medskip

We have a canonical isomorphism of functors
\begin{equation} \label{e:Whit of CT}
({}'\imath_M)^\dagger\circ \on{coeff}_{G,P}\simeq \on{coeff}_{M,M}\circ \on{CT}_P.
\end{equation} 

\begin{rem}  \label{r:partial Whit}
From \eqref{e:Whit of CT} we obtain the following way of thinking about 
the functor $\on{coeff}_{G,P}$: 

\medskip

In the same way as
the functor $\on{CT}^{\on{enh}}_P$ captures more information than the usual
functor $\on{CT}_P$, the functor $\on{coeff}_{G,P}$ captures more information
than the composition 
$$\on{coeff}_{M,M}\circ \on{CT}_P:\Dmod(\Bun_G)\to \Whit(M,M).$$
\end{rem}

\ssec{Spectral description of the degenerate Whittaker category}

\sssec{}

We have the following assertion: 

\begin{qthm}  \label{t:deg Whit spec} 
There exists a canonically defined fully faithful functor
$$\BL^{\on{Whit}}_{G,P}: \sF_{\cP}\mod(\QCoh(\LocSys_{\cP}))\to \Whit(G,P),$$
compatible
with the actions of the monoidal category $\Rep(\cG)_{\Ran(X)}$. 

\end{qthm}

We note that Quasi-Theorem \ref{t:deg Whit spec} \emph{does not} assume \conjref{c:GL}
for $M$; in particular it includes the case of $P=M=G$. 

\medskip

Note that for $P=G$, the corresponding functor $\BL^{\on{Whit}}_{G,P}$ 
is the functor that we had earlier denoted $\BL^{\on{Whit}}_{G,G}$, and it
is fully faithful by Quasi-Theorem \ref{t:Cass Shal}(b).

\medskip

Note also that in the other extreme case, namely when $P=B$, the assertion
of Quasi-Theorem \ref{t:deg Whit spec} coincides with that of 
Quasi-Theorem \ref{t:principal series}.

\sssec{}

The proof of Quasi-Theorem \ref{t:deg Whit spec} is parallel but simpler than that of 
Quasi-Theorem \ref{t:principal series}. 

\medskip

Namely, we embed both sides into the
corresponding local categories (i.e., ones living over $\Ran(X)$) and use Bezrukavnikov's
theory to relate the resulting category of D-modules on the affine Grassmannian to
the Langlands dual group. 

\sssec{}

From now on, until the end of this subsection, we will assume that \conjref{c:GL}
holds for $M$, and will relate Quasi-Theorem \ref{t:deg Whit spec} to Quasi-Theorem \ref{t:principal series}.

\medskip

The following assertion comes along with the proof:

\begin{prop} \label{p:spec Whit of CT}
We have a commutative diagram of functors:
$$
\CD
\sF_{\cP}\mod(\QCoh(\LocSys_{\cP}))  @>{\BL^{\on{Whit}}_{G,P}}>>   \Whit(G,P)  \\
@A{\Psi_\cP}AA        @AA{\on{coeff}_{P,M}}A   \\
\sF_{\cP}\mod(\IndCoh_{\on{Nilp}^{\on{glob}}_\cP}(\LocSys_\cP))    @>{\BL_P}>>   \on{I}(G,P).
\endCD
$$
\end{prop}

\sssec{}

We can now formulate the following property of the geometric Langlands functor $\BL_G$ that contains Property
$\on{Wh}$ as a particular case for $P=G$:

\medskip

\noindent{\bf Property $\mathbf{Wh}^{\mathbf{deg}}$:} {\it We shall say that the functor $\BL_G$ satisfies
\emph{Property $\on{Wh}^{\on{deg}}$} for the parabolic $P$ if the following diagram is commutative:
\begin{equation} \label{e:cond W deg}
\CD
\sF_{\cP}\mod(\QCoh(\LocSys_{\cP}))    @>{\BL^{\on{Whit}}_{G,P}}>>   \Whit(G,P) \\
@A{\Psi_\cP}AA    & & \\
\sF_{\cP}\mod(\IndCoh_{\on{Nilp}^{\on{glob}}_\cP}(\LocSys_\cP)) & &  @AA{\on{coeff}_{G,P}}A  \\
@A{\on{CT}^{\on{enh}}_{\cP,\on{spec}}}AA  & &  \\
\IndCoh_{\on{Nilp}^{\on{glob}}_\cG}(\LocSys_{\cG})   @>{\BL_G}>> \Dmod(\Bun_G).
\endCD
\end{equation}}

Note, however, that Property $\on{Wh}^{\on{deg}}$ is a formal consequence of Property $\on{Ei}^{\on{enh}}$
and \propref{p:spec Whit of CT}. 

\ssec{(Degenerate) Whittaker coefficients and Eisenstein series}

Let $P'\subset G$ be another parabolic. In this subsection we will assume that \conjref{c:GL} holds for its
Levi quotient $M'$. However, we will \emph{not} be assuming that \conjref{c:GL} holds for $M$. 

\sssec{}

By concatenating the commutative diagrams \eqref{e:cond W deg} and \eqref{e:CT spec rat} we obtain a commutative diagram
\begin{equation} \label{e:Whit and Eis}
\CD
\sF_{\cP}\mod(\QCoh(\LocSys_{\cP}))  @>{\BL^{\on{Whit}}_{G,P}}>>  \Whit(G,P)    \\
@A{\Psi_\cP}AA   & &  \\
\sF_{\cP}\mod(\IndCoh_{\on{Nilp}^{\on{glob}}_\cP}(\LocSys_{\cP})) & & @AA{\on{coeff}_{G,P}}A   \\
@A{\on{CT}^{\on{enh}}_{\cP,\on{spec}}}AA   & &  \\ 
\IndCoh_{\on{Nilp}^{\on{glob}}_\cG}(\LocSys_{\cG})  & &  \Dmod(\Bun_G)  \\
@A{\Eis^{\on{enh}}_{\cP',\on{spec}}}AA    @AA{\Eis_{P'}^{\on{enh}}}A   \\
\sF_{\cP'}\mod(\IndCoh_{\on{Nilp}_{\on{glob},\cP'}}(\LocSys_{\cP'}))  @>{\BL_{P'}}>>  \on{I}(G,P')
\endCD
\end{equation}

\sssec{}

We have:

\begin{qthm} \label{t:Whit and Eis}
The diagram \eqref{e:Whit and Eis} commutes unconditionally (i.e., without assuming the validity of \conjref{c:GL} for $G$).
\end{qthm}

\begin{rem}
Note that if we do assume that \conjref{c:GL} holds for $M$, then in this case the assertion of
Quasi-Theorem \ref{t:Whit and Eis} follows from Quasi-Theorem \ref{t:CT and Eis} and 
\propref{p:spec Whit of CT}.
\end{rem}

\sssec{}

Finally, we remark that for $P=G$, the assertion of Quasi-Theorem \ref{t:Whit and Eis} is built in the proof of
Quasi-Theorem \ref{t:deg Whit spec} (for $P'$). 

\section{The extended Whittaker model}  \label{s:ext Whit}

In this section we will introduce a crucial player for our approach to the geometric Langlands conjecture,
the extended Whittaker category, denoted $\Whit^{\on{ext}}(G,G)$. The idea is that, on the one hand, according to \conjref{c:Whit ext ff}, 
discussed below, the category $\Whit^{\on{ext}}(G,G)$ receives a fully faithful functor from the automorphic category $\Dmod(\Bun_G)$,
and on the other hand, it can be related to the spectral side. How the latter is done will be the subject of \secref{s:gluing}.

\ssec{The variety of characters}  \label{ss:ch}

When defining degenerate Whittaker categories, we had to consider characters of the group $N$, whose degeneracies
varied with the parabolic. In this subsection we will combine all these categories into one family. 

\sssec{}

Let $\ft_{\on{adj}}$ denote the Lie algebra of the torus $T/Z^0_G$. In this subsection we will introduce
a certain toric variety $\fch(G)$ endowed with a finite map 
\begin{equation} \label{e:finite map}
\fch(G)\to \ft_{\on{adj}}.
\end{equation}

The map \eqref{e:finite map} will be an isomorphism when $G$ has a connected center.

\sssec{}

Let $\Lambda$ denote the weight lattice of $G$; let 
$$\Lambda^{\on{pos}}\subset \Lambda \text{ and } 
\Lambda^{\on{pos},\BQ}\subset \Lambda^\BQ:=\Lambda\underset{\BZ}\otimes \BQ$$
be the sub-monoids of weights that can be expressed as integral (resp., rational) non-negative combinations 
of simple roots. Let $\Lambda^{\on{pos},\on{sat}_G}$
be the saturation of $\Lambda^{\on{pos}}$, i.e.,
$$\Lambda^{\on{pos},\on{sat}_G}:=\Lambda\cap \Lambda^{\on{pos},\BQ}.$$

Note that the inclusion $\Lambda^{\on{pos}}\hookrightarrow \Lambda^{\on{pos},\on{sat}_G}$
is an equality if $G$ has a connected center. 

\sssec{}

We define
$$\fch(G):=\Spec(k[\Lambda^{\on{pos},\on{sat}_G}]).$$
I.e., $\fch(G)$ classifies maps of monoids $\Lambda^{\on{pos},\on{sat}_G}\to \BA^1$, where $\BA^1$ is a monoid
with respect to the operation of multiplication. 

\medskip

The group $T/Z^0_G$, which can be thought of that classifying maps of monoids 
$\Lambda^{\on{pos},\on{sat}_G}\to \BG_m$, acts on $\fch(G)$.

\medskip

Let $\ofc(G)\subset \fch(G)$ be the open subscheme corresponding to maps $\Lambda^{\on{pos},\on{sat}_G}\to (\BA^1-0)=\BG_m$.
It is clear that the action of $T/Z^0_G$ on $\ofc(G)$ is simply transitive. 

\sssec{}

Let $P\subset G$ be a parabolic, with Levi quotient $M$. Consider the closed subscheme of $\fch(G)$ that
corresponds to maps $\Lambda^{\on{pos},\on{sat}_G}\to \BA^1$ that vanish on any element $\mu$ with 
$$\mu\in \Lambda^{\on{pos},\on{sat}_G}- \Lambda^{\on{pos},\on{sat}_M}.$$

It is easy to see that this subscheme identifies with the corresponding scheme $\fch(M)$, in a way compatible
with the actions of 
$$T/Z^0_G\twoheadrightarrow T/Z^0_M.$$

\medskip

Furthermore, it is clear that $\fch(G)$ decomposes as a union of locally closed subschemes
$$\fch(G)\simeq \underset{P}\sqcup\,\, \ofc(M).$$

\ssec{The extended Whittaker category}

In this subsection we will finally define the extended Whittaker category $\Whit^{\on{ext}}(G,G)$. The definition
will follow the same pattern as in the case of $\Whit(G)$, $\on{I}(G,P)$ and $\Whit(G,P)$.

\sssec{}

We define the prestack $\CQ_{G,G}^{\on{ext}}$ as follows. The definition repeats that of $\CQ_{G,G}$ with
the following difference: when considering quadruples $(\CP_G,U,\alpha,\gamma)$, we let $\gamma$ be
a section over $U$ of the scheme $\fch(G)_{\check\rho(\omega_X)|_U\otimes \CP^{-1}_T}$. 

\medskip

In other words, the datum of $\gamma$ assigns to every $\mu\in \Lambda^{\on{pos},\on{sat}_G}$ a map
of line bundles over $U$:
\begin{equation} \label{e:gamma mu}
\gamma(\mu):\mu(\CP_T)\to (\omega_X^{\frac{1}{2}})^{\otimes \langle \mu,2\check\rho\rangle}|_U.
\end{equation}

\medskip

(Note that when $G$ has a connected center, the datum of $\gamma$ amounts to a map
$\alpha_i(\CP_T)\to \omega_X|_U$ for every simple root $\alpha_i$ of $G$.)

\medskip

We let $\fr_{G,G}^{\on{ext}}$ denote the forgetful map $\CQ^{\on{ext}}_{G,G}\to \Bun_G$.

\sssec{}

The groupoid of $k$-points of $\CQ^{\on{ext}}_{G,G}$ identifies with the quotient
$$T(K)\backslash \Bigl(N(K)\backslash G(\BA)/G(\BO)\times \fch(G)(K)\Bigr),$$
where $T$ acts on $\fch$ via the projection $T\twoheadrightarrow T/Z^0_G$. 

\sssec{}

We let $\bN_{\CQ^{\on{ext}}_{G,G}}$ the groupoid on $\left(\CQ^{\on{ext}}_{G,G}\times \Ran(X)\right)_{\on{good}}$ obtained by lifting the
groupoid $\bN$ on $\left(\Bun_G^{B\on{-gen}}\times \Ran(X)\right)_{\on{good}}$. 

\medskip

As in the case of $\CQ_{G,G}$, the groupoid $\bN_{\CQ^{\on{ext}}_{G,G}}$ is endowed with a canonical
character $\chi^{\on{ext}}$ with values in $\BG_a$. 

\sssec{}

We consider the twisted $\bN_{\CQ^{\on{ext}}_{G,G}}$-equivariant category of 
$\Dmod\left(\left(\CQ^{\on{ext}}_{G,G}\times \Ran(X)\right)_{\on{good}}\right)$,
and as in \propref{p:embed Whit ff Ran}, the forgetful functor 
$$\Dmod\left(\left(\CQ^{\on{ext}}_{G,G}\times \Ran(X)\right)_{\on{good}}\right)^{\bN_{\CQ^{\on{ext}}_{G,G}},\chi^{\on{ext}}}\to
\Dmod\left(\left(\CQ^{\on{ext}}_{G,G}\times \Ran(X)\right)_{\on{good}}\right)$$
is fully faithful.

\sssec{}

We define the \emph{extended Whittaker category} $\Whit^{\on{ext}}(G,G)$ as the preimage of
$$\Dmod\left(\left(\CQ^{\on{ext}}_{G,G}\times \Ran(X)\right)_{\on{good}}\right)^{\bN_{\CQ^{\on{ext}}_{G,G}},\chi^{\on{ext}}}\subset
\Dmod\left(\left(\CQ^{\on{ext}}_{G,G}\times \Ran(X)\right)_{\on{good}}\right)$$
under the pull-back functor
$$\Dmod(\CQ^{\on{ext}}_{G,G})\to \Dmod\left(\left(\CQ^{\on{ext}}_{G,G}\times \Ran(X)\right)_{\on{good}}\right).$$

I.e.,
\begin{multline*}
\Whit^{\on{ext}}(G,G):=\\
=\Dmod(\CQ^{\on{ext}}_{G,G})\underset{\Dmod\left(\left(\CQ^{\on{ext}}_{G,G}\times \Ran(X)\right)_{\on{good}}\right)}\times
\Dmod\left(\left(\CQ^{\on{ext}}_{G,G}\times \Ran(X)\right)_{\on{good}}\right)^{\bN_{\CQ^{\on{ext}}_{G,G}},\chi^{\on{ext}}}.
\end{multline*}

\sssec{}

As in the case of $\Whit(G,G)$, the (fully faithful) forgetful functor 
$$\Whit^{\on{ext}}(G,G)\to \Dmod(\CQ^{\on{ext}})$$
admits a right adjoint, that we denote by $\on{Av}^{\bN,\chi^{\on{ext}}}$.

\medskip

We observe that as in \propref{p:Hecke preserves Whit}, we have a canonical action of the monoidal category 
$\Rep(\cG)_{\Ran(X)}$ on $\Dmod(\CQ^{\on{ext}})$ that preserves the full subcategory
$$\Whit^{\on{ext}}(G,G)\subset \Dmod(\CQ^{\on{ext}})$$
and commutes with the functor $\on{Av}^{\bN,\chi^{\on{ext}}}$.

\sssec{}

We introduce the functor of extended Whittaker coefficient
$$\on{coeff}^{\on{ext}}_{G,G}:\Dmod(\Bun_G)\to \Whit^{\on{ext}}(G,G)$$
to be
$$\on{coeff}^{\on{ext}}_{G,G}:=\on{Av}^{\bN,\chi^{\on{ext}}}\circ (\fr^{\on{ext}}_{G,G})^\dagger.$$

By construction, the functor $\on{coeff}^{\on{ext}}_{G,G}$ is compatible with the action of the
monoidal category $\Rep(\cG)_{\Ran(X)}$.

\sssec{}

We propose the following crucial conjecture:

\begin{conj}  \label{c:Whit ext ff}
The functor $\on{coeff}^{\on{ext}}_{G,G}$ is fully faithful.
\end{conj}

\medskip

We have:

\begin{thm}  \label{t:Whit for GLn}
Conjecture \ref{c:Whit ext ff} holds for $G=GL_n$.
\end{thm}

This theorem has been recently established by D.~Beraldo. The proof uses the mirabolic subgroup and the
classical strategy of expressing the functor $\on{coeff}^{\on{ext}}_{G,G}$ as a composition
of $n-1$ Fourier transform functors.

\ssec{Extended vs. degenerate Whittaker models}

\sssec{}

Let $P$ be a parabolic in $G$ with Levi quotient $M$. Note that we have a canonically defined locally closed embedding
of prestacks:
$$\bi_P:\CQ_{G,P}\to \CQ^{\on{ext}}_{G,G}.$$

Namely, it corresponds to the locally closed subscheme $\ofc(M)\subset \fch(G)$. In other words, 
$\Maps(S,\CQ_{G,P})$ is a subgroupoid of $\Maps(S,\CQ^{\on{ext}}_{G,G})$, corresponding to those
$(\CP_G,U,\alpha,\gamma)$, for which the maps $\gamma(\mu)$ of \eqref{e:gamma mu} satisfy:

\smallskip

\begin{itemize}

\item For $\mu\notin \Lambda^{\on{pos},\on{sat}_M}$, we have $\gamma(\mu)=0$.

\item For $\mu\in  \Lambda^{\on{pos},\on{sat}_M}$, the map $\gamma(\mu)$ is an isomorphism
(possibly, after shrinking the open subset $U$). 

\end{itemize} 

\sssec{}

For $P=G$ we will sometimes use the notation $\bj$ instead of $\bi_G$, to emphasize that
we are dealing with an open embedding. 

\medskip

For the same reason, we will use the notation
$\bj^\bullet$ instead of $\bj^\dagger$. The functor $\bj^\bullet$ admits a \emph{right} adjoint, denoted
$\bj_\bullet$, given by the D-module direct image.

\sssec{}

The restriction of the groupoid $\bN_{\CQ^{\on{ext}}_{G,G}}$ to $$\left(\CQ_{G,P}\times \Ran(X)\right){}_{\on{good}}$$ identifies
with $\bN_{\CQ_{G,P}}$, and the character $\chi^{\on{ext}}$ restricts to $\chi_P$. 

\medskip

Hence, the functor $(\bi_P)^\dagger$ gives rise to a functor
$$(\bi_P)^\dagger:\Whit^{\on{ext}}(G,G)\to \Whit(G,P).$$

One shows that the partially defined left adjoint $(\bi_P)_\dagger$ to $(\bi_P)^\dagger$ is well-defined 
on the full subcategory
$$ \Whit(G,P)\subset \Dmod(\CQ_{G,P}).$$

Hence, we obtain a functor
$$(\bi_P)_\dagger:\Whit(G,P)\to \Whit^{\on{ext}}(G,G),$$
which is fully faithful, since $\bi_P$ is a locally closed embedding. 

\sssec{}

In particular, the functor $\on{coeff}_{G,G}^{\on{ext}}$ contains the information of all
the functors $\on{coeff}_{G,P}$:

\begin{equation} \label{e:deg via ext}
\on{coeff}_{G,P}\simeq (\bi_P)^\dagger\circ \on{coeff}^{\on{ext}}_{G,G}.
\end{equation}

Note that we have the following consequence of \conjref{c:Whit ext ff}:

\begin{corconj}  \label{cc:coeff 0}
Let $\CM\in \Dmod(\Bun_G)$ be such that $\on{coeff}_{G,P}(\CM)=0$
for all parabolics $P$ (including $P=G$). Then $\CM=0$.
\end{corconj}

\ssec{Cuspidality}

\sssec{}

We shall call an object $\CM\in \Dmod(\Bun_G)$ \emph{cuspidal}
if it is annihilated by the functors $\on{CT}_P$ for all \emph{proper} parabolics $P$ of $G$.
We let
$$\Dmod(\Bun_G)_{\on{cusp}}\subset \Dmod(\Bun_G)$$
the full subcategory spanned by cuspidal objects. 

\sssec{}

Note that since for a given parabolic $P$, the functor $\imath^\dagger_P$ is conservative,
an object $\CM\in \Dmod(\Bun_G)$ is annihilated by $\on{CT}_P$ if and only if it is
annihilated by $\on{CT}^{\on{enh}}_P$.

\medskip 

From \eqref{e:deg Whit via CT} we obtain that if $\CM$ is cuspidal then all $\on{coeff}_{G,P}(\CM)$
(for $P$ being a proper parabolic) are zero. In particular, we have:

\begin{cor}  \label{c:cusp}
Let $\CM\in \Dmod(\Bun_G)$ be cuspidal. Then the canonical map
$$\on{coeff}^{\on{ext}}_{G,G}(\CM)\to \bj_\bullet\circ \bj^\bullet (\on{coeff}^{\on{ext}}_{G,G}(\CM))$$
is an isomorphism.
\end{cor}

\sssec{}

Note, however, that from \eqref{e:Whit of CT} and \corref{cc:coeff 0} (applied to proper Levi subgroups of $G$),
we obtain:

\begin{corconj}
If $\CM\in \Dmod(\Bun_G)$ is such that $\on{coeff}_{G,P}(\CM)=0$ for all proper parabolics $P$, then
$\CM$ is cuspidal. 
\end{corconj}

And, hence:

\begin{corconj}
Let $\CM\in \Dmod(\Bun_G)$ be such that the map
$$\on{coeff}^{\on{ext}}_{G,G}(\CM)\to \bj_\bullet\circ \bj^\bullet (\on{coeff}^{\on{ext}}_{G,G}(\CM))$$
is an isomorphism. Then $\CM$ is cuspidal.
\end{corconj}

\section{The gluing procedure}  \label{s:gluing}

In this section we will match the category $\Whit^{\on{ext}}(G,G)$ with a category that can be described 
purely in spectral terms. 

\ssec{Gluing of DG categories, a digression}

In subsection we will describe the general paradigm in which one can define the procedure of gluing 
of DG categories.

\sssec{}

Let $A$ be an index category, and let $\bC$ 
$$(a\in A) \mapsto  \bC_a,\quad (a_1\overset{\phi}\to a_2) \mapsto (\bC_{a_1}\overset{\bC_\phi}\longrightarrow \bC_{a_2})$$
be a \emph{lax} diagram of DG categories, parameterized by $A$.

\medskip

Informally, this means that for a pair of composable arrows 
$$a_1\overset{\phi}\longrightarrow a_2 \overset{\psi}\longrightarrow a_3$$
we have a \emph{natural transformation}  (but not necessarily an isomorphism)
\begin{equation} \label{e:gluing nat trans}
\bC_{\psi}\circ \bC_{\phi}\to \bC_{\psi\circ \phi},
\end{equation} 
equipped with a homotopy-coherent system of compatibilities for higher-order compositions.

\medskip

In the $\infty$-categorical language, we should think of $\bC$ as a category $\bC_A$, equipped with a functor
to $A$, which is a \emph{locally co-Cartesian fibration}.

\sssec{}  \label{sss:glue}

To $\bC$ as above we assign its \emph{lax limit} 
$$\on{Glue}(\bC)\in \on{DGCat}_{\on{cont}}.$$

\medskip

In the $\infty$-categorical language, $\on{Glue}(\bC)$ is the category of sections of the functor $\bC_A\to A$. 

\medskip

One can characterize $\on{Glue}(\bC)$ by the following universal property. For $\bD\in  \on{DGCat}_{\on{cont}}$, 
the datum of a continuous functor 
$$F:\bD\to \on{Glue}(\bC)$$
is equivalent to that of a collection of continuous functors 
$$F_a:\bD\to \bC_a,\quad a\in A,$$
equipped with a compatible system of natural transformations
$$\bC_\phi\circ F_{a_1}\overset{F_\phi}\longrightarrow F_{a_2} \text{ for } a_1\overset{\phi}\to a_2.$$

Note, however, that we \emph{do not} require that the natural transformations $F_\phi$ be isomorphisms. 

\medskip

Taking $\bD$ to be $\Vect$, we obtain a description of the $\infty$-groupoid of objects of $\on{Glue}(\bC)$. These are assignments
$$(a\in A) \mapsto \bc_a\in \bC_a, \quad (a_1\overset{\phi}\to a_2) \mapsto (\bC_\phi(\bc_{a_1}) \overset{\bc_\phi}\longrightarrow \bc_{a_2}),$$
equipped with a homotopy-coherent system of compatibilities for higher-order compositions.

\begin{rem}
The category $\on{Glue}(\bC)$ contains a full subcategory, denoted $\on{Glue}(\bC)^{\on{strict}}$, that consists of 
those assigments for which the maps $\bc_\phi$ above are isomorphisms. 

\medskip 

If $\bC$ was itself a strict functor $A\to \on{DGCat}_{\on{cont}}$ (i.e., if the natural transformations \eqref{e:gluing nat trans}
were isomorphisms, or equivalently $\bC_A\to A$ was a \emph{co-Cartesian fibration}), then 
$\on{Glue}(\bC)^{\on{strict}}$ identifies with the limit of $\bC$, 
$$\underset{a\in A}{\on{lim}}\, \bC_a \in \on{DGCat}_{\on{cont}}.$$ 
\end{rem} 

\sssec{}

We have the natural evaluation functors 
$$\on{ev}_{a}: \on{Glue}(\bC) \to \bC_{a},\quad a\in A.$$
These functors admit left adjoints, denoted $\on{ins}_{a}$\footnote{The notation ``$\on{ins}$" is for ``insert".}.

\medskip

Explicitly, the composition 
$$\on{ev}_{a_2}\circ \on{ins}_{a_1}:\bC_{a_1}\to \bC_{a_2}$$
is calculated as the colimit in $\on{Funct}_{\on{cont}}(\bC_{a_1},\bC_{a_2})$ over
the $\infty$-groupoid $\Maps_{A}(a_1,a_2)$ of the functor
$$(\phi\in \Maps_{A}(a_1,a_2))\mapsto (\bC_\phi\in \on{Funct}_{\on{cont}}(\bC_{a_1},\bC_{a_2})).$$

\medskip

In particular, we have:

\begin{lem} \label{l:ins}
Suppose that $a\in A$ is such that $\Maps_{A}(a,a)$ contractible. Then the functor $\on{ins}_a$
is fully faithful.
\end{lem}

\sssec{}  \label{sss:glue top}

Here are is a typical example of the above situation. Let $\CY$ be a topological space and let
$$\CY=\underset{a\in A} \cup\, \CY_a$$
be its decomposition into locally closed subsets, indexed by a poset $A$, so that 
$$\CY_{a_1}\cap \ol\CY_{a_2}\neq \emptyset \, \Rightarrow\, a_1\geq a_2.$$

For each index $a$ let $\bi_a$ denote the corresponding locally closed embedding, and let
$$(\bi_a)_\dagger:\on{Shv}(\CY_a)\rightleftarrows \on{Shv}(\CY):(\bi_a)^\dagger$$
be the corresponding adjoint pair.

\medskip

We define the diagram $\bC$ by sending
$a\mapsto  \on{Shv}(\CY_a)$ and 
$(a_1\leq a_2)$ to the functor 
$$(\bi_{a_2})^\dagger\circ (\bi_{a_1})_\dagger:\on{Shv}(\CY_{a_1})\to \on{Shv}(\CY_{a_2}).$$

\medskip

Consider the resulting category $\on{Glue}(\bC)$. We have a naturally defined functor
\begin{equation} \label{e:glue sheaves}
\on{Shv}(\CY)\to \on{Glue}(\bC),
\end{equation}
given by sending 
$a\mapsto (\bi_a)^\dagger$ and $(a_1\leq a_2)$ to the natural transformation
$$(\bi_{a_2})^\dagger\circ (\bi_{a_1})_\dagger \circ (\bi_{a_1})^\dagger \to
(\bi_{a_2})^\dagger.$$

It is well known that the functor \eqref{e:glue sheaves} is an equivalence. This is the source of the name ``gluing" for the construction of
\secref{sss:glue}.

\sssec{}

Let now $F:\bC'\to \bC''$ be a \emph{lax} natural transformation. Informally, this means having a collection of functors
$$F_a:\bC'_a\to \bC''_a, \quad a\in A,$$
equipped with natural transformations 
\begin{equation} \label{e:funct nat trans}
\bC''_\phi\circ F_{a_1}\to F_{a_2}\circ \bC'_\phi, \quad a_1\overset{\phi}\longrightarrow a_2,
\end{equation} 
and a homotopy-coherent system of compatibilities for higher-order compositions.

\medskip

In the $\infty$-categorical language, the datum of $F$ amounts to that of a 
functor $F_A:\bC'_A\to \bC''_A$, compatible with the projections to $A$. 

\medskip

We shall say that $F$ is \emph{strict} if the natural transformations \eqref{e:funct nat trans} are isomorphisms.
In the $\infty$-categorical language, this can be formulated as saying that $F_A$ takes co-Cartesian arrows to
co-Cartesian arrows. 

\medskip

Given $F$ as above, we have a naturally defined functor
$$\on{Glue}(F): \on{Glue}(\bC')\to \on{Glue}(\bC'').$$

\medskip

We have:

\begin{lem}  \label{l:glue fully faithful}
Assume that each of the functors $F_a:\bC'_a\to \bC''_a$ is fully faithful and that $F$ is strict. 
Then $\on{Glue}(F)$ is fully faithful. 
\end{lem}  

\ssec{The extended Whittaker model as a glued category}

In this subsection we will see that the category $\Whit^{\on{ext}}(G,G)$, introduced in \secref{s:ext Whit}, can be naturally
obtained by a gluing procedure from the categories $\Whit(G,P)$, 

\sssec{}

We let $A$ be the category $\on{Par}(G)$ opposite to the poset of standard parabolics of $G$. For each parabolic we consider the category
$$\Whit(G,P).$$

We extend the assignment $P\mapsto \Whit(G,P)$ to a lax diagram of DG categories, parameterized by $\on{Par}(G)$, 
by sending an inclusion $P_1\subset P_2$ to the functor
$$(\bi_{P_1})^\dagger\circ (\bi_{P_2})_\dagger.$$ 

Let $\on{Glue}(G)_{\on{geom}}$ denote the resulting lax limit category. 

\sssec{}

We have a naturally defined functor
$$\Whit^{\on{ext}}(G,G)\to \on{Glue}(G)_{\on{geom}}$$
corresponding to the collection of functors $(\bi_{P_1})^\dagger$.

\medskip

As in \secref{sss:glue top}, we have:

\begin{lem} \label{l:glue Whit}
The above functor $\Whit^{\on{ext}}(G,G)\to \on{Glue}(G)_{\on{geom}}$ is an equivalence.
\end{lem}

By definition, the resulting adjoint pair of functors 
$$\on{ins}_P:\Whit(G,P)\rightleftarrows \on{Glue}(G)_{\on{geom}}:\on{ev}_P$$
identifies with the adjoint pair $(\bi_P)_\dagger,(\bi_P)^\dagger$. 

\ssec{The glued category on the spectral side}  \label{ss:gluing spec}

In this subsection we will perform another construction, crucial for our approach to geometric Langlands.

\medskip

We will show that the category $\IndCoh_{\on{Nilp}^{\on{glob}}_\cG}(\LocSys_\cG)$, appearing 
on the spectral side of the correspondence can be embedded into a category, obtained by a gluing 
procedure from the $\QCoh$-categories for the parabolics of $G$. 

\medskip

This gives a precise expression to 
the idea that the difference between the categories $\IndCoh_{\on{Nilp}^{\on{glob}}_\cG}(\LocSys_\cG)$ and
$\QCoh(\LocSys_\cG)$ is captured by the proper parabolics of $G$. 

\sssec{}

Consider again the category $\on{Par}(G)$.  For each parabolic we consider the category
$$\sF_{\cP}\mod(\QCoh(\LocSys_{\cP})).$$

We are now going to upgrade the assignment 
$$P\mapsto \sF_{\cP}\mod(\QCoh(\LocSys_{\cP}))$$
to a lax diagram of DG categories, parameterized by $\on{Par}(G)$. 

\sssec{}

For $P_1\subset P_2$, let $\sfp_{P_1/P_2,\on{spec}}$ denote the corresponding map
$$\LocSys_{\cP_1}\to \LocSys_{\cP_2}.$$

As in \secref{ss:spectral parabolic}, we have an adjoint pair of functors
$$(\sfp_{P_1/P_2,\on{spec}})^{\IndCoh}_*:\IndCoh_{\on{Nilp}^{\on{glob}}_{\cP_1}}(\LocSys_{\cP_1})\rightleftarrows 
\IndCoh_{\on{Nilp}^{\on{glob}}_{\cP_2}}(\LocSys_{\cP_2}):\sfp_{P_1/P_2,\on{spec}}^!,$$
and the same-named pair of functors
\begin{multline*}
(\sfp_{P_1/P_2,\on{spec}})^{\IndCoh}_*:\sF_{P_1}\mod(\IndCoh_{\on{Nilp}^{\on{glob}}_{\cP_1}}(\LocSys_{\cP_1}))
\rightleftarrows  \\
\sF_{P_2}\mod(\IndCoh_{\on{Nilp}^{\on{glob}}_{\cP_2}}(\LocSys_{\cP_2})):\sfp_{P_1/P_2,\on{spec}}^!
\end{multline*}
that commute with the forgetful functors
$$\oblv_{\sF_{P_i}}:
\sF_{P_i}\mod(\IndCoh_{\on{Nilp}^{\on{glob}}_{\cP_i}}(\LocSys_{\cP_i}))\to
\IndCoh_{\on{Nilp}^{\on{glob}}_{\cP_i}}(\LocSys_{\cP_i}).$$

\sssec{}

Recall also the functors
$$\Xi_{\cP_i}:\sF_{P_i}\mod(\QCoh(\LocSys_{\cP_i}))\rightleftarrows
\sF_{P_i}\mod(\IndCoh_{\on{Nilp}^{\on{glob}}_{\cP_i}}(\LocSys_{\cP_i})):\Psi_{\cP_i}.$$

\medskip

We define the functor 
$$\sF_{P_2}\mod(\QCoh(\LocSys_{\cP_2}))\to 
\sF_{P_1}\mod(\QCoh(\LocSys_{\cP_1}))$$
to be the composition
\begin{equation} \label{e:spectral gluing functors}
\Psi_{\cP_1}\circ \sfp_{P_1/P_2,\on{spec}}^!\circ \Xi_{\cP_2}.
\end{equation}

\sssec{}

We denote the resulting lax limit category by 
$$\on{Glue}(\cG)_{\on{spec}}.$$

\medskip

For a parabolic $P$, we let $\on{ev}_{\cP,\on{spec}}$ denote the corresponding evaluation
functor 
$$\on{Glue}(\cG)_{\on{spec}}\to 
\sF_{\cP}\mod(\QCoh(\LocSys_{\cP})),$$
and by $\on{ins}_{\cP,\on{spec}}$ its left adjoint. 

\medskip

By \lemref{l:ins}, the functors $\on{ins}_{\cP,\on{spec}}$ are fully faithful. 

\medskip

Note that since the functors \eqref{e:spectral gluing functors} are compatible with the
action of the monoidal category $\QCoh(\LocSys_\cG)$, the category $\on{Glue}(\cG)_{\on{spec}}$ also
acquires a $\QCoh(\LocSys_\cG)$-action.

\sssec{}

We now claim that there exists a canonically defined functor
$$\on{Glue}(\on{CT}^{\on{enh}}_{\on{spec}}):
\IndCoh_{\on{Nilp}^{\on{glob}}_\cG}(\LocSys_\cG)\to \on{Glue}(\cG)_{\on{spec}}.$$

Namely, it is given by the collection of functors for each parabolic $P$:
\begin{multline*}
\IndCoh_{\on{Nilp}^{\on{glob}}_\cG}(\LocSys_\cG)\overset{\on{CT}^{\on{enh}}_{\cP,\on{spec}}}\longrightarrow \\
\to \sF_{\cP}\mod(\IndCoh_{\on{Nilp}^{\on{glob}}_\cP}(\LocSys_{\cP}))\overset{\Psi_{\cP}}\longrightarrow
\sF_{\cP}\mod(\QCoh(\LocSys_{\cP})).
\end{multline*}

By construction, the functor $\on{Glue}(\on{CT}^{\on{enh}}_{\on{spec}})$ respects the action of the monoidal
category $\QCoh(\LocSys_\cG)$.

\sssec{}

We propose:

\begin{conj}  \label{c:coLoc ext ff}
The functor $\on{Glue}(\on{CT}^{\on{enh}}_{\on{spec}})$ is fully faithful.
\end{conj}

The following has been recently proved by D.~Arinkin and the author:

\begin{thm} \label{t:coLoc ext ff}
\conjref{c:coLoc ext ff} holds for all reductive groups $G$.
\end{thm} 

\ssec{Extended Whittaker compatibility}

Having expressed $\Whit^{\on{ext}}(G,G)$ as a glued category, we can now relate it to the spectral side. 
This will be done in the present subsection. 

\sssec{}

We now make the following crucial statement:

\begin{qthm}  \label{t:glued equiv}  \hfill

\smallskip

\noindent{\em(a)}
The assignment that sends a parabolic $P$ to the functor
$$\BL^{\on{Whit}}_{G,P}:\sF_{\cP}\mod(\QCoh(\LocSys_{\cP}))\to \Whit(G,P)$$
extends to a \emph{strict} natural transformation of the corresponding \emph{lax} diagrams.

\smallskip

\noindent{\em(b)}
The resulting functor $\BL^{\on{Whit}^{\on{ext}}}_{G,G}$
$$\on{Glue}(\cG)_{\on{spec}}\to \on{Glue}(G)_{\on{geom}}\simeq \Whit^{\on{ext}}(G,G)$$
is compatible with the actions of $\Rep(\cG)_{\Ran(X)}$.

\end{qthm}

\begin{rem}

In fact, Quasi-Theorem \ref{t:glued equiv} is a theorem modulo Quasi-Theorem \ref{t:principal series}. By definition,
its statement amounts to a compatible family of commutative diagrams
$$
\CD
\sF_{P_2}\mod(\QCoh(\LocSys_{\cP_2}))   @>{\BL^{\on{Whit}}_{G,P_2}}>>  \Whit(G,P_2)  \\
@V{\Psi_{\cP_1}\circ \sfp_{P_1/P_2,\on{spec}}^!\circ \Xi_{\cP_2}}VV     @VV{(\bi_{P_1})^\dagger\circ (\bi_{P_2})_\dagger}V    \\
\sF_{P_1}\mod(\QCoh(\LocSys_{\cP_1}))  @>{\BL^{\on{Whit}}_{G,P_1}}>>    \Whit(G,P_1)
\endCD
$$
for $P_1\subset  P_2$. Thus, the proof of Quasi-Theorem \ref{t:glued equiv} amounts an explicit understanding of the gluing
functors 
$$(\bi_{P_1})^\dagger\circ (\bi_{P_2})_\dagger:\Whit(G,P_2)\to \Whit(G,P_1).$$

\end{rem}

\sssec{}

Combined with Lemmas \ref{l:glue Whit} and \ref{l:glue fully faithful}, Quasi-Theorem \ref{t:glued equiv} implies:

\begin{qthm}  \label{t:glued functor}
The functor 
$$\BL^{\on{Whit}^{\on{ext}}}_{G,G}:\on{Glue}(\cG)_{\on{spec}}\to \Whit^{\on{ext}}(G,G)$$
is fully faithful.
\end{qthm} 

\sssec{}  \label{sss:cond W ext}

We are now ready to state Property $\on{Wh}^{\on{ext}}$ of the geometric Langlands 
functor $\BL_G$ in \conjref{c:GL}:

\medskip

\noindent{\bf Property $\mathbf{Wh}^{\mathbf{ext}}$:} {\it We shall say that the functor
$\BL_G$ satisfies \emph{Property $\on{Wh}^{\on{ext}}$} if the following diagram is commutative:
\begin{equation} \label{e:cond W ext}
\CD
\on{Glue}(\cG)_{\on{spec}}   @>{\BL^{\on{Whit}^{\on{ext}}}_{G,G}}>>   \Whit^{\on{ext}}(G,G)  \\
@A{\on{Glue}(\on{CT}^{\on{enh}}_{\on{spec}})}AA    @AA{\on{coeff}^{\on{ext}}_{G,G}}A    \\
\IndCoh_{\on{Nilp}^{\on{glob}}_\cG}(\LocSys_\cG)   @>{\BL_G}>> \Dmod(\Bun_G).
\endCD
\end{equation}}

Note that Property $\on{Wh}^{\on{ext}}$ contains as a particular case Property $\on{Wh}^{\on{deg}}$, 
by concatenating \eqref{e:cond W ext} with the commutative diagram 
$$
\CD
\sF_{\cP}\mod(\QCoh(\LocSys_{\cP}))   @>{\BL^{\on{Whit}}_{G,P}}>>  \Whit(G,P)  \\  
@A{\on{ev}_P}AA   @AA{(\bi_P)^\dagger}A   \\
\on{Glue}(\cG)_{\on{spec}}   @>{\BL^{\on{Whit}^{\on{ext}}}_{G,G}}>>   \Whit^{\on{ext}}(G,G)
\endCD
$$

\sssec{}

Note that by combining \conjref{c:Whit ext ff}, \thmref{t:coLoc ext ff} and Quasi-Theorem \ref{t:glued functor},
we obtain:

\begin{corconj}  \hfill  \label{cc:when equivalence}

\smallskip

\noindent{\em(a)}
Property $\on{Wh}^{\on{ext}}$ determines the equivalence $\BL_G$ uniquely, and if the latter exists, 
it satisfies property $\on{He}^{\on{naive}}$.

\smallskip

\noindent{\em(b)}
The equivalence $\BL_G$ exists if and only if
the essential images of $$\IndCoh_{\on{Nilp}^{\on{glob}}_\cG}(\LocSys_\cG) \text{ and }\Dmod(\Bun_G)$$ in 
$\Whit^{\on{ext}}(G,G)$ under the functors
$$ \BL^{\on{Whit}^{\on{ext}}}_{G,G}\circ \on{Glue}(\on{CT}^{\on{enh}}_{\on{spec}}) \text{ and } \on{coeff}^{\on{ext}}_{G,G},$$
respectively, coincide.
\end{corconj}

In \secref{ss:proof of main} we will show (assuming Quasi-Theorem \ref{t:glued equiv} and Quasi-Theorem \ref{t:Whit ext and Eis}
below) that the condition of \corref{cc:when equivalence}(b) is satisfied for $G=GL_2$, thereby proving \conjref{c:GL} in this case. 

\sssec{}

Let us for a moment assume the validity of \conjref{c:GL}. We obtain
the following geometric characterization of the full subcategory $$\Dmod(\Bun_G)_{\on{temp}}\subset \Dmod(\Bun_G).$$

\begin{corconj}
An object $\CM\in  \Dmod(\Bun_G)$ belongs to the subcategory $\Dmod(\Bun_G)_{\on{temp}}$ if and only if the canonical map
$$\bj_\dagger(\on{coeff}_{G,G}(\CM))\to \on{coeff}^{\on{ext}}_{G,G}(\CM)$$
is an isomorphism.
\end{corconj} 

\ssec{Extended Whittaker coefficients and Eisenstein series compatibility}

Let $P'\subset G$ be another parabolic. In this subsection we will assume that \conjref{c:GL} holds for its
Levi quotient $M'$. 

\sssec{}

By concatenating the commutative diagrams \eqref{e:cond W ext} and \eqref{e:CT spec rat} we obtain the
following commutative diagram
\begin{equation} \label{e:Whit ext and Eis}
\CD
\on{Glue}(\cG)_{\on{spec}}   @>{\BL^{\on{Whit}^{\on{ext}}}_{G,G}}>>   \Whit^{\on{ext}}(G,G)  \\
@A{\on{Glue}(\on{CT}^{\on{enh}}_{\on{spec}})}AA    @AA{\on{coeff}^{\on{ext}}_{G,G}}A   \\
\IndCoh_{\on{Nilp}^{\on{glob}}_\cG}(\LocSys_\cG)  & & \Dmod(\Bun_G) \\
@A{\Eis^{\on{enh}}_{\cP',\on{spec}}}AA    @AA{\Eis_{P'}^{\on{enh}}}A   \\
\sF_{\cP'}\mod(\IndCoh_{\on{Nilp}_{\on{glob},\cP'}}(\LocSys_{\cP'}))  @>{\BL_{P'}}>>  \on{I}(G,P').
\endCD
\end{equation}

\sssec{}

We have:

\begin{qthm} \label{t:Whit ext and Eis}
The diagram \eqref{e:Whit ext and Eis} commutes unconditionally (i.e., without assuming the validity of \conjref{c:GL} for $G$).
\end{qthm}


\section{Compatibility with Kac-Moody localization and opers}   \label{s:opers}

We will now change gears and discuss a very different approach to the construction of objects of $\Dmod(\Bun_G)$. 
This construction has to do with localization of modules over the Kac-Moody algebra, first explored by \cite{BD2}. 

\medskip

As was explained in the introduction, we need this other construction for our approach to the proof of geometric
Langlands: some of the objects of $\Dmod(\Bun_G)$ obtained in this way will provide generators of this category, 
on which the functor $\on{coeff}^{\on{ext}}_{G,G}$ can be calculated explicitly. 

\medskip

The spectral counterpart of the Kac-Moody localization construction has to do with the scheme of $\cG$-opers,
also studied in this section. 

\ssec{The category of Kac-Moody modules}

In this subsection we will define what we mean by the category of Kac-Moody modules.

\medskip

Many of the objects discussed in this subsection do not, unfortunately, admit adequate references in the existing
literature. Hopefully, these gaps will be filled soon. 

\sssec{}

Let $\kappa$ be a \emph{level}, i.e., a $G$-invariant quadratic form on $\fg$.
We consider the corresponding affine Kac-Moody Lie algebra $\hg_\kappa$, which is a central
extension
$$0\to k\to \fL(\fg,\kappa)\to \fg\ppart\to 0,$$
and the category $\fL(\fg,\kappa)\mod$ as defined in \cite[Sect. 23.1]{FG2}.\footnote{In {\it loc.cit.}
it was referred to as the \emph{renormalized} category of Kac-Moody modules.}

\medskip

We consider the group-scheme $\fL^+(G):=G\qqart$, and our primary interest is the category 
$\on{KL}(G,\kappa)$ of $\fL^+(G)$-equivariant objects in $\fL(\fg,\kappa)\mod$. 

\begin{rem}
The eventually coconnective part of $\on{KL}(G,\kappa)$ (the subcategory of objects that are $>-\infty$ with respect to 
the natural t-structure) can be defined by the procedure of \cite[Sect. 20.8]{FG3}. The entire $\on{KL}(G,\kappa)$ is 
defined so that it is compactly generated by the Weyl modules.
\end{rem}  

\sssec{}

For this paper we will need the following generalization of the category $\on{KL}(G,\kappa)$: 

\medskip

For a finite set $I$, we consider the variety $X^I$. Over it there exists a group ind-scheme $\fL(G)_{X^I}$,
equipped with a connection; and a group subscheme $\fL^+(G)_{X^I}$. Let $\fL(\fg)_{X^I}$
denote the corresponding sheaf of topological Lie algebras over $(X^I)_\dr$. 

\medskip

We let $\fL(\fg,\kappa)_{X^I}$
be the central extension of $\fL(\fg)_{X^I}$ corresponding to $\kappa$, and we 
consider the corresponding category
$$\on{KL}(G,\kappa)_{X^I}:=\fL(\fg)_{X^I}\mod^{\fL^+(G)_{X^I}}.$$

\sssec{}

For a surjective map of finite sets $I_2\twoheadrightarrow I_2$ there is a naturally defined functor
\begin{equation} \label{e:restr finite set KL}
\on{KL}(G,\kappa)_{X^{I_1}}\to \on{KL}(G,\kappa)_{X^{I_2}}.
\end{equation}

The assignment $I\mapsto \on{KL}(G,\kappa)_{X^I}$ extends to a functor
$$(\on{fSet}_{\on{surj}})^{\on{op}}\to \on{DGCat}_{\on{cont}}$$
(see \secref{sss:finite sets} for the notation)
and we set
$$\on{KL}(G,\kappa)_{\on{\Ran(X)}}:=\underset{I\in (\on{fSet}_{\on{surj}})^{\on{op}}}
{\underset{\longrightarrow}{colim}}\, \on{KL}(G,\kappa)_{X^I}.$$

\sssec{}

Assume now that $\kappa$ is integral, which means by definition, that the central exetension of Lie algebras
$$0\to k\to \fL(\fg,\kappa)\to \fL(\fg)\to 0$$ 
comes from a central extension of group ind-schemes
$$1\to \BG_m\to \fL(G,\kappa)_{X^I}\to \fL(G)_{X^I}\to 1,$$
functorial in $I\in \on{fSet}_{\on{surj}}$.

\medskip

In this case, for every finite set $I$, there exists a canonically localization functor
$$\on{Loc}_{G,X^I}:\on{KL}(G,\kappa)_{X^I}\to \Dmod(\Bun_G).$$

These functors are compatible with the functors \eqref{e:restr finite set KL}, and hence we obtain
a functor
$$\on{Loc}_G:\on{KL}(G,\kappa)_{\Ran(X)}\to \Dmod(\Bun_G).$$

\medskip

We have the following assertion:

\begin{prop} \label{p:loc surj}
Let $\sU\subset \Bun_G$ be an open substack such that its intersection with every connected
component of $\Bun_G$ is quasi-compact. Then the composed functor
$$\on{KL}(G,\kappa)_{\Ran(X)}\overset{\on{Loc}_G}\longrightarrow \Dmod(\Bun_G)
\overset{\text{restriction}}\longrightarrow \Dmod(\sU)$$
is a localization, i.e., admits a fully faithful right adjoint.
\end{prop}

Note that from \propref{p:loc surj} we obtain:

\begin{cor} \label{c:loc surj}
Let $\sU\subset \Bun_G$ be an open substack such that its intersection with every connected
component of $\Bun_G$ is quasi-compact. Then the essential images of the functors
$$\on{KL}(G,\kappa)_{X^I}\overset{\on{Loc}_{G,X^I}} \longrightarrow \Dmod(\Bun_G)\to \Dmod(\sU),$$
as $I$ runs over $\on{fSet}$, generate $\Dmod(\sU)$.
\end{cor}

\begin{rem}
For any $\kappa$, we have a functor $\on{Loc}_G$ from $\on{KL}(G,\kappa)_{\Ran(X)}$ to the
corresponding category of $\kappa$-twisted D-modules on $\Bun_G$, and the analog of 
\propref{p:loc surj} holds. The proof amounts to a calculation of chiral homology of the chiral algebra
of differential operators on $G$, introduced in \cite{ArkhG}. 
\end{rem}

\sssec{}

From now on we will fix $\kappa$ to be the critical level, i.e., $-\frac{\kappa_{\on{Kil}}}{2}$, where $\kappa_{\on{Kil}}$
is the Killing form. 

\begin{rem}  \label{r:localization tempered}
Using an intrinsic characterization of the subcategory $$\Dmod(\Bun_G)_{\on{temp}}\subset \Dmod(\Bun_G)$$
described in \secref{sss:derived Satake temp}, and using the properties of the category $\on{KL}(G,\crit)$
with respect to the Hecke action (essentially, given by \cite[Theorem 8.22]{FG3}),
one shows that the essential image of the functor 
$$\on{Loc}_G:\on{KL}(G,\crit)_{\Ran(X)}\to \Dmod(\Bun_G)$$
lands inside $\Dmod(\Bun_G)_{\on{temp}}$. For the latter it is crucial that the value of $\kappa$ is critical
(as opposed to arbitrary integral).  
\end{rem}




\ssec{The spaces of local and global opers}

In this subsection we will introduce the scheme of $\cG$-opers. Quasi-coherent sheaves on the scheme of opers will be 
the spectral counterpart of Kac-Moody representations. 

\sssec{}

Let $I$ be again a finite non-empty set, and let $\lambda^I$ be a map from $I$ to the set $\Lambda^+$
of dominant weights of $G$, which are the same as dominant co-weights of $\cG$. 

\medskip

Local $\lambda^I$-opers for the group $\cG$ form a DG scheme mapping to $X^I$, denoted $\on{Op}(\cG)^{\on{loc}}_{\lambda^I}$,
defined as follows.

\medskip

For $S\in \affdgSch$, an $S$-point of $\on{Op}(\cG)^{\on{loc}}_{\lambda^I}$ is the data of 
$$(\ul{x},\CP_\cG,\CP_\cB,\alpha,\nabla),$$
where:

\medskip

\begin{itemize}

\item $\ul{x}$ is an $S$-point of $X^I$; we let $\cD_{\ul{x}}$ be the corresponding
parameterized family of formal discs over $S$.

\medskip

\item $\CP_\cG$ is a $\cG$-bundle over $\cD_{\ul{x}}$.

\medskip

\item $\alpha$ is a datum of reduction of $\CP_\cG$ to a $\cB$-bundle $\CP_\cB$, 
whose induced $\cT$-bundle $\CP_\cT$ is identified with $\rho(\omega_X)\otimes (-\lambda^I\cdot \ul{x})|_{\cD_{\ul{x}}}$,
where we regard $\lambda^I\cdot \ul{x}$ as a $\Lambda$-valued Cartier divisor on $S\times X$. 

\medskip

\item $\nabla$ is a datum of ``vertical" connection on $\CP_\cG$ along the fibers of the map $\cD_{\ul{x}}\to S$, 
i.e., a datum of lift of $\CP_\cG$ from a $\cG$-bundle
on $(\cD_{\ul{x}})_\dr\underset{S_\dr}\times S$.

\end{itemize}

Note that the discrepancy between $\alpha$ and $\nabla$
is given by a section of 
$$\nabla\,\on{mod}\, \cB\in (\cg/\cb)_{\CP_\cB}\otimes \omega_X|_{\cD_{\ul{x}}}.$$

We require that the following compatibility condition be satisfied:

\medskip

\begin{itemize}

\item $\nabla\,\on{mod}\, \cB$ belongs to the sub-bundle $(\cg^{-1}/\cb)_{\CP_\cB}\otimes \omega_X|_{\cD_{\ul{x}}}$;
where $(\cg^{-1}/\cb)\subset (\cg/\cb)$ is the $\cB$-subrepresentation spanned by negative simple roots. 

\medskip

\item For each vertex of the Dynkin diagram $i$, the resulting section of 
$$(\cg^{-1,\check\alpha_i}/\cb)_{\CP_\cB}\otimes \omega_X|_{\cD_{\ul{x}}}\simeq -\check\alpha_i(\CP_\cT)\otimes \omega_X|_{\cD_{\ul{x}}}$$
is the canonical map
$$\CO_{\cD_{\ul{x}}}\to \CO(\langle \lambda^I,\check\alpha_i\rangle \cdot \ul{x})|_{\cD_{\ul{x}}}\simeq 
-\check\alpha_i(\CP_\cT)\otimes \omega_X|_{\cD_{\ul{x}}}.$$

\end{itemize} 

\begin{rem}
As $S$ is a derived scheme, some care is needed to makes sense of the expression $\nabla\,\on{mod}\, \cB$
(see, e.g., \cite[Sect. 10.5]{AG} for how to do this). However, a posteriori, one can show that the DG scheme 
$\on{Op}(\cG)^{\on{loc}}_{\lambda^I}$ is classical, so we could restrict our attention to those $S\in \affdgSch$
that are themselves classical.
\end{rem}

\sssec{}

We define the DG scheme $\on{Op}(\cG)^{\on{glob}}_{\lambda^I}$ similarly, with the only difference 
that instead of the parameterized formal disc $\cD_{\ul{x}}$ we consider the entire scheme $S\times X$.

\medskip

By construction, we have the forgetful maps
$$
\CD
\LocSys_\cG   \\
@A{\sv_{\lambda^I}}AA  \\
\on{Op}(\cG)^{\on{glob}}_{\lambda^I}   @>{\su_{\lambda^I}}>>  \on{Op}(\cG)^{\on{loc}}_{\lambda^I}.
\endCD
$$ 

\begin{rem}
We note that, unlike, $\on{Op}(\cG)^{\on{loc}}_{\lambda^I}$, the DG scheme $\on{Op}(\cG)^{\on{glob}}_{\lambda^I}$
is typically \emph{not} classical.
\end{rem}

\sssec{}

Let $\LocSys_\cG^{\on{irred}}\subset \LocSys_\cG$ be the open substack 
corresponding to irreducible local systems. Let 
$$\on{Op}(\cG)^{\on{glob,irred}}_{\lambda^I}\subset \on{Op}(\cG)^{\on{glob}}_{\lambda^I}$$
be the preimage of $\LocSys_\cG^{\on{irred}}$ under the map $\sv_{\lambda^I}$. 

\medskip

We have:

\begin{lem}  \label{l:opers proper}
The map
$$\sv_{\lambda^I}:\on{Op}(\cG)^{\on{glob,irred}}_{\lambda^I}\to \LocSys_\cG^{\on{irred}}$$
is proper.
\end{lem}

Consider the functor
\begin{equation} \label{e:pull-back to opers}
(\sv_{\lambda^I})^!:\QCoh(\LocSys_\cG^{\on{irred}})\to \QCoh(\on{Op}(\cG)^{\on{glob,irred}}_{\lambda^I})
\end{equation}
right adjoint to 
\begin{equation} \label{e:pushforward from opers}
(\sv_{\lambda^I})_*:\QCoh(\on{Op}(\cG)^{\on{glob,irred}}_{\lambda^I})\to \QCoh(\LocSys_\cG^{\on{irred}}).
\end{equation}

\sssec{}

The next conjecture, along with \conjref{c:Whit ext ff}, is the second element in the proof of geometric Langlands
that still remains mysterious in the case of an arbitrary group $G$: 

\begin{conj} \label{c:Op gen}
Let $\CF\in \QCoh(\LocSys_\cG^{\on{irred}})$ be such that $(\sv_{\lambda^I})^!(\CF)=0$
for all finite sets $I$ and $\lambda^I:I\to \Lambda^+$. Then $\CF=0$.
\end{conj}

However, we have: 

\begin{thm} 
\conjref{c:Op gen} holds for $G=GL_n$.
\end{thm}

We also note that recent progress made by
D.~Kazhdan and T.~Schlank implies that \conjref{c:Op gen} holds also for $G=Sp(2n)$.

\sssec{}

We can reformulate \conjref{c:Op gen} as follows:

\begin{corconj} \label{cc:Op gen}
The union of the essential images of the functors 
$$(\sv_{\lambda^I})_*:\QCoh(\on{Op}(\cG)^{\on{glob,irred}}_{\lambda^I})\to \QCoh(\LocSys_\cG^{\on{irred}})$$
over all finite sets $I$ and $\lambda^I:I\to \Lambda^+$, generates $\QCoh(\LocSys_\cG^{\on{irred}})$.
\end{corconj}

\ssec{Compatibility between opers and Kac-Moody localization}

In this subsection we will match the local category on the geometric side, i.e., $\on{KL}(G,\crit)_{X^I}$, with the
local category on the spectral side, i.e., $\QCoh(\on{Op}(\cG)^{\on{loc}}_{\lambda^I})$.  

\sssec{}

The following is an extension of \cite[Proposition 3.5]{FG1}:

\begin{prop} \label{p:from Op to KM}  \hfill

\smallskip

\noindent{\em(a)} For a finite set $I$ and $\lambda^I:I\to \Lambda^+$ there exists a canonically defined functor
$$\BL_G^{\on{Op}_{\lambda^I}}:\QCoh(\on{Op}(\cG)^{\on{loc}}_{\lambda^I})\to \on{KL}(G,\crit)_{X^I}.$$

\smallskip

\noindent{\em(b)} For a fixed finite set $I$, the union of essential images of the functors $\BL_G^{\on{Op}_{\lambda^I}}$
over $\lambda^I:I\to \Lambda^+$ generates $\on{KL}(G,\crit)_{X^I}$.

\end{prop}

\sssec{}

The next theorem is a moving-point version of \cite[Theorem 5.2.9]{BD2}:

\begin{thm}  \label{t:Hitchin}  \hfill

\smallskip

\noindent{\em(a)} The composed functor
$$\QCoh(\on{Op}(\cG)^{\on{loc}}_{\lambda^I})\overset{\BL_G^{\on{Op}_{\lambda^I}}}\longrightarrow
\on{KL}(G,\crit)_{X^I} \overset{\on{Loc}_{G,X^I}}\longrightarrow \Dmod(\Bun_G)$$
canonically factors as
$$\QCoh(\on{Op}(\cG)^{\on{loc}}_{\lambda^I})\overset{(\su_{\lambda^I})^*}\longrightarrow
\QCoh(\on{Op}(\cG)^{\on{glob}}_{\lambda^I})\overset{\on{q-Hitch}_{\lambda^I}}\longrightarrow
\Dmod(\Bun_G).$$

\smallskip

\noindent{\em(b)} The resulting functor \footnote{The notation ``$\on{q-Hitch}$" stands for ``quantized Hitchin map".}
$$\on{q-Hitch}_{\lambda^I}:\QCoh(\on{Op}(\cG)^{\on{glob}}_{\lambda^I})\to \Dmod(\Bun_G)$$
respects the action of $\Rep(\cG)_{\Ran(X)}$, where the latter acts on 
$\QCoh(\on{Op}(\cG)^{\on{glob}}_{\lambda^I})$ via the composition $(\sv_{\lambda^I})^*\circ \on{Loc}_{\cG,\on{spec}}$
and on $\Dmod(\Bun_G)$ via the functor $\on{Sat}(G)^{\on{naive}}_{\Ran(X)}$. 

\end{thm}

\sssec{} \label{sss:cond O prel}

We are finally ready to state Property $\on{Km}^{\on{prel}}$ of the geometric Langlands functor $\BL_G$ in \conjref{c:GL}
(``Km" stands for Kac-Moody): 

\medskip

\noindent{\bf Property $\mathbf{Km}^{\mathbf{prel}}$:} {\it We shall say that the functor $\BL_G$ satisfies
\emph{Property $\on{Km}^{\on{prel}}$} if for every finite set $I$ and $\lambda^I:I\to \Lambda^+$, the following diagram
is commutative:
\begin{equation} \label{e:cond O prel}
\CD
\IndCoh_{\on{Nilp}^{\on{glob}}_\cG}(\LocSys_\cG)   @>{\BL_G}>> \Dmod(\Bun_G)  \\
@A{\Xi_\cG}AA     & &   \\
\QCoh(\LocSys_\cG)  & &  @AA{\on{q-Hitch}_{\lambda^I}}A  \\
@A{(\sv_{\lambda^I})_*}AA    & &  \\
\QCoh(\on{Op}(\cG)^{\on{glob}}_{\lambda^I})  @>{\on{Id}}>> \QCoh(\on{Op}(\cG)^{\on{glob}}_{\lambda^I}).
\endCD
\end{equation}}

\ssec{The oper vs. Whittaker compatibility}

The reason the localization procedure is useful is that one can explicitly control the Whittaker coefficients
of D-modules obtained in this way. How this is done will be explained in the present subsection. 

\sssec{}

For a finite set $I$ and a map $\Lambda^I:I\to \Lambda^+$ let us concatenate the diagrams
\eqref{e:cond O prel} and \eqref{e:cond W}. 

\medskip

Using the fact that $\Psi_\cG\circ \Xi_\cG\simeq \on{Id}$, 
we obtain a commutative diagram
\begin{equation} \label{e:Op and Whit}
\CD
\Rep(\cG)_{\Ran(X)}\underset{\Dmod(\Ran(X))}\otimes \Vect @>{\BL^{\on{Whit}}_{G}}>>  \on{Whit}(G)  \\
@A{\on{co-Loc}^{\on{unital}}_{\cG,\on{spec}}}AA    @AA{\on{coeff}_{G}}A \\
\QCoh(\LocSys_\cG)  & &  \Dmod(\Bun_G)  \\
@A{(\sv_{\lambda^I})_*}AA     @AA{\on{q-Hitch}_{\lambda^I}}A \\
\QCoh(\on{Op}(\cG)^{\on{glob}}_{\lambda^I})  @>{\on{Id}}>> \QCoh(\on{Op}(\cG)^{\on{glob}}_{\lambda^I}).
\endCD
\end{equation}

We claim:

\begin{thm} \label{t:Op and Whit}
The diagram \eqref{e:Op and Whit} commutes unconditionally, i.e., without assuming the validity of
\conjref{c:GL}.
\end{thm} 

\begin{rem}
The proof of \thmref{t:Op and Whit} amounts to a computation of chiral homology of the chiral algebra
responsible for the scheme of opers, and the Feigin-Frenkel isomorphism that identifies it with
the Whittaker BRST reduction of the chiral algebra corresponding to $\fL(\fg,\crit)$.
\end{rem}

\sssec{}

As a corollary of \thmref{t:Op and Whit} we obtain:

\begin{cor} \label{c:Op and Whit}
The following diagram is commutative
$$
\CD
\QCoh(\LocSys_\cG)   @>{\BL^{\on{Whit}}_{G,G}}>>  \Whit(G,G)   \\
@A{\Psi_\cG}AA   @AA{\on{coeff}_{G,G}}A   \\
\IndCoh_{\on{Nilp}^{\on{glob}}_\cG}(\LocSys_\cG) & & \Dmod(\Bun_G)    \\
@A{\Xi_\cG\circ (\sv_{\lambda^I})_*}AA    @AA{\on{q-Hitch}_{\lambda^I}}A   \\
\QCoh(\on{Op}(\cG)^{\on{glob}}_{\lambda^I})  @>{\on{Id}}>> \QCoh(\on{Op}(\cG)^{\on{glob}}_{\lambda^I}).
\endCD
$$
\end{cor}

\ssec{Full compatibility with opers}  \label{ss:opers Ran}

The material in this subsection will not be used elsewhere in the paper. We will discuss a stronger version
of Property $\on{Km}^{\on{prel}}$ of the geometric Langlands functor $\BL_G$ that we call Property $\on{Km}$.
As adequate references are not available, we will only indicate the formal structure
of the theory once the appropriate definitions are given.

\sssec{}

For every finite set $I$ one can introduce prestacks $\on{Op}(\cG)^{\on{loc}}_{X^I}$ and 
$\on{Op}(\cG)^{\on{glob}}_{X^I}$, by considering ``opers with singularities but without monodromy", 
instead of $\lambda^I$-opers for a specified $\lambda^I:I\to \Lambda^+$. For $I=\{1\}$ the local
version is defined in \cite[Sect. 2.2]{FG1}. 

\medskip

We have a diagram 
$$
\CD
\LocSys_\cG  \\
@A{\sv_{X^I}}AA  \\
\on{Op}(\cG)^{\on{glob}}_{X^I}  @>{\su_{X^I}}>>  \on{Op}(\cG)^{\on{loc}}_{X^I}.
\endCD
$$

\sssec{}

We have:

\begin{conj}  \label{c:FF}
There exists a canonically defined equivalence
$$\BL_G^{\on{Op}_{X^I}}:\QCoh(\on{Op}(\cG)^{\on{loc}}_{X^I})\to \on{KL}(G,\crit)_{X^I},$$
extending the functors $\BL_G^{\on{Op}_{\lambda^I}}$ of \propref{p:from Op to KM}(a).
\end{conj}

\sssec{}

Passing to the limit over $I\in  (\on{fSet}_{\on{surj}})^{\on{op}}$, we obtain a diagram
\begin{equation} \label{e:arb opers}
\CD
\LocSys_\cG  \\
@A{\sv_{\Ran(X)}}AA  \\
\on{Op}(\cG)^{\on{glob}}_{\Ran(X)}  @>{\su_{\Ran(X)}}>>  \on{Op}(\cG)^{\on{loc}}_{\Ran(X)}
\endCD
\end{equation}
and an equivalence
$$\BL_G^{\on{Op}_{\Ran(X)}}:\QCoh(\on{Op}(\cG)^{\on{loc}}_{\Ran(X)})\to \on{KL}(G,\crit)_{\Ran(X)}.$$

\medskip

Push-pull along \eqref{e:arb opers} defines a functor
$$\Poinc_{\cG,\on{spec}}:\QCoh(\on{Op}(\cG)^{\on{loc}}_{\Ran(X)})\to \QCoh(\LocSys_\cG).$$

\sssec{}

The full Property $\on{Km}$ of the geometric Langlands functor $\BL_G$ reads:

\medskip

\noindent{\bf Property $\mathbf{Km}$:} {\it We shall say that the functor $\BL_G$ satisfies
\emph{Property $\on{Km}$} if the diagram is commutative:
$$
\CD
\IndCoh_{\on{Nilp}^{\on{glob}}_\cG}(\LocSys_\cG)   @>{\BL_G}>> \Dmod(\Bun_G)  \\
@A{\Xi_\cG}AA     & &   \\
\QCoh(\LocSys_\cG)  & &  @AA{\on{Loc}_G}A  \\
@A{\Poinc_{\cG,\on{spec}}}AA    & &  \\
\QCoh(\on{Op}(\cG)^{\on{loc}}_{\Ran(X)})  @>{\BL_G^{\on{Op}_{\Ran(X)}}}>> \on{KL}(G,\crit)_{\Ran(X)}.
\endCD
$$}

\sssec{}

Finally, we propose the following two closely related conjectures: 

\begin{conj}  \label{c:Op ff}  
The fibers of the map $\sv_{\Ran(X)}$ are $\CO$-contractible, i.e., the functor
$$\sv_{\Ran(X)}^*:\QCoh(\LocSys_\cG)\to \QCoh(\on{Op}(\cG)^{\on{glob}}_{\Ran(X)})$$
is fully faithful. 
\end{conj}

\begin{rem}
\conjref{c:Op ff} can be reformulated as saying that the functor
$$(\sv_{\Ran(X)})_*:\QCoh(\on{Op}(\cG)^{\on{glob}}_{\Ran(X)})\to \QCoh(\LocSys_\cG)$$
is a \emph{co-localization}, i.e., it identifies the homotopy category
of $\QCoh(\LocSys_\cG)$ with a Verdier quotient of $\QCoh(\on{Op}(\cG)^{\on{glob}}_{\Ran(X)})$. 
Note that this equivalent to $(\sv_{\Ran(X)})_*$ being a \emph{localization}, i.e., that its 
(not necessarily continuous) \emph{right} adjoint is fully faithful.
\end{rem}

\begin{rem}
Note that \conjref{c:Op ff} is a strengthening of \conjref{c:Op gen}, and it should be within reach
for $G=GL_n$. The recent work of D.~Kazhdan and T.~Schlank indicates that it also holds for $G=Sp(2n)$.
\end{rem}

\begin{conj} \label{c:Op quot}
The functor $\Poinc_{\cG,\on{spec}}$ is a \emph{localization}, i.e., it identifies the homotopy category
of $\QCoh(\LocSys_\cG)$ with a Verdier quotient of $\QCoh(\on{Op}(\cG)^{\on{loc}}_{\Ran(X)})$. 
\end{conj} 

\section{The proof modulo the conjectures}  \label{s:proof}

In this section we will assemble the ingredients developed in the previous sections to prove
\conjref{c:GL}, assuming Conjectures \ref{c:Whit ext ff} and \ref{c:Op gen} and all the quasi-theorems. 

\ssec{Proof of the vanishing theorem}   \label{ss:proof of vanishing}

One of the steps in the proof of \conjref{c:GL} is \thmref{t:generalized vanishing}. 
The proofs of both \thmref{t:generalized vanishing} of \conjref{c:GL} rely on the following description of generators
of the category $\Dmod(\Bun_G)$:

\begin{thm} \label{t:generate BunG}
The union of the essential images of the functors
$$\on{Loc}_{G,X^I}:\on{KL}(G,\kappa)_{X^I}\to \Dmod(\Bun_G)$$
and
$$\Eis_P:\Dmod(\Bun_M)\to \Dmod(\Bun_G)$$
for \emph{proper} parabolics $P\subset G$, generates $\Dmod(\Bun_G)$.
\end{thm}

\begin{proof}[Sketch of proof]

\medskip

Let $\CM\in \Dmod(\Bun_G)$ be right-orthogonal to both the essential image of $\Eis_P$ for 
all proper parabolics $P\subset G$ and $\on{Loc}_{G,X^I}$. We need to show that $\CM=0$. 

\medskip

Reduction theory (see \cite[Proposition 1.4.6]{DrGa3}) implies that there exists an open substack 
$$\sU\overset{f}\hookrightarrow \Bun_G$$
such that its intersection with every connected component of $\Bun_G$ is quasi-compact,
and which has the following property:

\medskip

For every object $\CM'\in  \Dmod(\Bun_G)_{\on{cusp}}$, the canonical arrow
\begin{equation}  \label{e:ext from big}
\CM'\to f_\bullet\circ f^\bullet(\CM')
\end{equation}
is an isomorphism. 

\medskip

On the one hand, the assumption that $\CM$ is right-orthogonal to the essential image of $\Eis_P$ for 
all proper parabolics $P\subset G$, implies that $\CM\in \Dmod(\Bun_G)_{\on{cusp}}$. Now
the fact that $\CM=0$ follows from \eqref{e:ext from big} and \corref{c:loc surj}.

\end{proof}

\sssec{}

We are now ready to prove \thmref{t:generalized vanishing}:

\begin{proof}
Let $\CF\in \Rep(\cG)_{\Ran(X)}$ be an object such that $\on{Loc}_{\cG,\on{spec}}(\CF)=0$.
We need to show that the action of $\CF$ on $\Dmod(\Bun_G)$ is also zero.
For that it is sufficient to show that $\CF$ acts by zero on a subcategory of 
$\Dmod(\Bun_G)$ that generates it. 

\medskip

By \thmref{t:Hitchin}(b), the action of $\CF$ on objects in the essential image
of the functor $\on{q-Hitch}_{\lambda^I}$ (for any finite set $I$ and $\lambda^I:I\to \Lambda^+$)
is zero. Combined with \propref{p:from Op to KM},
this implies that $\CF$ acts by zero on the essential image of 
$\on{Loc}_{G,X^I}$ for any finite set $I$.

\medskip

By \thmref{t:generate BunG}, it remains to show that $\CF$ acts by zero 
on the essential image of $\Eis_P$ for all proper parabolics $P\subset G$.

\medskip

This follows from the next assertion, which is itself a particular case of Quasi-Theorem \ref{t:principal series},
but can be proved in a more elementary way by generalizing  the argument of \cite[Theorem 1.11]{BG}:

\begin{prop}  Assume that \thmref{t:generalized vanishing} holds for the Levi quotient
$M$; in particular we have an action of $\QCoh(\LocSys_\cM)$ on $\Dmod(\Bun_M)$. 

\smallskip

\noindent{\em(a)} The functor
$$\Eis_P:\Dmod(\Bun_M)\to \Dmod(\Bun_G)$$
canonically factors as
\begin{multline*}
\Dmod(\Bun_M)\overset{\sfq_{\cP,\on{spec}}^*\otimes \on{Id}}\longrightarrow
\QCoh(\LocSys_\cP)\underset{ \QCoh(\LocSys_\cM)}\otimes 
\Dmod(\Bun_M) \overset{\Eis_P^{\on{int}}}\longrightarrow \\
\to \Dmod(\Bun_G).
\end{multline*}

\smallskip

\noindent{\em(b)} For $\CF\in \Rep(\cG)_{\Ran(X)}$ and 
$\CM\in \QCoh(\LocSys_\cP)\underset{ \QCoh(\LocSys_\cM)}\otimes 
\Dmod(\Bun_M)$, we have a canonical isomorphism
$$\CF\star \Eis^{\on{int}}_P(\CM)\simeq
\Eis_P^{\on{int}}\left(\sfp_{\cP,\on{spec}}^*(\on{Loc}_{\cG,\on{spec}}(\CF))\otimes \CM\right),$$
where $-\star-$ denotes the monoidal action of $\Rep(\cG)_{\Ran(X)}$ on $\Dmod(\Bun_G)$.
\end{prop}

\end{proof}

\begin{rem}

The functor $\Eis_P^{\on{int}}$ can also be interpreted as a composition of 
$\Eis_P^{\on{enh}}$ with a canonically defined functor
$$\QCoh(\LocSys_\cP)\underset{ \QCoh(\LocSys_\cM)}\otimes 
\IndCoh_{\on{Nilp}^{\on{glob}}_\cM}(\LocSys_\cM) \to \on{I}(G,P),$$
which in terms of the equivalence $\BL_P$ corresponds to 
\begin{multline*}
\QCoh(\LocSys_\cP)\underset{ \QCoh(\LocSys_\cM)}\otimes 
\IndCoh_{\on{Nilp}^{\on{glob}}_\cM}(\LocSys_\cM)\overset{\sfq^*_{\sP,\on{spec}}} \longrightarrow \\
\to \IndCoh_{\on{Nilp}^{\on{glob}}_\cP}(\LocSys_\cP)  \overset{\ind_{\sF_{\cP}}}\longrightarrow
\ind_{\sF_{\cP}}\mod(\IndCoh_{\on{Nilp}^{\on{glob}}_\cP}(\LocSys_\cP)),
\end{multline*}
up to an auto-equivalence of $\Dmod(\Bun_M)$.
\end{rem}




\ssec{Construction of the functor} 

From now on, we shall assume the validity of Conjectures \ref{c:Whit ext ff} and \ref{c:Op gen}
(which are theorems for $GL_n$) and all the Quasi-Theorems, and deduce \conjref{c:GL}. 

\medskip

By induction on the rank, we can assume the validity of \conjref{c:GL} for proper Levi subgroups of $G$.

\sssec{}  \label{sss:1st containment}

By \conjref{c:Whit ext ff}, the existence of the functor $\BL_G$ amounts to showing that
the essential image of the functor
$$
\CD
\on{Glue}(\cG)_{\on{spec}}   @>{\BL^{\on{Whit}^{\on{ext}}}_{G,G}}>>   \Whit^{\on{ext}}(G,G)  \\
@A{\on{Glue}(\on{CT}^{\on{enh}}_{\on{spec}})}AA     \\
\IndCoh_{\on{Nilp}^{\on{glob}}_\cG}(\LocSys_\cG)  
\endCD
$$
is contained in the essential image of the functor
$$
\CD
\Whit^{\on{ext}}(G,G)  \\
@AA{\on{coeff}^{\on{ext}}_{G,G}}A    \\
\Dmod(\Bun_G).
\endCD
$$

The latter is enough to check on the generators of $\IndCoh_{\on{Nilp}^{\on{glob}}_\cG}(\LocSys_\cG)$.

\medskip

By \conjref{cc:Op gen} and \propref{p:Eis generation}(b),
the category $\IndCoh_{\on{Nilp}^{\on{glob}}_\cG}(\LocSys_\cG)$ is generated by the union of the
essential images of the following functors:

\medskip

\begin{itemize}

\item $\Eis_{\cP,\on{spec}}:\IndCoh_{\on{Nilp}^{\on{glob}}_\cM}(\LocSys_\cM)\to \IndCoh_{\on{Nilp}^{\on{glob}}_\cG}(\LocSys_\cG)$
for all \emph{proper} parabolics $P\subset G$.

\medskip

\item $\Xi_\cG\circ \jmath_* \circ (\sv_{\lambda^I})_*:\QCoh(\on{Op}(\cG)^{\on{glob,irred}}_{\lambda^I})\to 
\IndCoh_{\on{Nilp}^{\on{glob}}_\cG}(\LocSys_\cG)$
for all finite sets $I$ and $\lambda^I:I\to \Lambda^+$, where $\jmath$ denotes the open embedding
$\LocSys^{\on{irred}}_\cG\hookrightarrow \LocSys_\cG$. 

\end{itemize} 

\sssec{}

The containment of \secref{sss:1st containment} for the functors 
$$\Eis_{\cP,\on{spec}}:\IndCoh_{\on{Nilp}^{\on{glob}}_\cM}(\LocSys_\cM)\to \IndCoh_{\on{Nilp}^{\on{glob}}_\cG}(\LocSys_\cG)$$
is equivalent to that for the functors 
$$\Eis^{\on{enh}}_{\cP,\on{spec}}:\sF_{\cP}\mod(\IndCoh_{\on{Nilp}^{\on{glob}}_\cP}(\LocSys_{\cP}))  
\to \IndCoh_{\on{Nilp}^{\on{glob}}_\cG}(\LocSys_\cG).$$

\medskip

For the latter, it follows from Quasi-Theorem \ref{t:Whit ext and Eis}.

\sssec{}

It remains to show the containment of \secref{sss:1st containment} for the functor 
$$\Xi_\cG\circ \jmath_* \circ (\sv_{\lambda^I})_*:\QCoh(\on{Op}(\cG)^{\on{glob,irred}}_{\lambda^I})\to 
\IndCoh_{\on{Nilp}^{\on{glob}}_\cG}(\LocSys_\cG)$$
for a fixed finite set $I$ and $\lambda^I:I\to \Lambda^+$.
 
\medskip 

Let $\CF$ be an object of 
$\QCoh(\on{Op}(\cG)^{\on{glob,irred}}_{\lambda^I})$. We claim that 
\begin{equation} \label{e:1st containment on loc}
\BL^{\on{Whit}^{\on{ext}}}_{G,G}\circ \on{Glue}(\on{CT}^{\on{enh}}_{\on{spec}})\circ \Xi_\cG\circ \jmath_* \circ (\sv_{\lambda^I})_*(\CF)\simeq 
\on{coeff}^{\on{ext}}_{G,G}\circ \on{q-Hitch}_{\lambda^I}(\CF).
\end{equation}

Clearly, \eqref{e:1st containment on loc} would imply the required assertion.

\sssec{}

First, we note that the isomorphsim
$$\bj^\bullet\left(\BL^{\on{Whit}^{\on{ext}}}_{G,G}\circ 
\on{Glue}(\on{CT}^{\on{enh}}_{\on{spec}})\circ \Xi_\cG \circ \jmath_* \circ (\sv_{\lambda^I})_*(\CF)\right)
\simeq \bj^\bullet\left(\on{coeff}^{\on{ext}}_{G,G}\circ \on{q-Hitch}_{\lambda^I}(\CF)\right)$$
follows from \corref{c:Op and Whit}.

\medskip

Second, we note that, by definition, 
$$\BL^{\on{Whit}^{\on{ext}}}_{G,G}\circ \on{Glue}(\on{CT}^{\on{enh}}_{\on{spec}})\circ \Xi_\cG\circ (\sv_{\lambda^I})_*(\CF)\simeq
\bj_\dagger\circ \bj^\bullet\left(\BL^{\on{Whit}^{\on{ext}}}_{G,G}\circ 
\on{Glue}(\on{CT}^{\on{enh}}_{\on{spec}})\circ \Xi_\cG \circ (\sv_{\lambda^I})_*(\CF)\right).$$

Hence, it is enough to show that the canonical map
\begin{equation} \label{e:irred locus}
\bj_\dagger\circ \bj^\bullet\left(\on{coeff}^{\on{ext}}_{G,G}\circ \on{q-Hitch}_{\lambda^I}(\CF)\right)\to
\on{coeff}^{\on{ext}}_{G,G}\circ \on{q-Hitch}_{\lambda^I}(\CF)
\end{equation}
is an isomorphism. 

\sssec{}  \label{sss:irred 1}

We consider the category $\Whit^{\on{ext}}(G,G)$ as acted on by the monoidal category $\Rep(\cG)_{\Ran(X)}$,
and recall that the functor $\on{coeff}^{\on{ext}}_{G,G}$ respects this action.

\medskip

Consider the cone of the map \eqref{e:irred locus}; denote it by $\CM'$.
On the one hand, by \thmref{t:Hitchin}(b), the action of $\Rep(\cG)_{\Ran(X)}$ on $\CM'$ factors through
$$\Rep(\cG)_{\Ran(X)}\overset{\on{Loc}_{\cG,\on{spec}}}\longrightarrow \QCoh(\LocSys_\cG)\overset{\jmath^*}\longrightarrow
\QCoh(\LocSys^{\on{irred}}_\cG).$$

\medskip

On the other hand, we claim that for any $\CM\in \Dmod(\Bun_G)$, the action of $\Rep(\cG)_{\Ran(X)}$
on
$$\CM':=\on{Cone}\left(\bj_\dagger\circ \bj^\bullet(\on{coeff}^{\on{ext}}_{G,G}(\CM))\to \on{coeff}^{\on{ext}}_{G,G}(\CM)\right)$$
factors through 
$$\Rep(\cG)_{\Ran(X)}\overset{\on{Loc}_{\cG,\on{spec}}}\longrightarrow \QCoh(\LocSys_\cG)\to
\QCoh(\LocSys_\cG)_{\LocSys^{\on{red}}_\cG},$$
where $\QCoh(\LocSys_\cG)_{\LocSys^{\on{red}}_\cG}$ is quotient of $\QCoh(\LocSys_\cG)$ by the monoidal
ideal given by the embedding $\QCoh(\LocSys^{\on{irred}}_\cG)\overset{\jmath_*}\hookrightarrow \QCoh(\LocSys_\cG)$. 

\medskip

This would imply that $\CM'=0$.

\sssec{}  \label{sss:irred 2}

The object $\CM'$ admits a filtration with subquotients of the form
$$(\bi_P)_\dagger\circ (\bi_P)^\dagger(\on{coeff}^{\on{ext}}_{G,G}(\CM))\simeq 
(\bi_P)_\dagger\circ  \on{coeff}_{G,P}(\CM)$$
for the proper parabolics $P$ of $G$.

\medskip

Hence, it suffices to check that for any proper parabolic $P\subset G$, the action of $\Rep(\cG)_{\Ran(X)}$
on 
$$\on{coeff}_{G,P}(\CM)\in \Whit(G,P)$$
factors through 
$$\Rep(\cG)_{\Ran(X)}\overset{\on{Loc}_{\cG,\on{spec}}}\longrightarrow \QCoh(\LocSys_\cG)
\overset{\sfp^*_{\cP,\on{spec}}}\longrightarrow \QCoh(\LocSys_\cP).$$

By \eqref{e:deg Whit via CT}, it suffices to establish the said factorization for the object
$$\on{CT}_P^{\on{enh}}(\CM)\in \on{I}(G,P).$$

Now the required assertion follows from Quasi-Theorem \ref{t:principal series}(b). 

\ssec{Proof of the equivalence}    \label{ss:proof of main}
 
We will now show that the functor $\BL_G$, whose existence was proved above, is an equivalence. 

\sssec{}

First, we claim that $\BL_G$ is fully faithful. This follows from \thmref{t:coLoc ext ff} and \conjref{c:Whit ext ff}.

\medskip

To prove that $\BL_G$ is essentially surjective,  it is enough to show that the generators of $\Dmod(\Bun_G)$
belong to the essential image of $\BL_G$.

\medskip

By \thmref{t:generate BunG} and \propref{p:from Op to KM}(b), the category $\Dmod(\Bun_G)$ 
is generated by the union of the essential images of the following functors:

\medskip

\begin{itemize}

\item $\Eis_P:\Dmod(\Bun_M)\to \Dmod(\Bun_G)$ for all proper parabolics $P\subset M$.

\medskip

\item $\on{q-Hitch}_{\lambda^I}:\QCoh(\on{Op}(\cG)^{\on{glob}}_{\lambda^I})\to \Dmod(\Bun_G)$
for all finite sets $I$ and $\lambda^I:I\to \Lambda^+$.

\end{itemize}

\sssec{}

First, we claim that the essential image of $\Eis_P$ is contained in the essential image of $\BL_G$.
This is equivalent for the corresponding assertion for the functor $\Eis^{\on{enh}}_P$.

\medskip

For the latter it suffices to show that  the essential image of the
functor
$$
\CD
\Whit^{\on{ext}}(G,G)  \\
@AA{\on{coeff}^{\on{ext}}_{G,G}}A    \\
\Dmod(\Bun_G) \\
@A{\Eis^{\on{enh}}_P}AA  \\
\on{I}(G,P)
\endCD
$$
is contained in the essential image of the functor
$$
\CD
\on{Glue}(\cG)_{\on{spec}}   @>{\BL^{\on{Whit}^{\on{ext}}}_{G,G}}>>   \Whit^{\on{ext}}(G,G)  \\
@A{\on{Glue}(\on{CT}^{\on{enh}}_{\on{spec}})}AA     \\
\IndCoh_{\on{Nilp}^{\on{glob}}_\cG}(\LocSys_\cG). 
\endCD
$$

However, this follows from from Quasi-Theorem \ref{t:Whit ext and Eis}.  

\sssec{A digression}

Let $\Dmod(\Bun_G)_{\Eis}$ denote the full subcategory of $\Dmod(\Bun_G)$ generated by
the essential images of the functors $\Eis_P$, for all proper parabolics $P\subset G$. By
the above, the subcategory $\Dmod(\Bun_G)_{\Eis}$ is contained in the essential image 
of the functor $\BL_G$. 

\medskip

This, every $\CM\in \Dmod(\Bun_G)$ canonically fits in an exact triangle
$$\CM_{\Eis}\to \CM\to \CM_{\on{cusp}},$$
where $\CM_{\Eis}\in \Dmod(\Bun_G)_{\Eis}$ and $\CM_{\on{cusp}}\in \Dmod(\Bun_G)_{\on{cusp}}$.

\sssec{}

It remains to show that for a finite set $I$, $\lambda^I:I\to \Lambda^+$, and 
$\CF\in \QCoh(\on{Op}(\cG)^{\on{glob}}_{\lambda^I})$, the object
$\on{q-Hitch}_{\lambda^I}(\CF)$
belongs to the essential image of $\BL_G$. 

\medskip

By the above, it is sufficient to show that the object
$$\left(\on{q-Hitch}_{\lambda^I}(\CF)\right)_{\on{cusp}}$$
belongs to the essential image of $\BL_G$. 

\medskip

We will construct an isomorphism
$$\left(\BL_G(\Xi_\cG\circ (\sv_{\lambda^I})_*(\CF))\right)_{\on{cusp}}\simeq 
\left(\on{q-Hitch}_{\lambda^I}(\CF)\right)_{\on{cusp}}.$$

\begin{rem}
It will follow \emph{a posteriori} that we actually have an isomorphism
$$\BL_G(\Xi_\cG\circ (\sv_{\lambda^I})_*(\CF))\simeq \on{q-Hitch}_{\lambda^I}(\CF),$$
which amounts to Property $\on{Km}^{\on{prel}}$ in \conjref{c:GL}.
\end{rem}

\sssec{}

Let us construct a map
\begin{equation} \label{e:map for loc}
\BL_G(\Xi_\cG\circ (\sv_{\lambda^I})_*(\CF))\to 
\on{q-Hitch}_{\lambda^I}(\CF).
\end{equation}

By \conjref{c:Whit ext ff}, this amounts to a map
\begin{equation} \label{e:map for loc Whit}
\on{coeff}^{\on{ext}}_{G,G}\circ \BL_G(\Xi_\cG\circ (\sv_{\lambda^I})_*(\CF))\to
\on{coeff}^{\on{ext}}_{G,G}(\on{q-Hitch}_{\lambda^I}(\CF)).
\end{equation}

Note that by \corref{c:Op and Whit}, we have an isomorphism
$$\on{coeff}_{G,G}\circ \BL_G(\Xi_\cG\circ (\sv_{\lambda^I})_*(\CF))\overset{\sim}\to
\on{coeff}_{G,G}(\on{q-Hitch}_{\lambda^I}(\CF)).$$

Furthermore, by construction, the map
$$\bj_\dagger\circ \bj^\bullet\left(\on{coeff}^{\on{ext}}_{G,G}\circ \BL_G(\Xi_\cG\circ (\sv_{\lambda^I})_*(\CF))\right)\to
\on{coeff}^{\on{ext}}_{G,G}\circ \BL_G(\Xi_\cG\circ (\sv_{\lambda^I})_*(\CF))$$
is an isomorphism.

\medskip

This gives rise to the desired map in \eqref{e:map for loc Whit}.

\sssec{}

Let $\CM$ denote the cone of the map \eqref{e:map for loc}. By construction, 
$$\bj^\bullet\circ \on{coeff}^{\on{ext}}_{G,G}(\CM)=\on{coeff}_{G,G}(\CM)=0.$$

We wish to show that $\CM_{\on{cusp}}=0$, which is equivalent to showing that
the canonical map
$$\CM\to \CM_{\on{cusp}}$$
vanishes.

\medskip

By \conjref{c:Whit ext ff}, it suffices to show that the map
$$\on{coeff}^{\on{ext}}_{G,G}(\CM)\to \on{coeff}^{\on{ext}}_{G,G}(\CM_{\on{cusp}})$$
vanishes. 

\medskip

However, since $\CM_{\on{cusp}}\in \Dmod(\Bun_G)_{\on{cusp}}$, the canonical map
$$\on{coeff}^{\on{ext}}_{G,G}(\CM_{\on{cusp}})\to
\bj_\bullet\circ \bj^\bullet(\on{coeff}^{\on{ext}}_{G,G}(\CM_{\on{cusp}}))$$
is an isomorphism (see \corref{c:cusp}). 

\medskip

Hence, the required vanishing holds by the $(\bj^\bullet,\bj_\bullet)$-adjunction.

\ssec{Proof of the properties and further remarks}  

\sssec{}

Thus, the equivalence $\BL_G$, satisfying Property $\on{Wh}^{\on{ext}}$, claimed in \conjref{c:GL}(a) has been constructed. 
Let us now prove the properties claimed in \conjref{c:GL}(b). 

\medskip

Property $\on{He}^{\on{naive}}$ follows from \corref{cc:when equivalence}. 

\medskip

Property $\on{Ei}^{\on{enh}}$ follows from Property $\on{Wh}^{\on{ext}}$ and Quasi-Theorem \ref{e:Whit ext and Eis}.  

\medskip

Property $\on{Km}^{\on{prel}}$ follows from Property $\on{Wh}^{\on{ext}}$ and \thmref{c:Op and Whit},
combined with the fact that the essential image of the functor $\on{Loc}_G$ belongs to
$\Dmod(\Bun_G)_{\on{temp}}$, using the intrinsic characterization of the latter
given in \secref{sss:derived Satake temp} (see Remark \ref{r:localization tempered}).

\sssec{Interdependence of the conjectures}

Recall, however, that the above proof of \conjref{c:GL} was conditional on the validity of
Conjectures \ref{c:Whit ext ff} and \ref{c:Op gen}.

\medskip

Let us now assume \conjref{c:GL} and comment on the above supporting conjectures.

\medskip

First, we note that \conjref{c:Op gen} follows formally from 
\thmref{t:generate BunG} modulo \conjref{c:GL}.

\medskip

Second, we note that \conjref{c:Whit ext ff} is equivalent to \thmref{t:coLoc ext ff} modulo \conjref{c:GL}. 

\medskip

I.e., we obtain that Conjectures \ref{c:Whit ext ff} and \ref{c:Op gen} are forced by \conjref{c:GL}. 

\sssec{Implications for the cuspidal category}

Let
$$\IndCoh_{\on{Nilp}^{\on{glob}}_\cG}(\LocSys_\cG)_{\on{cusp}}\subset \IndCoh_{\on{Nilp}^{\on{glob}}_\cG}(\LocSys_\cG)$$
be the full subcategory equal to the right orthogonal of the essential images of the functors
$$\Eis_{\cP,\on{spec}}:\IndCoh_{\on{Nilp}^{\on{glob}}_\cM}(\LocSys_\cM)\to
\IndCoh_{\on{Nilp}^{\on{glob}}_\cG}(\LocSys_\cG)$$
for proper parabolics $P\subset G$.

\medskip

The following results, e.g., from \propref{p:Eis generation}:

\begin{cor}
The subcategory $\IndCoh_{\on{Nilp}^{\on{glob}}_\cG}(\LocSys_\cG)_{\on{cusp}}$ equals
the image of
$$\QCoh(\LocSys_\cG^{\on{irred}})\overset{\jmath_*}\longrightarrow \QCoh(\LocSys_\cG)\overset{\Xi_\cG}\to
\IndCoh_{\on{Nilp}^{\on{glob}}_\cG}(\LocSys_\cG)_{\on{cusp}}.$$
\end{cor}

Hence, assuming \conjref{c:GL}, we obtain:

\begin{corconj}  \hfill

\smallskip

\noindent{\em(a)}
We have an inclusion $\Dmod(\Bun_G)_{\on{cusp}}\subset \Dmod(\Bun_G)_{\on{temp}}$. 

\smallskip

\noindent{\em(b)} We have:
$$\Dmod(\Bun_G)_{\on{cusp}}=\QCoh(\LocSys_\cG^{\on{irred}})\underset{ \QCoh(\LocSys_\cG)}\otimes
\Dmod(\Bun_G)$$
as subcategories of $\Dmod(\Bun_G)$.

\end{corconj}

\sssec{Generation by Kac-Moody representations}

Finally, we note that if we accept \conjref{c:Op quot}, by combining with \conjref{c:GL}, we obtain:

\begin{corconj}
The functor
$$\on{Loc}_G:\on{KL}(G,\kappa)_{\Ran(X)}\to \Dmod(\Bun_G)_{\on{temp}}$$
is a \emph{localization}, i.e., identifies the homotopy category of the target with a Verdier quotient
of the source.
\end{corconj}


\begin{thebibliography}{99}





\bibitem[AG]{AG} D.~Arinkin and D.~Gaitsgory, {\it Singular support of coherent sheaves and the geometric
Langlands conjecture}, arXiv: 1201.6343.

\bibitem[ArkhG]{ArkhG} S.~Arkhipov and D.~Gaitsgory, {\it Differential operators on the loop group via chiral
algebras}, Internat. Math. Res. Notices, {\bf 2002-4}, 165--210 (2002).

\bibitem[Bar]{Bar} J.~Barlev, {\it D-modules on spaces of rational maps and on other generic data}, arXiv: 1204.3469.

\bibitem[BB]{BB} A.~Beilinson and J.~Bernstein, {\it Localisation de g-modules}, C.R. Acad. Sci. Paris Ser. I Math. 
{\bf 292}, 15--18 (1981).

\bibitem[BD1]{BD1} A.~Beilinson and V.~Drinfeld, {\it Chiral algebras}, AMS Colloquium Publications {\bf 51}, AMS (2004). 

\bibitem[BD2]{BD2} A.~Beilinson and V.~Drinfeld, {\it Quantization of HitchinÕs integrable system and Hecke eigensheaves}, 
available at  http://math.uchicago.edu/~mitya/langlands.html. 

\bibitem[Bez]{Bez} R.~Bezrukavnikov, {\it On two geometric realizations of the affine Hecke algebra}, arXiv:1209.0403. 

\bibitem[BF]{BF} R.~Bezrukavnikov and M.~Finkelberg, {\it  Equivariant Satake category and Kostant-Whittaker reduction},
Mosc. Math. J. {\bf 8}, 39--72 (2008).

\bibitem[BG]{BG} A.~Braverman, D~Gaitsgory, {\it Deformations of local systems and Eisenstein series},
{\it Geometric and functional analysis}, {\bf 17}, 1788--1850 (2008).



\bibitem[Dr]{Dr} V.~Drinfeld, {\it Two-dimensional l-adic representations of the fundamental group of a curve over a
finite field and automorphic forms on GL(2)}, Amer. Jour. of Math. {\bf 105}, 85--114 (1983)

\bibitem[DrGa1]{DrGa1} V.~Drinfeld and D.~Gaitsgory,  {\it On some finiteness questions for algebraic stacks},
GAFA {\bf 23} (2013), 149--294.

\bibitem[DrGa2]{DrGa2} V.~Drinfeld and D.~Gaitsgory, {\it Compact generation of the category
of D-modules on the stack of $G$-bundles on a curve}, arXiv: 1112.2402.

\bibitem[DrGa3]{DrGa3} V.~Drinfeld and D.~Gaitsgory, {\it Geometric constant term functor(s)}, 
arXiv: 1311.2071. 



\bibitem[Fr]{Fr}  E.~Frenkel, {\it Affine Algebras, Langlands Duality and Bethe Ansatz}, 
in Proceedings of the International Congress of Mathematical Physics, Paris, 1994, ed. D.~Iagolnitzer, 606--642,
International Press (1995); arXiv:q-alg/9506003. 

\bibitem[FG1]{FG1} E.~Frenkel and D.~Gaitsgory, {\it Local geometric Langlands correspondence: the spherical case}, 
in Algebraic analysis and around: In honor of Professor Masaki Kashiwara's 60th Birthday, 
Advanced Studies in Pure Mathematics {\bf 54}, 167--186 (2009); arXiv:0711.1132.

\bibitem[FG2]{FG2} E.~Frenkel and D.~Gaitsgory, {\it D-modules on the affine flag variety and representations of affine Kac-Moody algebras},
J. Represent. Theory {\bf 13}, 470--608 (2009).

\bibitem[FG3]{FG3} E.~Frenkel and D.~Gaitsgory, {\it Local geometric Langlands correspondence and affine 
Kac-Moody algebras}, in Algebraic Geometry and Number Theory, Progr. Math. {\bf 253}, 69--260 (2006). 

\bibitem[FGKV]{FGKV} E.~Frenkel, D.~Gaitsgory, D.~Kazhdan and K.~Vilonen, {\it Geometric realization of
Whittaker functions and the Langlands conjecture}, J. Amer. Math. Soc. {\bf 11}, 151--484 (1998). 

\bibitem[FGV1]{FGV1} E.~Frenkel, D.~Gaitsgory and K.~Vilonen, {\it Whittaker patterns in the geometry of moduli spaces of
bundles on curves}, Annals of Math. {\bf 153}, 699--748 (2001).

\bibitem[FGV2]{FGV2} E.~Frenkel, D.~Gaitsgory and K.~Vilonen {\it On the geometric Langlands conjecture},
J. Amer. Math. Soc. {\bf 15}, 367--417 (2002).

\bibitem[Ga1]{Ga1} D.~Gaitsgory, {\it On a vanishing conjecture appearing in the geometric Langlands correspondence}, Ann. of
Math. (2) {\bf 160}, no. 2, 617--682 (2004).

\bibitem[Ga2]{Ga2} D.~Gaitsgory, {\it Contractibility of the space of rational maps}, Invent. Math. {\bf 191} (2013), 91--196.

\bibitem[Ga3]{Ga3} D.~Gaitsgory, {\it Ind-coherent sheaves}, Moscow Mathematical J. {\bf 13} (2013), 399--528. 

\bibitem[GR]{GR} D.~Gaitsgory and N.~Rozenblyum, {\it Crystals and D-modules}, arXiv: 1111.2087.















\bibitem[GL:DG]{DG}
Notes on Geometric Langlands, {\it DG categories}, \newline
available at http://www.math.harvard.edu/~gaitsgde/GL/.

\bibitem[GL:Stacks]{Stacks}
Notes on Geometric Langlands, {\it Stacks}, \newline
available at http://www.math.harvard.edu/~gaitsgde/GL/.

\bibitem[GL:QCoh]{QCoh}
Notes on Geometric Langlands, {\it Quasi-coherent sheaves on stacks}, \newline
available at http://www.math.harvard.edu/~gaitsgde/GL/.

\bibitem[GL:GenVan]{GenVan} 
Notes on Geometric Langlands, {\it A Generalized vanishing conjecture}, \newline
available at http://www.math.harvard.edu/~gaitsgde/GL/.

\bibitem[GL:ExtWhit]{ExtWhit} 
Notes on Geometric Langlands, {\it The extended Whittaker category}, \newline
available at http://www.math.harvard.edu/~gaitsgde/GL/.

\bibitem[Lau1]{Lau1} G.~Laumon, {\it Correspondence de Langlands g\'eom\'etrique pour les corps de fonctions}, 
Duke Math. Jour. {\bf 54}, 309--359 (1987)

\bibitem[Lau2]{Lau2} G.~Laumon, {\it Faisceaux automorphes pour GLn: la premi\`ere construction de Drinfeld}, 
arXiv:alg-geom/9511004.

\bibitem[Lu]{Lu} J.~Lurie, {\it Higher Topos Theoery}, Princeton Univ. Press (2009). 

\bibitem[MV]{MV} I.~Mirkovic and K.~Vilonen, 
{\it Geometric Langlands duality and representations of algebraic groups over commutative rings}, 
Ann. of Math. {\bf166}, 95--143 (2007).

\bibitem[Ne]{Ne} A.~Neeman, {\it The Grothendieck duality theorem via Bousfield's techniques and Brown representability}, 
J. Amer. Math. Soc. {\bf 9}, 205--236 (1996).

\bibitem[Ro]{Ro} N.~Rozenblyum, {\it Connections on Conformal Blocks}, PhD Thesis, MIT (2011). 

\bibitem[Sto]{Sto} A.~Stoyanovsky, {\it Quantum Langlands duality and conformal field theory}, arXiv:math/0610974


\end{thebibliography}
\end{document}